\documentclass[12pt]{amsart}
\usepackage{epsfig,color, amsmath}
\usepackage[normalem]{ulem} 
\usepackage{esint}
\usepackage{amssymb}
\usepackage{hyperref, mathrsfs}
\usepackage[dvipsnames]{xcolor}
\usepackage{bm}
\usepackage{enumitem}
\usepackage{mathtools, comment}
\usepackage{scalerel}
\usepackage{accents}

  \hypersetup{colorlinks}
\hypersetup{citecolor=blue}
\hypersetup{urlcolor=blue}

\makeatother

\theoremstyle{definition}

\def\fnum{equation} 
\newtheorem{Thm}[\fnum]{Theorem}
\newtheorem{Cor}[\fnum]{Corollary}

\newtheorem{Lem}[\fnum]{Lemma}

\newtheorem{Def}[\fnum]{Definition}

\newtheorem{Rem}[\fnum]{Remark}
\newtheorem{Pro}[\fnum]{Proposition}

\newtheorem{Cla}[\fnum]{Claim}

\numberwithin{equation}{section}
\renewcommand{\text}[1]{\textrm{\upshape #1}}

\newcommand{\HH}{{\mathcal{H}}}

\newcommand{\LL}{{\mathcal{L}}}

\newcommand{\Lip}{{\text {Lip}}}
\newcommand{\diam}{{\text {diam}}}

\newcommand{\supp}{{\text{supp}}}

 \newcommand{\N}{\ensuremath{\mathbb{N}}}
 \newcommand{\Q}{\ensuremath{\mathbb{Q}}}
  \newcommand{\R}{\ensuremath{\mathbb{R}}}

 \newcommand{\ba}{\begin{align*}}
 \newcommand{\ea}{\end{align*}}

\newcommand{\pr}{\text{Pr}}

\newcommand\Tau{\scalerel*{\tau}{T}}

 \newcommand{\norm}[2]{{ \ensuremath{\left\|} #1 \ensuremath{\right\|}}_{#2}}

\def\RR{{\mathbb R}}

\newcommand{\dd}{{\mathsf {d}}}

\newcommand{\e}{{\text {e}}}

\newcommand{\veps}{\varepsilon}
\newcommand{\rcd}{\text{RCD}}

\newcommand{\cde}{\text{CD}^e(K,N)}
\newcommand{\cden}{\text{CD}^e(K,n)}

\newcommand{\Ent}{\text{Ent}}
\newcommand{\Adm}{\text{Adm}}
\newcommand{\Opt}{\text{Opt}}

\DeclareMathOperator{\RCD}{RCD}
\DeclareMathOperator{\CD}{CD}
\DeclareMathOperator{\MCP}{MCP}

\DeclareMathOperator{\OptGeo}{OptGeo}

\DeclareMathOperator{\Geo}{Geo}
\DeclareMathOperator{\curv}{curv}

\newcommand{\mm}{\mathfrak{m}}

\newcommand{\mrestr}{\mathbin{\vrule height 1.6ex depth 0pt width
0.13ex\vrule height 0.13ex depth 0pt width 1.3ex}}

\title{Topology of non-collapsed three-dimensional RCD spaces}
\author{Qin Deng \and Alessandro Pigati}
\address{\parbox{\linewidth}{Bocconi University, Department of Decision Sciences.\\
	Via Roentgen 1,
	20136 Milano -- Italy\\[-4pt]\phantom{a}}}
\email{qin.deng@unibocconi.it}
\address{\parbox{\linewidth}{Bocconi University, Department of Decision Sciences.\\
	Via Roentgen 1,
	20136 Milano -- Italy\\[-4pt]\phantom{a}}}
\email{alessandro.pigati@unibocconi.it}

\begin{document}

\raggedbottom

\begin{abstract}
    We show that non-collapsed $\text{RCD}(K,3)$ spaces without boundary are orbifolds whose topological singularities are locally finite and locally homeomorphic to cones over $\mathbb{RP}^2$, and that the topology of such spaces is stable under non-collapsed Gromov--Hausdorff convergence. We study the notion of non-orientability on these spaces as a key part of our analysis and show that the property of non-orientability (on uniformly sized balls) is stable under non-collapsed Gromov--Hausdorff convergence. Finally, we show that any non-orientable non-collapsed $\text{RCD}(K,3)$ space without boundary admits a ramified double cover which is itself an orientable non-collapsed $\text{RCD}(K,3)$ space without boundary, and that such ramified double cover is stable under non-collapsed Gromov--Hausdorff convergence.
\end{abstract}

\maketitle
\tableofcontents

\section{Introduction}

In this paper, we study the topology of non-collapsed $\RCD(K,3)$ spaces without boundary. Our first main result concerns the topological regularity of these spaces.
\begin{Thm}\label{Thm: orbifold structure thm intro} (Orbifold structure theorem).
    Let $(X, \dd, \HH^3)$ be a non-collapsed $\RCD(K,3)$ space without boundary. Then $X$ is an orbifold and, denoting by $\mathcal{P}$ the set of all points that admit a tangent cone with cross-section homeomorphic to $\mathbb{RP}^2$, the following hold:
    \begin{enumerate}
        \item $\mathcal{P}$ is locally finite;
        \item $X\setminus\mathcal{P}$ is a topological manifold;
        \item for each $x\in\mathcal{P}$, there exists a neighborhood $U$ of $x$ homeomorphic to $C(\mathbb{RP}^2)$.
    \end{enumerate}
\end{Thm}
This addresses a conjecture of Mondino on non-collapsed $\RCD(K,3)$ spaces when the space is without boundary \cite{BPS24} and partially generalizes to the $\RCD$ setting the topological regularity of $3$-dimensional Alexandrov spaces from Perelman's conical neighborhood theorem \cite{P93}. Our result is dimensionally sharp due to \cite{M00}, which gives a $4$-dimensional non-collapsed Ricci limit space with a point whose tangent cone is not homeomorphic to any neighborhood of the point. We point out that, by the manifold recognition theorem of \cite{BPS24} (cf.\ Theorem~\ref{loc.reg}), any non-collapsed $\RCD(K,3)$ space without boundary that does not admit a tangent cone (at any point) homeomorphic to $C(\mathbb{RP}^2)$ is a topological manifold. As such, our work deals with the presence of singularities of the form $C(\mathbb{RP}^2)$. In the case of $3$-dimensional non-collapsed Ricci limit spaces, topological regularity is completely understood. Indeed, it was shown that all such spaces are homology manifolds in \cite{Z93} and, more recently, bi-H\"older homeomorphic to smooth Riemannian manifolds in \cite{S12, ST21, ST22}. 

A corollary of Theorem~\ref{Thm: orbifold structure thm intro} is the following topological classification theorem for positively curved $\RCD$ spaces. 
\begin{Cor}\label{Cor: topological classification intro} (Topological classification theorem).
 Let $(X, \dd, \HH^3)$ be a non-collapsed $\rcd(K,3)$ space without boundary with $K > 0$. Then one of the following holds: 
 \begin{enumerate}
     \item $X$ is a spherical $3$-manifold, i.e., $X$ is an orientable topological manifold homeomorphic to $\mathbb{S}^3/\Gamma$, where $\Gamma < SO(4)$ is a finite subgroup acting freely by rotations;
     \item $X$ is homeomorphic to the spherical suspension over $\mathbb{RP}^2$.
 \end{enumerate}
\end{Cor}

Our second main result is the following topological stability theorem. 
\begin{Thm}\label{Thm: stability thm intro} (Topological stability theorem).
Given $0<v,D<\infty$, there exists $\veps(v,D)>0$ such that the following holds:
Assume that $(X,\dd_X,\HH^3)$ and $(Y,\dd_Y,\HH^3)$
are two non-collapsed $\RCD(-2,3)$ spaces without boundary, with diameter $\le D$
and $\HH^3(B_1(p))\ge v$, for all $p\in X$ and all $p\in Y$.
If $d_{GH}(X,Y)<\veps$ then $X$ and $Y$ are homeomorphic,
with a homeomorphism which can be taken $\delta$-close to the $\veps$-GH isometry and its inverse (with $\delta(\veps)\to0$ as $\veps\to0$).
\end{Thm}
This partially generalizes the topological stability theorem of Perelman \cite{P91} in the context of Alexandrov spaces and the topological stability theorem of \cite{BPS24} for non-collapsed $\RCD(K,3)$-spaces which are topological manifolds. As with the topological regularity theorem, this result is dimensionally sharp due to the examples constructed in \cite{A90}. 

As a byproduct of our techniques, we also obtain uniform contractibility of non-collapsed $\rcd(-2,3)$ spaces without boundary (see Corollary~\ref{cor.contr.bis}).

\begin{Pro} (Uniform contractibility)
    If $(X,\dd,\HH^3,p)$ is a non-collapsed $\rcd(-2,3)$ space without boundary
    with $\HH^3(B_1(p))\ge v>0$, then any ball $B_r(q)\subseteq B_1(p)$
    is contractible in $B_{C(v)r}(q)$, provided that $r\in(0,r_0(v))$ is small enough.
\end{Pro}

To prove Theorems~\ref{Thm: orbifold structure thm intro} and \ref{Thm: stability thm intro}, we have to understand the topology of $\RCD$ spaces with $C(\mathbb{RP}^2)$ singularities. As such, it is natural to study non-orientable $\RCD$ spaces as defined in \cite{BBP24} (see also \cite{H17} for an equivalent formulation in the setting of Ricci limit spaces). Our secondary result is the following stability theorem that answers a question of \cite{BBP24} in the $3$-dimensional case.

\begin{Thm}\label{Thm: non-orientability stability intro} (Stability of non-orientability).
Let $(X_i, \dd_i, \HH^3, p_i)_{i\in \N}$ be a sequence of non-orientable $\RCD(K,3)$ spaces without boundary converging in the pmGH sense to some $\RCD(K,3)$ space without boundary $(X, \dd, \HH^3, p)$. If, for some $R > 0$, $B_{R}(p_i)$ is non-orientable for all $i \in \N$, then $B_{R'}(p)$ is non-orientable for all $R' > R$.
\end{Thm}

The notion of orientable ramified double cover was introduced for non-orientable $\RCD$ spaces in \cite{BBP24} as a replacement for the orientable double cover of a smooth manifold in the singular setting. We prove the following regularity theorem for the ramified double cover. 

\begin{Thm}\label{Thm: ramified double cover is RCD intro} (Regularity of ramified double cover).
Let $(X, \dd, \HH^3)$ be a non-orientable $\RCD(K,3)$ space without boundary. Then the ramified double cover $(\hat{X}, \hat\dd, \HH^3)$ is an orientable $\RCD(K,3)$ space without boundary. 
\end{Thm}

The following stability theorem for the ramified double cover follows directly from Theorem~\ref{Thm: ramified double cover is RCD intro} and \cite[Theorem 4.2]{BBP24}. 

\begin{Thm}\label{Thm: ramified double cover stability intro} (Stability of ramified double cover).
Under the assumptions of Theorem~\ref{Thm: non-orientability stability intro}, denote by $(\hat{X}_i, \hat\dd_i, \HH^3, \hat{p}_i)$ the ramified double covers of $X_i$. Then $(\hat{X}_i, \hat\dd_i, \HH^3, \hat{p}_i)$ converges in the pmGH sense to $(\hat{X}, \hat\dd, \HH^3, \hat{p})$, where the latter is the ramified double cover of $X$ and $\pi(\hat{p}) = p$ (with $\pi: \hat{X} \to X$ being the associated projection map).
\end{Thm}

\subsection{Outline of the proof}
Here we briefly outline our strategy for the proof of Theorems~\ref{Thm: orbifold structure thm intro} and \ref{Thm: stability thm intro}, leaving out a discussion of the other results.

Let $(X, \dd, \HH^3)$ be a non-orientable (see Definition~\ref{Def: Orientability}) $\RCD(-2,3)$ space without boundary, and let $\mathcal{P}$ be the set of points whose tangent cones all have cross-sections that are homeomorphic to $\mathbb{RP}^2$ (we note that the tangent cones at any particular point are either all homeomorphic to $\RR^3$ or all homeomorphic to $C(\mathbb{RP}^2)$ by Remark~\ref{Rem: one tangent cone all tangent cone}). If $\mathcal{P} = \emptyset$, then $X$ is a topological manifold due to the manifold recognition theorem of \cite{BPS24}. Suppose now that $\mathcal{P} \neq \emptyset$. Let $(\hat{X}, \hat\dd, \HH^3)$ be the ramified double cover (see Theorem~\ref{Thm: Ramified double cover existence} for the definition) with associated projection map $\pi: \hat{X} \to X$ and isometric involution $\Gamma: \hat{X} \to \hat{X}$. In the intuitive picture of what the ramified double cover should look like, one might expect that it has the following nice properties:
\begin{enumerate}
    \item $\hat{X}$ is a non-collapsed $\RCD(-2,3)$ space without boundary;
    \item $\pi: \hat{X} \to X$ is a double cover and local isometry away from $\pi^{-1}(\mathcal{P})$. In particular, given $\hat{p} \in \pi^{-1}(X\setminus\mathcal{P})$, the tangent cones over $\pi(\hat p)$ are exactly the tangent cones over $\hat{p}$ and hence they are homeomorphic to $\R^3$, as the latter are orientable;
    \item each $p \in \mathcal{P}$ has a unique lift $\hat{p} \in \pi^{-1}(p)$  and the tangent cones over $\hat{p}$ are the ramified double covers of the tangent cones over $p$, and hence homeomorphic to $\R^3$.
\end{enumerate}
As an instructive example, we can consider 
the ramified double cover of the spherical suspension $S(Y^2)$, where $Y^2 \cong \mathbb{RP}^2$ is a $2$-dimensional Alexandrov space with $\curv \geq 1$. In this case, $\mathcal{P}$ consists exactly of the two tips of the suspension and the ramified double cover is the suspension over the double cover of $Y^2$, so it is easy to see that all the properties listed above hold. Whenever these hold, then one can apply the manifold recognition theorem of \cite{BPS24} to conclude that $\hat{X}$ is a topological manifold. The orbifold structure and topological stability claimed in Theorems~\ref{Thm: orbifold structure thm intro} and \ref{Thm: stability thm intro} can then be proved by keeping track of the isometric involution on $\hat{X}$. We remark that this last step requires several rather involved arguments using the regularity properties of good Green balls (defined in \cite{BPS24}) as well as tools from geometric topology, which we will not go over in this outline due to their technical nature. 

We now discuss how we will prove the nice properties of the ramified double cover listed above. The problem can be divided into two parts: 
\begin{enumerate}
    \item proving the properties assuming that $\mathcal{P}$ is locally finite;
    \item proving that $\mathcal{P}$ is locally finite.
\end{enumerate}

    To illustrate our ideas for the first problem, let us only consider the simpler case where $\mathcal{P} = \{p\}$ is a singleton (the general case is not so different). Denote by $\mathcal{P}' \subseteq X$ the set of all points that have a unique lift via $\pi$. It can be proved that $\pi$ is a local isometry and a double cover away from $\mathcal{P}'$, and so, in particular, $\pi^{-1}(X \setminus\mathcal{P}')$ looks $\RCD(-2,3)$ locally, in the sense that every point has a neighborhood that is metric measure isomorphic to some open subset of an $\RCD(-2,3)$ space. One would then hope to show that each point in $\pi^{-1}(\mathcal{P}')$ also has a neighborhood that is locally strongly $\CD^e(-2,3)$ (see Definition~\ref{Def: local CDe definition}) and then globalize the $\CD^e(-2,3)$ condition. As such, it would be ideal if $\pi^{-1}(\mathcal{P}')$ had no accumulation points, since then we could deal with neighborhoods which have only a single element of $\pi^{-1}(\mathcal{P}')$.

    It can be checked that $\mathcal{P}'$ is exactly the set of all locally non-orientable points, i.e., points that are not contained in any orientable open set (see Definition~\ref{Def: lno points}). Due to the stability of orientability proved in \cite{BBP24}, it is known that any sequence of orientable balls must converge to an orientable ball under pmGH convergence. Since $C(\mathbb{RP}^2)$ is non-orientable, it follows that $\mathcal{P} \subseteq \mathcal{P}'$. Conjecturally, the two sets should be equal. This would follow from the stability of non-orientability, which is not known in general. For example, it is not ruled out that one can construct a sequence of non-orientable $\RCD(-(n-1),n)$ spaces converging to a cone over a small $(n-1)$-sphere, so that the ``non-orientable topology" is concentrated closer and closer to the tip, i.e., all orientation-reversing loops (see Definition~\ref{Def: orientation-reversing loop}) must pass through a smaller and smaller neighborhood of the cone tip. In this case, the ramified double covers of the sequence would converge to two copies of the cone over the sphere glued at the tip. We point out that this space clearly does not have $\CD$ structure (for instance due to branching around the tip), which is something we will leverage later. 
    
    Nevertheless, in the case where $\mathcal{P}$ is a singleton, we can show $\mathcal{P} = \mathcal{P}'$. Indeed, by localizing \cite{BPS24}, we can prove a local manifold recognition theorem (Theorem~\ref{loc.reg}), which says that an open subset of a non-collapsed $\RCD(-2,3)$ space for which all tangent cones (at all points) are homeomorphic to $\RR^3$ must be a topological manifold. As such, $X \setminus \mathcal{P}$ is a topological manifold and every point in $X \setminus \mathcal{P}$ is locally orientable. This implies $\mathcal{P}' = \mathcal{P}$ and so, in particular, $\mathcal{P}'$ is also a singleton. 

    Having established that $\mathcal{P'} = \mathcal{P} = \{p\}$, let us denote by $\hat{p}$ the unique lift of $p$ via $\pi$. As mentioned earlier, we would like to prove that there is a neighborhood of $\hat{p}$ in $\hat{X}$ that satisfies the local strong $\CD^e(-2,3)$ condition. 
    
    As an aside, this is a special case of the problem of lifting curvature conditions from some quotient space $X/G$ to the covering space $X$ for some discrete group $G$ acting on $X$ which, to the best of our knowledge, is not well-understood. Some positive results are known for $3$-dimensional Alexandrov spaces under the assumption that the ($1$-dimensional) fixed point set is extremal \cite{GW14}; their proof can likely be used to obtain more general results in the Alexandrov setting (see for instance \cite[Proposition 2.4]{DGGM18}).  We mention that even more generally this is a special case of the problem of globalizing the $\CD$ or $\RCD$ condition from local ones away from some low-dimensional set, which was investigated in \cite{HS25} (although their results do not seem to apply in our case). As suggested by the extremality assumption in \cite{GW14}, it is likely that one needs to assume some geometric condition on this low-dimensional set in order to achieve a positive result; we point out that the closure of the set of points whose tangent cone is homeomorphic to some given space (e.g., $\overline{\mathcal{P}}$) is automatically extremal in the Alexandrov setting \cite{P07}.
    
    Coming back to the work at hand, a key lemma is that no optimal dynamical plan between two measures of bounded density in $\hat{X}$ can be concentrated (or have positive measure) on the set of geodesics that pass through $\hat{p}$. The proof uses a blow-up argument. Assuming that such an optimal dynamical plan $\nu$ does exist, one can then construct a sequence of optimal dynamical plans $\nu_k$ so that:
    \begin{enumerate}
        \item $\nu_k$ is supported on the set of geodesics whose image is contained in $B_{R/k}(\hat{p})$ for some fixed $R$;
        \item $\nu_k$ is the optimal dynamical plan between two absolutely continuous measures whose densities are bounded by $ck^{-3}$, where $c$ is independent of $k$. 
    \end{enumerate}
    To construct $\nu_k$, we first set $\nu_k'$ to be the normalization of $\nu$ restricted to the set of geodesics which pass through $\hat{p}$ in a suitably chosen time interval $I_k := [t_k, t_k+1/k] \subseteq [0,1]$. Next we define $F_k: \Geo(\hat{X}) \to \Geo(\hat{X})$ (see \eqref{Eq: Geo(X)} for notation) to be the map which takes a geodesic $\gamma$ to $\gamma \lvert_{[t_k - 1/k, t_k+2/k]}$ (linearly reparameterized on $[0,1]$) and set $\nu_k := (F_k)_*(\nu_k')$. Clearly $\nu_k$ is an optimal dynamical plan between $(e_0)_*(\nu_k)$ and $(e_1)_*(\nu_k)$, where $e_t: \Geo(\hat{X}) \to \hat{X}$ denotes the evaluation map at $t$. It is not difficult to see that if the density of $(e_0)_*(\nu_k)$ is bounded by $C$, then the density of $(e_0)_*(\nu_k')$ can be bounded by $Ck^{-1}$ and, using the local $\CD^{e}(-2,3)$ condition away from $\hat{p}$, the density of $(e_0)_*(\nu_k)$ can be bounded by $c_0Ck^{-4}$, where $c_0$ does not depend on $k$. This is not quite good enough for our purposes but it turns out that we can save a factor of $k$ by an argument using $1$D-localization with respect to the distance function to $p$ on $X$ and the monotonicity of optimal transport for quadratic cost on $\R$. Having constructed such a sequence of $\nu_k$, we can pushforward the plan onto $X$ itself (by mapping curves on $\hat{X}$ to their projections) and take a limit under blow-up to obtain a contradiction, using the geometry of the tangent cones at $p$. 

    Once the above lemma is proved for $\hat{p}$, it is reasonable to expect that the presence of $\hat{p}$ should not affect the concavity of entropy along Wasserstein geodesics between absolutely continuous measures. In particular, since $\hat{X}$ is locally strongly $\CD^{e}(-2,3)$ away from $\hat{p}$, the usual proof of globalization of the $\CD^e(-2,3)$ condition goes through and prove that $\hat{X}$ is globally strongly $\CD^e(-2,3)$. We note that this argument would also give the local strong $\CD^e(-2,3)$ condition for a neighborhood of $\hat{p}$, if instead we only assumed that there is a neighborhood $U$ of $p$ such that $\mathcal{P} \cap U$ is a singleton.

Now we discuss our strategy for proving that $\mathcal{P}$ is locally finite. This will be done by an induction argument on the density $\theta$. Fix any bounded open $B \subseteq X$ (so that the density is bounded away from $0$ for points in $B$) and some suitable small $\delta > 0$. We define the open sets $U_k := \{x \in B \,:\, \theta(x) > 4\pi - k\delta\}$ and show by induction on $k$ that $\mathcal{P}$ cannot accumulate in $U_k$. This clearly holds for $U_0$ since $U_0 = \emptyset$. Suppose that we have shown that there are no accumulation points of $\mathcal{P}$ in $U_k$; we would like to prove that there are also no accumulation points of $\mathcal{P}$ in $U_{k+1}$. Assume for the sake of contradiction that there is some $p \in U_{k+1}$ that is an accumulation point of $\mathcal{P}$. Trivially, we have either $p \in \mathcal{P}$ or $p \notin \mathcal{P}$. We focus on the first case in the outline; the second one is more difficult but uses some similar ideas.

Assuming $p \in \mathcal{P}$, there exists some $s > 0$ such that, for all $s' < s$, $B_{s'}(p)$ (rescaled to a ball of radius $1$) is close in Gromov--Hausdorff distance to the ball of radius $1$ in some $\RCD(0,3)$ cone with cross-section $\cong \mathbb{RP}^2$. Therefore, there must be a definite increase in density for the points close to $p$ compared to the density of $p$ itself. This means that, by choosing $\delta$ sufficiently small, we can find some small $r > 0$ so that $B_{r}(p) \setminus \{p\} \subseteq U_k$. By the induction hypothesis, this implies that no point $q \in B_{r}(p) \setminus \{p\}$ is an accumulation point of $\mathcal{P}$. It follows from what we discussed in the first part that there must then be a small neighborhood around any lift of $q$ in $\hat{X}$ that is locally $\CD^e(-2,3)$. Zooming in around $p$, we see that there is a sequence of spaces $(X_i, p_i)$ (with $p_i$ being the point corresponding to $p$ in the blow-up) converging to an $\RCD(0,3)$ cone with cross-section $\cong \mathbb{RP}^2$ satisfying the following properties.
\begin{enumerate}
    \item For every $q \in B_{2}(p_i)\setminus \{p_i\} \subseteq X_i$, there is a neighborhood around $q$ that is locally strongly $\CD^e(-2,3)$. 
    \item For every $i$, there is a point $q_i \in \mathcal{P}$ that is distance $1$ away from $p_i$.
\end{enumerate}
Globalizing the first property away from $p_i$, we see that each $\hat{X}_i$ is locally $\CD^e(-2,3)$ on a uniformly sized ball around $q_i$. This is enough to show (recall the example of the two cones glued at the tip mentioned earlier) that the local non-orientability of $q_i$ is stable, i.e., $q_i$ must converge, up to a subsequence, to some $q$ in the $\RCD(0,3)$ cone that is also locally non-orientable. However, the only locally non-orientable point of the cone is the cone tip itself and $q$ is distance $1$ away from the tip by the second property, which gives a contradiction.  

\section*{Acknowledgment}
We thank Elia Bruè for many insightful comments throughout the writing of this paper. We also thank Camillo Brena, Nicola Gigli, Vitali Kapovitch, and Daniele Semola for helpful discussions.
A.P.'s research is funded by the European Research Council (ERC) through StG 101165368 ``MAGNETIC.''
Views and opinions expressed are however those of the authors only and do not necessarily reflect those of the European Union or the European Research Council.

\section{Preliminaries}\label{sec: prelim}
In this section, we will collect some preliminary results on $\RCD(K,N)$ spaces. As the theory is at this stage quite well-developed, we will only present background relevant to our arguments and will not attempt to be comprehensive, referring to the survey papers \cite{A18, G23, S23} and the references therein for an overview of the subject. In the development of the theory, various results were proven by assuming the $\RCD^{*}(K,N)$ condition. This has since been proven to be equivalent to the $\RCD(K,N)$ condition in \cite{CM21} (see also \cite{L24}) and hence we will simply take our assumption to be the $\RCD(K,N)$ condition when citing various results, even if they were originally proved for $\RCD^*(K,N)$ spaces. 
In the present work, we will mostly deal with non-collapsed $\RCD(K,N)$ spaces 
without boundary (see \cite{DPG18} for the definition of non-collapsed $\RCD(K,N)$ spaces and of the boundary); sometimes we will omit the non-collapsed qualifier or the assumption of empty boundary for brevity, when they are obvious from the context. 

In this paper, a \textit{metric space} $(X, \dd)$ is always assumed to be proper (i.e.,  closed and bounded sets are compact) unless stated otherwise. Note that proper metric spaces are Polish (i.e., complete and separable). We say that a metric space $(X, \dd)$ is \textit{geodesic} if for any $x, y \in X$ there exists $\gamma \in \Geo(X)$ such that $\gamma(0) = x$ and $\gamma(1) = y$, with
\begin{equation}\label{Eq: Geo(X)}
    \Geo(X) := \Big\{\gamma \in C([0,1], X) \,:\, \dd(\gamma(s), \gamma(t)) = |s-t|\dd(\gamma(0), \gamma(1)), \text{ for $s, t \in [0,1]$} \Big\}. 
\end{equation}
We say that a curve is \textit{geodesic} if and only if it is an element of $\Geo(X)$. In particular, geodesics in this paper are constant speed length-minimizing curves parameterized on $[0,1]$. It can be checked that, on complete geodesic spaces, local compactness is equivalent to properness. We will often also assume that $(X, \dd)$ is geodesic. 

A \textit{metric measure space} is always taken to mean a triple $(X,\dd,\mm)$ where $(X,\dd)$ is a metric space and $\mm$ is a nonnegative, nonzero Borel measure on $X$ that is finite on bounded sets with $\supp(\mm)=X$, unless stated otherwise. 

It is known that finite Borel measures on a Polish space are inner regular with respect to compact sets and outer regular with respect to open sets (see for instance \cite[Theorem 7.1.7]{B07}). It is not difficult to check from this that locally finite (i.e., finite on bounded sets) Borel measures on Polish spaces satisfy the same inner and outer regularity.

\subsection{Pointed measured Gromov--Hausdorff convergence}
In this subsection, we review the various notions of Gromov--Hausdorff convergence. We will consider only metric spaces which are proper, complete and separable, and Borel measures on such spaces which are nonnegative, nonzero, and Radon. We do not necessarily assume that the support of the measure is $X$ in this particular subsection, as the class of such pointed metric measure spaces is not closed in the topology induced by pointed measured Gromov--Hausdorff convergence (see \cite[Remark 3.25]{GMS15}).   

We take the following definition from \cite[Definition 2.1]{DPG18}. There are several other definitions which are known to be equivalent in the case of geodesic metrics and uniformly locally doubling measures; we refer to the paragraph prior to \cite[Definition 2.1]{DPG18} for a more detailed discussion. 
\begin{Def}\label{Def: pmGH convergence}(Gromov--Hausdorff convergence).
Let $(X_n, \dd_n)$, $n \in \N \cup \{\infty\}$ be metric spaces. We say that $(X_n, \dd_n)$ converges to $(X_\infty, \dd_{\infty})$ in the \textit{Gromov--Hausdorff} (GH for short) sense provided that there exist a metric space $(Z, \dd_Z)$ and isometric embeddings $\iota_n: X_n \to Y$ for each $n \in \N \cup \{\infty\}$ so that 
\begin{equation*}
    \dd_{H}(\iota_n(X_n), \iota_\infty(X_\infty)) \to 0\quad \text{as $n \to \infty$,}
\end{equation*}
where $\dd_H$ denotes the Hausdorff distance in $Z$. 

If the spaces are pointed, i.e., for each $n$ we have some reference point $x_n \in X_n$, then we say that $(X_n, \dd_n, x_n)$ converges to $(X_\infty, \dd_\infty, x_\infty)$ in the \textit{pointed Gromov--Hausdorff} (pGH for short) sense provided that there exist a metric space $(Z, \dd_Z)$ and isometric embeddings $\iota_n: X_n \to Y$ for each $n \in \N \cup \{\infty\}$ so that:
\begin{enumerate}
    \item $\iota_n(x_n) \to \iota_\infty(x_\infty)$ in $Z$;
    \item for every $R'>R > 0$ and $\veps>0$, eventually $\iota_n(B_R(x_n))$ is included in
    the $\veps$-neighborhood of $\iota_\infty(B_{R'}(x_\infty))$, and
    $\iota_\infty(B_R(x_\infty))$ is included in the $\veps$-neighborhood of $\iota_n(B_{R'}(x_n))$.
    For geodesic spaces, this is equivalent to require that, for each $R>0$,
    \begin{equation*}
          \dd_{H}\Big(\iota_n(B_R(x_n)), \iota_\infty(B_R(x_\infty))\Big) \to 0 \quad \text{as $n \to \infty$.}
    \end{equation*}
\end{enumerate}

If moreover the spaces $X_n$ are endowed with Radon measures $\mm_n$ which are finite on bounded sets, we say that $(X_n, \dd_n, \mm_n, x_n)$ converges to $(X_\infty, \dd_\infty, \mm_\infty, x_\infty)$ in the \textit{pointed measured Gromov--Hausdorff} sense (pmGH for short) provided that there exist $(Z, \dd_{Z})$ and $\{\iota_n\}_{n \in \N \cup \{\infty\}}$ satisfying the conditions (1), (2) above and moreover it holds that:
\begin{enumerate}
    \setcounter{enumi}{2}
    \item $(\iota_n)_*(\mm_n)$ converges weakly to $(\iota_\infty)_*(\mm_\infty)$, i.e., for every $\varphi \in C_b(Z)$ with bounded support, where $C_b(Z)$ denotes the set of continuous bounded functions on $Z$, we have
    \begin{equation*}
        \int \varphi \, d(\iota_n)_*(\mm_n) \to \int \varphi \, d(\iota_\infty)_*(\mm_\infty) \quad \text{as $n \to \infty$.}
    \end{equation*}
\end{enumerate}

In any of the above cases, $(Z, \dd_Z, \{\iota_n\})$ is called a \textit{realization of the convergence} and, given $y_n \in X_n$, $n \in \N \cup \{\infty\}$, we say that $y_n$ converges to $y_\infty$, and write $y_n \overset{\text{GH}}{\to} y_\infty$, provided that there exists a realization such that 
\begin{equation*}
    \lim_{n \to \infty} \dd_Z(\iota_n(y_n), \iota_\infty(y_\infty)) =0.
\end{equation*}
\end{Def}

\begin{Rem}\label{Rem: proper realization}
If, for each $n \in \N \cup \{\infty\}$, $(X_n, \dd_n)$ is a proper metric space, i.e., every closed ball is compact, then we can find a realization of the convergence where $(Z, \dd_Z)$ in the realization is proper as well. Indeed, given any realization $(Z, \dd_Z, \{\iota_n\})$ we may take $Z' := \bigcup_{n \in \N \cup \{\infty\}} \iota_n(X_n) \subseteq Z$ and take $\dd_{Z'}$ the restriction of $\dd_{Z}$. It is not difficult to check that $(Z', \dd_{Z'}, \{\iota_n\})$ is also a realization of the convergence and $(Z', \dd_{Z'})$ is a proper metric space. 
\end{Rem}

We will be interested in the convergence of probability measures in the context of a pmGH convergence. For a metric space $X$, we denote by $\mathcal{P}(X)$ the space of Borel probability measures on $X$. Fix a sequence $(X_n, \dd_n, \mm_n, x_n)$ converging to $(X_\infty, \dd_\infty, \mm_\infty, x_\infty)$ in the pmGH sense with $(Z, \dd_Z, \{\iota_n\})$ a realization of the convergence. In addition, we assume that each $(X_n, \dd_n)$ is proper, which implies that we may assume $(Z, \dd_Z)$ is proper thanks to Remark~\ref{Rem: proper realization}.
\begin{Lem}\label{Lem: weak convergence of probability measures}
Let $R > 0$ and suppose that for each $n\in\N$ we have $\mu_n \in \mathcal{P}(X_n)$ such that $\mu_n(B_R(x_n)) = 1$. There exists $\mu_\infty \in \mathcal{P}(X_\infty)$ such that $(\iota_n)_*(\mu_n)$ subconverges weakly to $(\iota_\infty)_*(\mu_\infty)$, i.e., for any $\varphi \in C_b(Z)$, 
\begin{equation*}
    \int \varphi \, d(\iota_n)_*(\mu_n) \to \int \varphi \, d(\iota_\infty)_*(\mu_\infty) \quad \text{as $n \to \infty$.}
\end{equation*}
\end{Lem}
\begin{proof}
By definition of pmGH convergence, for large $n$ we have that $(\iota_n)_*(\mu_n)$ is concentrated in $B_{R+1}(\iota_{\infty}(x_\infty))$, i.e., $((\iota_n)_*(\mu_n))(B_{R+1}(\iota_{\infty}(x_\infty)))=1$. Since $(Z, \dd_Z)$ is proper, by Prokhorov's theorem, up to choosing a subsequence, $(\iota_n)_*(\mu_n)$ weakly converges to some $\mu_\infty' \in \mathcal{P}(Z)$. It is not difficult to check that $\mu_\infty'(\iota_\infty(X_\infty))=1$ by applying the definition of weak convergence to bounded continuous functions supported on open balls disjoint from $\iota_\infty(X_\infty)$. Therefore, $\mu_\infty' = (\iota_\infty)_*(\mu_\infty)$ for some $\mu_\infty \in \mathcal{P}(X_\infty)$ as required. 
\end{proof}

We have the following lemma. 
\begin{Lem}\label{Lem: weak convergence density bound}
Let $c > 0$ and, for each $n \in \N$, let $\mu_n \in \mathcal{P}(X_n)$ be so that $\mu_n \leq c\mm_n$, i.e., for every Borel set $A \subseteq X_n$, $\mu_n(A) \leq c\mm_n(A)$. If  $\mu_\infty \in \mathcal{P}(X_\infty)$ is so that $(\iota_\infty)_*(\mu_\infty)$ is a weak limit of $(\iota_n)_*(\mu_n)$, then $\mu_\infty \leq c\mm_\infty$.
\end{Lem}

\begin{proof}
The lemma follows easily by applying the definition of weak convergence along with Lemma~\ref{Lem: dual density bound}. 
\end{proof}

\begin{Lem}\label{Lem: dual density bound}
Let $(X, \dd)$ be a Polish space and $\mu, \nu$ be nonnegative Borel measures on $X$ which are finite on bounded sets, and let $c > 0$. The following are equivalent:
\begin{enumerate}
    \item $\mu \leq c\nu$;
    \item for every nonnegative $f \in C_b(X)$ with compact support, \begin{equation*}
        \int f \, d\mu \leq c\int f \, d\nu.
    \end{equation*}
\end{enumerate}
\end{Lem}
\begin{proof}
The implication from (1) to (2) is obvious by the definition of integration. For the other direction it suffices to show that, for any bounded Borel $A \subseteq X$, $\mu(A) \leq c\nu(A)$.
Since $\mu,\nu$ are outer regular, we can assume that $A=U$ is open.
Using the distance function to $X \setminus U$, we can construct a monotone increasing sequence of nonnegative functions $f_i \in C_b(X)$ with bounded support converging to $1_U$, the indicator function of $U$. Clearly $\int f_i \, d\nu \to \nu(U)$ and $\int f_i \, d\mu \to \mu(U)$ as $i \to \infty$ and so $\mu(U) \leq c\nu(U)$ by our assumption. 
\end{proof}

\subsection{Optimal transport, non-branching, and curvature-dimension condition} In this subsection we review some basics of optimal transport theory, as well as the essentially non-branching condition and the entropic curvature-dimension condition $\cde$. 

Let $(X, \dd)$ be a metric space. We denote by $\mathcal{P}(X)$ the space of its (Borel) probability measures. We say that $\mu \in \mathcal{P}(X)$ has \textit{finite second moment} if $\int_{X} \dd^2(x, x_0)\,d\mu(x) < \infty$ for some (and hence all) $x_0 \in X$, and denote by $\mathcal{P}_2(X) \subseteq \mathcal{P}(X)$ the set of all such probability measures. Given $\mu_1, \mu_2 \in \mathcal{P}_2(X)$, the \textit{$L^2$-Wasserstein distance} $W_2$ between them is defined by
\begin{equation}\label{Eq: L2 Wass def}
W_2^2(\mu_1, \mu_2) := \inf_\gamma \int_{X \times X} \dd^2(x,y) \, d\gamma(x,y),
\end{equation}
where the infimum is taken over all $\gamma \in \mathcal{P}(X \times X)$ satisfying $(\pi_1)_*(\gamma)=\mu_1$ and $(\pi_2)_*(\gamma)=\mu_2$, where $\pi_j: X \times X \to X$ is the projection onto the $j$-th coordinate for $j = 1,2$. Such measures $\gamma$ are called \textit{admissible plans} for the pair $(\mu_1, \mu_2)$ and we denote by $\Adm(\mu_1, \mu_2)$ the set of all admissible plans. We say that an admissible plan $\gamma$ is an \textit{optimal plan} if the infimum in \eqref{Eq: L2 Wass def} is realized at $\gamma$, and denote by $\Opt(\mu_1, \mu_2)$ the set of all optimal plans for $(\mu_1, \mu_2)$. It turns out optimal plans always exist for the cost function $c = \dd^2$ if $(X,\dd)$ is Polish (see for instance \cite[Theorem 1.5]{AG13}).

In the case where $(X,\dd)$ is geodesic,  $(\mathcal{P}_2(X), W_2)$ is also geodesic. Moreover, any Wasserstein geodesic $(\mu_t)_{t \in [0,1]}$ in $(P_2(X), W_2)$ can be lifted to a measure $\nu \in \mathcal{P}(\Geo(X))$ so that $(e_t)_{*}(\nu)=\mu_t$ for all $t \in [0,1]$, where $e_t: \Geo(X) \to X$ is the evaluation map
\begin{equation}\label{Eq: evaluation map def}
    e_t(\gamma) := \gamma(t).
\end{equation}
We say that $\nu \in \mathcal{P}(\Geo(X))$ is a \textit{dynamical optimal plan} from $\mu_0$ to $\mu_1$ if $(e_0, e_1)_{*}(\nu) \in \Opt(\mu_0, \mu_1)$, and denote by $\OptGeo(\mu_0, \mu_1)$ the space of all such plans.

Let $(X, \dd)$ be a geodesic space.  
\begin{Def}\label{Def: non-branching} (Non-branching condition).
Given $\gamma_1, \gamma_2 \in \Geo(X)$, we say that $\gamma_1$ and $\gamma_2$ are \textit{branching} if $\gamma_1 \neq \gamma_2$ and there exists $t \in(0,1)$ so that, for any $s \in [0,t]$, $\gamma_1(s) = \gamma_2(s)$. If $\gamma_1, \gamma_2$ are branching, then there exists a unique $t_0 \in (0,1)$ such that the following hold:
\begin{enumerate}
    \item $\gamma_1(t_0) = \gamma_2(t_0)$;
    \item $\gamma_1(s) = \gamma_2(s)$ for all $s \in [0,t_0]$;
    \item for any $\veps > 0$, there exists $s \in (t_0, t_0+\veps)$ such that $\gamma_1(s) \neq \gamma_2(s)$.
\end{enumerate}
We say that a set $S \subseteq \Geo(X)$ is \textit{non-branching} if there does not exist $\gamma_1, \gamma_2 \in \Geo(X)$ which are branching, and that $X$ is \textit{non-branching} if $\Geo(X)$ is non-branching.
\end{Def}

Many results were shown for various kinds of $\CD$ spaces under an additional non-branching assumption. In \cite{RS14}, a weaker condition called \textit{essentially non-branching} was introduced under which most of these results generalize.

\begin{Def}\label{Def: essentially non-branching} (Essentially non-branching condition).
A geodesic metric measure space $(X, \dd, \mm)$ is said to be \textit{essentially non-branching} if any dynamical optimal plan $\nu$ between two probability measures $\mu_0, \mu_1 \in \mathcal{P}(X)$, with $\mu_0, \mu_1 \ll \mm$, is concentrated on a non-branching set, i.e., there exists $A \subseteq \Geo(X)$ Borel so that $\nu(\Geo(X) \setminus A) = 0$ and $A$ is non-branching. 
\end{Def}

It was shown in \cite{RS14} that $\RCD(K,N)$ spaces are essentially non-branching. This was strengthened to non-branching in \cite{D25}. Later in our paper, we will want to obtain global geometric properties from local ones (see for instance Theorem~\ref{Thm: main globalization theorem}), which motivates the following definition.

\begin{Def}\label{Def: locally essentially non-branching}(Local essentially non-branching condition).
Let $U \subseteq X$ be open and bounded. We say that $X$ is \textit{essentially non-branching on $U$} if any dynamical optimal plan $\nu$ between two probability measures $\mu_0, \mu_1$ supported in $U$, with $\mu_0, \mu_1 \ll \mm$, is concentrated on a non-branching set. For $p \in X$ we say that $X$ is \textit{locally essentially non-branching at $p$} if there exists an open and bounded $U$ so that $p \in U$ and $X$ is essentially non-branching on $U$. 
\end{Def}

We now review the entropic curvature-dimension condition $\CD^{e}(K,N)$ of \cite{EKS15}. Let $(X, \dd, \mm)$ be a metric measure space. Given a measure $\mu \in \mathcal{P}_2(X)$, we define its relative entropy (with respect to $\mm$) by
\begin{equation*}
    \Ent(\mu) := \int \rho \log(\rho) \, d\mm,
\end{equation*}
if $\mu = \rho\mm$ is absolutely continuous with respect to $\mm$ and $(\rho\log\rho)_+$ is integrable. Otherwise, we set $\Ent(\mu):=\infty$. We denote by $\mathcal{P}_2^*(X) \subseteq \mathcal{P}_2(X)$ the subset of probability measures whose entropy is finite. For $N > 0$, define the functional $U_N: \mathcal{P}_2(X) \to [0, \infty]$ by
\begin{equation*}
    U_N(\mu):=\exp\Big(-\frac{1}{N}\Ent(\mu)\Big).
\end{equation*}
For $K \in \R$ and $N > 0$, we define the \textit{distortion coefficients} $\sigma_{K, N}^{t}$ by: 
\begin{align}\label{Eq: distortion coefficients}
\begin{split}
 (t, \theta) \in [0,1] \times \mathbb{R}^{+} \mapsto \sigma^{(t)}_{K,N}(\theta):=\begin{cases}
	\infty &\text{if } K\theta^2 \geq N\pi^{2},\\
	\frac{\sin(t\theta\sqrt{K\slash N })}{\sin(\theta\sqrt{K \slash N})} &\text{if } 0 < K\theta^2 < N\pi^2,\\
	t &\text{ if } K\theta^2 = 0,\\
	\frac{\sinh(t\theta\sqrt{-K\slash N })}{\sinh(\theta\sqrt{-K \slash N})} &\text{if } K\theta^2 < 0.\\
	\end{cases}
\end{split}
\end{align}

The entropic curvature-dimension condition is defined as follows. 
\begin{Def}\label{Def: CDe definition}(Entropic curvature-dimension condition \cite[Definition 3.1]{EKS15}).
Let $K \in \R$ and $N \in (0, \infty)$. We say that a geodesic metric measure space $(X, \dd, \mm)$ satisfies the \textit{entropic curvature-dimension condition} $\cde$ if, for any absolutely continuous $\mu_0 = \rho_0\mm, \mu_1 = \rho_1\mm \in \mathcal{P}^*_2(X)$, there exists a Wasserstein geodesic $(\mu_t)_{t \in [0,1]}$ connecting $\mu_0$ and $\mu_1$ so that for all $t \in [0,1]$ we have
\begin{equation}\label{Eq: entropy convexity}
U_N(\mu_t)\geq \sigma_{K,N}^{(1-t)}(W_2(\mu_0, \mu_1))U_N(\mu_0)+\sigma_{K,N}^{(t)}(W_2(\mu_0, \mu_1))U_N(\mu_1).
\end{equation}
If \eqref{Eq: entropy convexity} holds for any Wasserstein geodesic $(\mu_t)$ between $\mu_0$ and $\mu_1$ we say that $(X, \dd, \mm)$ satisfies the \emph{strong $\cde$ condition}. 
\end{Def}

We also use the following local definitions.
\begin{Def}\label{Def: local CDe definition}
Let $U \subseteq X$ be open and bounded. We say that $X$ is \textit{locally $\cde$ on $U$} if, for any absolutely continuous probability measures $\mu_0 = \rho_0\mm, \mu_1 = \rho_1\mm \in \mathcal{P}_2^*(X)$ supported in $U$, there exists a Wasserstein geodesic $(\mu_t)_{t \in [0,1]}$ connecting $\mu_0$ and $\mu_1$ so that \eqref{Eq: entropy convexity} holds. Similarly, we say that $X$ is \textit{locally strongly $\cde$ on $U$} if \eqref{Eq: entropy convexity} holds for any Wasserstein geodesic $(\mu_t)$ between $\mu_0$ and $\mu_1$. Finally, for any $p \in X$, we say that $X$ is \textit{locally (resp.\ locally strongly) $\cde$ at $p$} if there exists an open and bounded $U$ so that $p \in U$ and $X$ is locally (resp.\ locally strongly) $\cde$ on $U$. 
\end{Def}

The $\cde$ condition has the advantage that it is easier to work with (for instance in terms of proving globalization properties) than the $\CD(K,N)$ condition from \cite{S06b} and the $\CD^*(K,N)$ condition from \cite{BS10}. Under an essentially non-branching condition, the three conditions are equivalent by \cite[Theorem 3.12]{EKS15} and \cite[Theorem 1.1]{CM21} (see also \cite[Theorem 1.1]{L24}). 

As noted in \cite{EKS15}, a local $\cde$ condition implies a local essentially non-branching condition.
\begin{Lem}\label{Lem: cde implies essential non-branching}
Let $R>0$. If $(X, \dd, \mm)$ is locally strongly $\cde$ on $B_{2R}(p)$ for some $p \in X$, then $X$ is essentially non-branching on $B_{R}(p)$.
\end{Lem}

\begin{proof}
By \cite[Lemma 2.12]{EKS15}, we see that if $X$ is locally strongly $\cde$ on $B_{2R}(p)$ then it is locally strongly $\CD(K, \infty)$ on $B_{2R}(p)$, where the definition of $\CD(K,\infty)$ (see \cite[Definition 4.5]{S06a}, \cite[Definition 0.7]{LV09}) is extended in an analogous way (as Definition~\ref{Def: local CDe definition}) to locally strongly $\CD(K, \infty)$. The proof of \cite[Theorem 1.1]{RS14} can then be repeated verbatim to conclude. 
\end{proof}

Denote by $\Geo_p(X)$ the set of geodesics in $X$ that passes $p$, i.e., 
\begin{equation*}
\Geo_p(X) := \{\gamma \in \Geo(X)\,:\,\exists t \in [0,1] \text{ such that } \gamma(t) = p\}.
\end{equation*}
We will need the following globalization theorem (cf.\ \cite[Theorem 3.14]{EKS15}). 

\begin{Thm}\label{Thm: main globalization theorem} (Globalization theorem). 
Let $R > 0$ and $p \in X$. Suppose that the following conditions hold:
\begin{enumerate}
    \item there exists a countable open cover of $B_{2R}(p) \setminus \{p\}$ by open balls $\{B_{r_i}(p_i)\}$ so that $X$ is locally strongly $\cde$ on $B_{10r_i}(p_i)$ for each $i$;
    \item for any optimal dynamical plan $\nu$ between two absolutely continuous measures supported in $B_{R}(p)$, $\nu(\Geo_p(X))=0$.
\end{enumerate}
Then $X$ is locally strongly $\cde$ on $B_{R}(p)$. 
\end{Thm}

\begin{proof}
We only sketch the proof, as it uses many of the same ingredients as the proof of \cite[Theorem 3.14]{EKS15}. For each $k \in \N$, define $G_k \subseteq \Geo(X)$ as the set of all geodesics whose image is contained in $\bigcup_{i=1}^k B_{r_i}(p_i)$. For each $k$ let $n_k \in \N$ be sufficiently large so that $10R/r_i < n_k$ for each $i = 1,\dots, k$. Let $I_k = \{1,\dots,k\}^{n_k}$ be the set of all $n_k$-tuples with entries in $\{1,\dots,k\}$. For any $\alpha = (\alpha_1,\dots, \alpha_{n_k}) \in I_k$, we define $H_\alpha \subseteq G_k$ by
\begin{equation*}
    H_{\alpha} := \{\gamma \in G_k \,:\, \gamma(i/n_k) \in B_{r_{\alpha_i}}(p_{\alpha_i}) \text{ for all $i=0,\dots,n_k-1$}\}.
\end{equation*}
Clearly, $\{H_{\alpha}\}_{\alpha \in I_k}$ forms a (not necessarily disjoint) partition of $G_k$. 

Fix any Wasserstein geodesic $(\mu_t)_{t \in [0,1]}$ between $\mu_0 = \rho_0\mm, \mu_1 = \rho_1\mm \in \mathcal{P}_2^*(X)$ concentrated in $B_{R}(p)$ and let $\nu \in \mathcal{P}(\Geo(X))$ be an associated optimal dynamical plan, i.e., $(e_t)_*(\nu) = \mu_t$ for each $t \in [0,1]$. By the triangle inequality, $\nu$ is concentrated on a set of geodesics of length at most $2R$ and whose image is contained in $B_{2R}(p)$. From this and our choice of $n_k$, for any $\alpha \in I_k$, $\nu \mrestr {H_\alpha}$ must be concentrated on the set of curves $H_\alpha' \subseteq H_\alpha$ defined by
\begin{equation*}
    H'_{\alpha} := \{\gamma \in H_{\alpha} \,:\, \gamma([i/n_k, (i+2)/n_k]) \subseteq B_{2r_{\alpha_i}}(p_{\alpha_i}) \text{ for any $i = 0,\dots, n_{k}-2$}\}.
\end{equation*}
For any $\alpha \in I_k$ where $\nu(H_\alpha) > 0$, we define $\nu_\alpha := \nu \mrestr {H_{\alpha}}/\nu(H_\alpha)$. For each $i = 0,\dots, n_k-2$, let $F_i: \Geo(X) \to \Geo(X)$ be the map which takes $\gamma$ to $F_i(\gamma)$ defined by $F_i(\gamma)(t) := \gamma(i/n_k + 2t/n_k)$, i.e., $F_i(\gamma)$ is a linear reparameterization on $[0,1]$ of $\gamma$ restricted to $([i/n_k, (i+2)/n_k]$. Then $(F_i)_{*}(\nu_\alpha)$ is concentrated on the set of geodesics with image in $B_{2r_{\alpha_i}}(p_{\alpha_i})$.

As in the proof of \cite[Theorem 3.14]{EKS15}, we may now apply the locally strong $\cde$ condition of each $B_{2r_{\alpha_i}}(p_{\alpha_i})$ to each $(F_i)_{*}(\nu_\alpha)$ and then use \cite[Lemma 2.8 (i)]{EKS15} to conclude that the Wasserstein geodesic $((e_t)_*(\nu_{\alpha}))_t$ satisfies \eqref{Eq: entropy convexity}. To be slightly more precise than the proof of \cite[Theorem 3.14]{EKS15}, it is not difficult to check (using for instance a local version of \cite[Theorem 1.1]{CM17a}) that the local strong $\cde$ condition on each $B_{4r_{\alpha_i}}(p_{\alpha_i})$ implies that, for any Wasserstein geodesic $(\mu_t)$ between two probabilities $\mu_0, \mu_1$ concentrated in $B_{2r_{\alpha_i}}(p_{\alpha_i})$, if $\mu_0 \in \mathcal{P}_2^*(X)$ then $(\mu_t)$ is the unique Wasserstein geodesic between $\mu_0$ and $\mu_1$, and $\mu_t \in \mathcal{P}_2^*(X)$ for any $t \in [0,1)$. We can then conclude inductively that, since $(e_0)_*((F_0)_*(\nu_{\alpha})) = (e_0)_*(\nu_{\alpha}) \in \mathcal{P}_2^*(X)$, each $(e_{1/2})_*((F_i)_*(\nu_{\alpha})) \in \mathcal{P}_2^*(X)$ as well. As such, \eqref{Eq: entropy convexity} holds for each Wasserstein geodesic $((e_t)_*((F_i)_*(\nu_\alpha)))_t$ and then we can use the characterization of \cite[Lemma 2.8 (i)]{EKS15} to stitch together the estimate. 

For any $k \in \N$ so that $\nu \mrestr {G_k} \neq 0$, we define $\nu_k := \nu \mrestr {G_k}/\nu(G_k)$. Since $I_k$ is finite and for each $\alpha \in I_k$, $(e_t)_*(\nu_\alpha) \in \mathcal{P}_2^*(X)$ for all $t \in [0,1]$, we can conclude that $(e_t)_*(\nu_k) \in \mathcal{P}_2^*(X)$ and in particular $(e_t)_*(\nu_k) \ll \mm$. Applying Lemma~\ref{Lem: essentially non-branching globalization}, we conclude that $\nu_k$ is concentrated on a non-branching set. It now follows that $\nu$ is also concentrated on a non-branching set. We note that at this point we have proven that any optimal dynamical plan $\nu$ between two probability measures in $\mathcal{P}_2^*(X)$ and concentrated in $B_{R}(p)$ is concentrated on a non-branching set. Keeping in mind Lemma~\ref{Lem: mutually disjoint}, we may now argue as in \cite[Theorem 3.14]{EKS15} to conclude that \eqref{Eq: entropy convexity} holds for each $((e_t)_*(\nu_k))_t$, and then finally that it holds for $((e_t)_*(\nu))_t$. 
\end{proof}

\begin{Lem}\label{Lem: essentially non-branching globalization}
Let $R > 0$ and $p \in X$. Suppose that there exists a countable open cover of $B_{2R}(p) \setminus \{p\}$ by open balls $\{B_{r_i}(p_i)\}$ where $X$ is essentially non-branching on each $B_{10r_i}(p_i)$. Let $\nu$ be an optimal dynamical plan between two absolute measures $\mu_0$, $\mu_1$ supported in $B_{R}(p)$ so that:
\begin{enumerate}
    \item $\mu_t := (e_t)_*(\nu) \ll \mm$ for every $t \in [0,1]$;
    \item $\nu(\Geo_p(X))=0$.
\end{enumerate}
Then $\nu$ is concentrated on a non-branching set. 
\end{Lem}

\begin{proof}
By the triangle inequality, $\nu$ is concentrated on a set of geodesics of length at most $2R$ and whose image is contained in $B_{2R}(p)$. 

For each $k \in \N$, define $G_k \subseteq \Geo(X)$ to be the set of geodesics whose image is contained in $\bigcup_{i=1}^k B_{r_i}(p_i)$. It suffices to prove that each $\nu_k := \nu \mrestr {G_k}$ is concentrated on a non-branching set, i.e., there exists non-branching $A_k'$ so that $\nu_k(\Geo(X) \setminus A_k')=0$ (notice that any null measure is concentrated on the non-branching set $\emptyset$). Indeed, in this case by adjusting $A_k'$ one can construct a monotone sequence of non-branching sets $A_k$ so that $\nu_k$ is concentrated on $A_k$. It follows that $A := \bigcup_{k=1}^{\infty} A_k$ is a non-branching set. Clearly, $\nu$ is concentrated on $A$ by the second assumption of the lemma. 

Suppose for the sake of contradiction that there exists some $k$ so that $\nu_k$ is not concentrated on any non-branching set. Let $n$ be sufficiently large so that $10R/n < r_i$ for each $i = 1,\dots,k$. For each $j = 0, \dots , n-2$, define the map $F_j: \Geo(X) \to \Geo(X)$ which takes $\gamma$ to $F_j(\gamma)$ given by $F_j(\gamma)(t) := \gamma(j/n+2t/n)$, and define $\nu_{k,j} := (F_j)_*(\nu_k)$. There must exist some $j$ such that $\nu_{k,j}$ is not concentrated on any non-branching set. 

Finally, for each $m = 1,\dots,k$, we define $T_m \subseteq \Geo(X)$ to be the set of geodesics $\gamma$ so that $\gamma(0) \in B_{r_m}(p_m)$. Defining $\nu_{k,j,m} = \nu_{k, j} \mrestr {T_m}$, again there must be some $m$ so that $\nu_{k, j, m}$ is not concentrated on any non-branching set. Due to our choice of $n$ and the observation that $\nu$ is supported on the set of geodesics with length less than $2R$, we see that $(e_{0})_*(\nu_{k, j, m}), (e_{1})_*(\nu_{k, j, m})$ must be supported in $B_{10r_i}(p_i)$. Moreover, by our initial assumption that $(e_t)_*(\nu) \ll \mm$, we have $(e_{0})_*(\nu_{k, j, m}), (e_{1})_*(\nu_{k, j, m}) \ll \mm$ as well. Therefore, $\nu_{k,j,m}/\nu_{k, j, m}(\Geo(X))$ is an optimal dynamical plan between two absolutely continuous probability measures supported in $B_{10r_m}(p_m)$ that is not concentrated on any non-branching set. This is a contradiction. 
\end{proof}

We have the following minor variant of \cite[Lemma 3.11]{EKS15}.
\begin{Lem}\label{Lem: mutually disjoint}
Let $(X, \dd, \mm)$ be a geodesic metric measure space. Let $\mu_0, \mu_1 \in \mathcal{P}(X)$ be so that $\mu_0, \mu_1 \ll \mm$ and assume that any optimal dynamical plan between $\mu_0, \mu_1$ is concentrated on a non-branching set. Assume that $\nu = \sum_{k = 1}^{n} \alpha_k\nu^k$ for $\alpha_k > 0$ and optimal dynamical plans $\nu_k$. Define $\mu^k_t := (e_t)_*(\nu^k)$ for each $t \in [0,1]$. If $\{\mu^k_0\}_{k}$ is mutually singular, then $\{\mu_t^k\}_k$ is also mutually singular for each $t \in (0,1)$.
\end{Lem}

The proof is exactly the same as that of \cite[Lemma 3.11]{EKS15}. We note that in \cite{EKS15} it was assumed that $X$ is essentially non-branching. However, as is evident in the proof, for any particular $\mu_0, \mu_1$, it is only necessary to assume that any optimal dynamical plan between $\mu_0$ and $\mu_1$ is concentrated on a non-branching set. 

We now give some properties of the local $\cde$ condition. By localizing the proof of \cite[Proposition 3.6]{EKS15}, we have the following version of the Bishop--Gromov inequality.
\begin{Pro}\label{Pro: Bishop--Gromov}(Local Bishop--Gromov inequality). Let $R > 0$. If $(X, \dd, \mm)$ satisfies the local $\cde$ condition on $B_{2R}(p)$ for some $p \in X$, then for any $q \in B_R(p)$ and $0 < r_1 < r_2 \leq \min\{R - \dd(p,q), \pi\sqrt{N/K^+}\}$ we have
\begin{equation*}
\frac{\mm(B_{r_1}(q))}{\mm(B_{r_2}(q))} \geq \frac{\int_{0}^{r_1} \sin_{K/N}(t)^N \,dt}{\int_{0}^{r_2} \sin_{K/N}(t)^N \,dt},
\end{equation*}
where for each $\kappa \in \R$ the function $\sin_{\kappa}: [0, \infty] \to \R$ is defined by
\begin{align*}
\sin_{\kappa}(t):=\begin{cases}
    \frac{1}{\sqrt{\kappa}}\sin(\sqrt{\kappa}t) &\text{if } \kappa > 0,\\
	\; t &\text{if $\kappa =0$,}\\
    \frac{1}{\sqrt{-\kappa}}\sinh(\sqrt{-\kappa}t) &\text{if  $\kappa < 0$.}
	\end{cases}
\end{align*}
\end{Pro}

By localizing the arguments of \cite[Theorem 3.1]{CM17a}, we have the following variant of the theorem.
\begin{Pro}\label{Pro: local cde density control}
Let $R > 0$, $K \in \R$, and $N \in [1,\infty)$. There exists a locally bounded function $C_{K, N, R}: [0,1) \to \R_+$ (i.e., $C_{K, N, R}$ is bounded on $[0,s)$ for any $s < 1$) such that the following holds: Let $(X, \dd, \mm)$ satisfy the local $\cde$ condition on $B_{2R}(p)$ for some $p \in X$. Let $\mu_0 = \rho_0\mm \ll \mm$ be supported in $B_{R}(p)$ with $\rho_0 \in L^\infty(\mm)$, and let $q \in B_{R}(p)$. Then there exists an optimal dynamical plan $\nu$ between $\mu_0$ and $\mu_1 := \delta_q$ so that $\mu_t := (e_t)_*(\nu) \ll \mm$ for any $t = [0,1)$. Moreover, writing $\mu_t = \rho_t\mm$ we have that $\rho_t \in L^\infty(\mm)$ and
\begin{equation*}
    \norm{\rho_t}{L^{\infty}(\mm)} \leq C_{K,N,R}(t) \norm{\rho_0}{L^\infty(\mm)}.
\end{equation*}
\end{Pro}

Given a metric measure space $(X, \dd, \mm)$, we say that $X$ is \textit{$c$-doubling} for some $c > 0$ provided that 
\begin{equation*}
    \mm(B_{2r}(x)) \leq c\mm(B_{r}(x))
\end{equation*}
for any $x \in X$ and any $r > 0$; recall that we assume $\mm\neq0$. For a fixed $c > 0$, a family of bounded and $c$-doubling  spaces is uniformly totally bounded (i.e., for any $\veps$ there exists $n_\veps$ such that each of the spaces in the family can be covered by at most $n_{\veps}$ balls of radius $\veps$). From this, we have the following version of Gromov's precompactness theorem, whose statement we take from \cite[Lemma 3.32]{GMS15} (see also \cite[Theorem 8.8.10]{BBI01}).
\begin{Lem}\label{Lem: pmGH precompactness}
Let $(X_n ,\dd_n, \mm_n, x_n)$, $n \in \N$, be a sequence of pointed metric measure spaces. Assume that, for some $c > 0$, all of them are $c$-doubling and there exists $r > 0$ such that $0<\liminf_{n\to\infty} \mm_n(B_r(x_n)) < \infty$. Then the sequence is precompact in the pmGH topology and any limit space $(X_\infty, \dd_\infty, \mm_\infty, x_\infty)$ is $c$-doubling as well. 
\end{Lem}

Let $(X, \dd, \mm)$ be locally $\cde$ on $B_{4R}(p)$ for some $R > 0$ and $p \in X$. Consider $(\overline{B_{R}(p)}, \dd', \mm')$, where $\dd'$ is the restriction of $\dd$ to $\overline{B_{R}(p)}$ and $\mm'$ is the normalization of the restriction of $\mm$ to $\overline{B_{R}(p)}$, i.e.,
\begin{equation*}
    \mm' = \frac{\mm \mrestr {\overline{B_{R}(p)}}}{\mm(\overline{B_{R}(p)})}. 
\end{equation*}
Since $(X, \dd, \mm)$ is assumed to satisfy the usual conditions we require for metric measure spaces, $(B_{R}(p), \dd', \mm')$ does as well. In particular, $(\overline{B_R(p)}, \dd')$ is a compact metric space (and therefore Polish) and, since $\supp(\mm)=X$, we have $\supp(\mm') = \overline{B_R(p)}$. Note also that, although $(\overline{B_{R}(p)}, \dd', \mm')$ is no longer a geodesic space, any $x, y \in B_{R/2}(p)$ can be connected by a geodesic in $\overline{B_{R}(p)}$. The local Bishop--Gromov inequality (Proposition~\ref{Pro: Bishop--Gromov}) implies that $(X, \dd, \mm)$ is $c_{R, K, N}$-doubling for balls in $B_{2R}(p)$ and so, in view of (the proof of) Lemma~\ref{Lem: pmGH precompactness}, we have the following lemma.
\begin{Lem}\label{Lem: local cde ball precompactness}
Let $R>0$, $K \in \R$, and $N \in [1, \infty)$. Let $(X_i, \dd_i, \mm_i, p_i)$ be a sequence of pointed metric measure spaces. Assume that $(X_i, \dd_i, \mm_i)$ is locally $\cde$ on $B_{4R}(p_i)$ for each $i \in \N$. Then  the sequence $(\overline{B_R(p_i)}, \dd_i', \mm_i', p_i)$ is precompact in the pmGH topology. 
\end{Lem}
Arguing now as in \cite[Theorem 4.9]{GMS15}, we have the following stability theorem for the local $\cde$ condition.
\begin{Lem}\label{Lem: local cde stability}
Any limit space $(X_\infty, \dd_\infty, \mm_\infty, p_\infty)$ of a subsequence of $(\overline{B_R(p_i)}, \dd_i', \mm_i', p_i)_i$ appearing in Lemma~\ref{Lem: local cde ball precompactness} is locally $\cde$ on $B_{R/2}(p_\infty)$.
\end{Lem}

\begin{Rem}
We note that there is a minor technical issue with Lemma~\ref{Lem: local cde stability} as stated since the local $\cde$ condition was defined only for geodesic spaces and $(X_\infty, \dd_\infty, \mm_\infty, p_\infty)$ is not necessarily a geodesic space. Nevertheless, as noted earlier, any two points in $B_{R/2}(p_i)$ are connected by a geodesic and this property carries over to $B_{R/2}(p_\infty)$. This is enough to formulate the local $\cde$ condition for $B_{R/2}(p_\infty)$, even if $X_\infty$ is not geodesic; this formulation is what we mean when we say $(X_\infty, \dd_\infty, \mm_\infty)$ is locally $\cde$ on $B_{R/2}(p_\infty)$. Moreover, since the proofs for Propositions~\ref{Pro: Bishop--Gromov} and \ref{Pro: local cde density control} are completely local, they both hold with $B_{R/2}(p_\infty)$ in place of the larger ball where the local $\cde$ condition is assumed. 
\end{Rem}

\subsection{Ramified double cover}\label{subsec: Ramified double cover}
Next, we discuss the notion of orientable ramified double cover of an $\RCD$ space introduced in \cite{BBP24}.
We begin by recalling the following definition of \textit{orientability} for $\RCD$ spaces.
\begin{Def}\label{Def: Orientability}(Orientable $\RCD$ spaces \cite[Definition 1.6]{BBP24}).
Let $(X, \dd, \mathcal{H}^n)$ be a non-collapsed $\RCD(-(n-1),n)$ space without boundary. We say that $X$ is \textit{orientable} if every open $A \subseteq X$ that is a topological manifold is orientable.
\end{Def}
As shown in \cite{BBP24}, the above definition turns out to be equivalent to the one given in \cite{H17} for non-collapsed Ricci limits. In the setting of non-collapsed $\RCD(-(n-1),n)$ spaces without boundary, the $\veps$-regularity theorem of \cite{CC97, DPG18} ensures that, for $0 < \veps < \veps(n)$, the open subset
\begin{equation}\label{Eq: epsilon regular set}
    A_\veps(X) := \Bigg\{x \in X \,:\, \frac{\HH^n(B_r(x))}{\omega_nr^n} > 1- \veps \text{ for some $r \in (0, \veps)$}\Bigg\},
\end{equation}
where $\omega_n := \HH^n(B_1(0^n))$, is a connected topological manifold without boundary whose complement has Hausdorff dimension smaller than or equal to $n-2$ (see \cite{KM21}). It turns out that to check the orientability of $X$ it suffices to check orientability with respect to only one such $A_\veps(X)$. More precisely, we have the following proposition.
\begin{Pro}\label{Pro: Equivalent definition of orientability}(Equivalent definition of orientability). 
Let $(X, \dd, \mathcal{H}^n)$ be a non-collapsed $\RCD(-(n-1),n)$ space without boundary. Then $X$ is orientable in the sense of Definition~\ref{Def: Orientability} if and only if there exists some $\veps < \veps(n)$ so that $A_\veps(X)$ is orientable.
\end{Pro}
We refer to \cite[Theorem 2.1]{BBP24} for the proof of this fact and other equivalent definitions of orientability. Both definitions can be applied locally following \cite[Remark 1.8]{BBP24}.
\begin{Def}\label{Def: local orientability}(Local orientability).
Let $(X, \dd, \mathcal{H}^n)$ be a non-collapsed $\RCD(-(n-1),n)$ space without boundary. We say that an open subset $U \subseteq X$ is \textit{orientable} if $U \cap A$ is orientable for every $A$ as in Definition~\ref{Def: Orientability}. Moreover, this is equivalent to $U\cap A_\veps(X)$ being orientable as in Proposition~\ref{Pro: Equivalent definition of orientability}.
\end{Def}

We will make use of the concept of orientation-reversing and preserving loops. We refer to the beginning of \cite[Subsection 2.1]{BBP24} for a discussion on the definitions and properties of such loops in the context of manifolds. 
\begin{Def}\label{Def: orientation-reversing loop} (Orientation-reversing and preserving loops) Let $(X, \dd, \HH^n)$ be a non-collapsed $\RCD(-(n-1),n)$ without boundary. We say that a continuous loop $\gamma:[0,1] \to X$ (i.e., $\gamma(0) = \gamma(1)$) is orientation-reversing (resp.\ orientation-preserving) if the following hold:
\begin{enumerate}
    \item there exists an open $A \subseteq X$ that is a topological manifold so that $\gamma([0,1]) \subseteq A$;
    \item $\gamma$ is orientation-reversing (resp.\ orientation-preserving) in the standard sense as a loop in $A$.
\end{enumerate}
\end{Def}
From standard algebraic topology, it is clear that any continuous loop whose image lies in the topological manifold part of $X$ is either orientation-reversing or preserving, and the definition is independent of the choice of $A$. The following is immediate from the definitions. 
\begin{Pro}\label{Pro: non-orientable loops equivalence}
Let $U \subseteq X$ be open. Then $U$ is orientable if and only if there are no orientation-reversing loops whose image lies in $U$.  
\end{Pro}

Orientability in stable under non-collapsing Gromov--Hausdorff convergence \cite[Theorem 1.16]{BBP24}. We note that the non-collapsed limit of a sequence of a non-collapsed $\RCD(-(n-1),n)$ spaces without boundary is also without boundary  \cite[Theorem 1.6]{BNS22}. 
\begin{Thm}\label{Thm: stability of orientability}(Stability of orientability).
Let $(X_i, \dd_i, \HH^n, p_i)$ be a sequence of non-collapsed $\RCD(-(n-1),n)$ spaces without boundary which converges in the pointed measured Gromov--Hausdorff sense to $(X, \dd, \HH^n, p)$. If $(X_i, \dd_i, \HH^n)$ is orientable for every $i$ then $(X, \dd, \HH^n)$ is orientable. 
\end{Thm}
\begin{Rem}\label{Rem: local stability of orientability}(Local stability of orientability).
It follows from the proof of the above theorem that the stability of orientability also holds locally. In other words, if $B_{r}(p_i)$ is orientable for every $i$, then $B_{r}(p)$ is orientable. 
\end{Rem}
On the other hand, the stability of non-orientability is only known for non-collapsed Ricci limit spaces \cite[Theorem 1.18]{BBP24}, and remains an open question for non-collapsed $\RCD$ spaces. 

As a key tool,
for non-orientable $\RCD$ spaces one can construct a \textit{ramified double cover}.
\begin{Thm}\label{Thm: Ramified double cover existence}(Ramified double cover for $\RCD$ spaces \cite[Theorem 3.1]{BBP24}).
Let $(X, \dd, \HH^{n})$ be a non-orientable, non-collapsed $\RCD(-(n-1), n)$ space without boundary. Then there exists a geodesic metric measure space $(\hat{X}, \hat{\dd}, \HH^n)$ along with an involutive isometry $\Gamma: \hat{X} \to \hat{X}$ so that the following hold: 
\begin{enumerate}
    \item $X = \hat{X}/\langle \Gamma \rangle$, and we denote by $\pi: \hat{X} \to X$ the projection map;
    \item there exists an open and dense subset $\hat{A} \subseteq \hat{X}$ that is an orientable topological manifold and a length space so that, denoting $A := \pi(\hat{A})$, the map $\pi: \hat{A} \to A$ is a local isometry and a double cover;
    \item the set $\mathcal{R}$ of regular points (i.e., whose tangent cone is $\R^n$) is contained in $A$.
\end{enumerate}
Moreover, the set $A$ may be taken to be $A_\veps(X)$ for $0 < \veps < \veps(n)$. The pair $((\hat{X}, \hat{\dd}, \HH^n), \pi)$ is unique up to isomorphism. Specifically, if $((\hat{X'}, \hat{\dd'}, \HH^n), \pi')$ is another such pair, then there exists an isometry $\Phi:(\hat{X}, \hat{\dd}) \to (\hat{X'}, \hat{\dd'})$ satisfying $\pi' \circ \Phi = \pi$ and thus also $\Phi\circ\Gamma=\Gamma'\circ\Phi$.
\end{Thm}

\begin{Rem}\label{Rem: ramified double cover Polish and proper}
It is not difficult to check from either the construction of $\hat{X}$ or the properties of Remark~\ref{Rem: Ramified double cover properties} below that, since $(X, \dd, \HH^n)$ is assumed to satisfy the usual conditions that we impose on metric measure spaces, $(\hat{X}, \hat{\dd}, \HH^n)$ does as well. In particular, $(\hat{X}, \hat{\dd}, \HH^n)$ is proper and $\HH^n$ is finite on bounded sets of $\hat{X}$ with $\supp(\HH^n) = \hat{X}$.  
\end{Rem}
It is an open question whether $\hat{X}$ is $\RCD$ in general. In the case of Ricci limits, it was shown in \cite[Theorem 1.21]{BBP24} that if $X$ is the non-collapsed Ricci limit of some sequence of non-orientable manifolds, then $\hat{X}$ is the Ricci limit of the double cover of the same sequence and hence $\RCD$. 
Nonetheless, in general
we have the following useful properties for the ramified double cover \cite[Remark 3.2]{BBP24}.
\begin{Rem}\label{Rem: Ramified double cover properties}
In the notations of Theorem~\ref{Thm: Ramified double cover existence}:
\begin{enumerate}
    \item $\pi$ is surjective and, if $x = \pi(\hat{x})$, then the fiber $\pi^{-1}(x) = \{\hat{x}, \Gamma(\hat{x})\}$;
    \item $\dd(\pi(\hat{x}), \pi(\hat{y})) = \min\{\hat{\dd}(\hat{x}, \hat{y}), \hat{\dd}(\Gamma \hat{x}, \hat{y})=\hat\dd(\hat x,\Gamma\hat y)\}$ for every $\hat{x}, \hat{y} \in \hat{X}$;
    \item $\pi$ is $1$-Lipschitz, and pushes forward the measure $\HH^n$ as $\pi_{*}(\HH^n_{\hat{X}})=2\HH^n_X$;
    \item $\hat{A} \subseteq \{\hat{x}: \Gamma\hat{x} \neq \hat{x}\}$, where both sets are open and dense;
    \item $\dim_{\HH}(\hat{X} \setminus \hat{A}) \leq n-2$;
    \item $\diam(\hat{X}) \leq 2\diam(X)$.
\end{enumerate}
\end{Rem}

Properties $(1)$ and $(2)$ of the above remark easily give the following lemma. 
\begin{Lem}\label{Lem: Geodesic break}
If $\gamma:[0,1] \to \hat{X}$ is a constant speed geodesic and, for some $t_0 \in [0,1]$, $\gamma(t_0)$ is the unique lift of $\pi(\gamma(t_0))$, then $(\pi \circ \gamma )\rvert_{[0,t_0]}$ and $(\pi \circ \gamma) \rvert_{[t_0, 1]}$ are both geodesics. 
\end{Lem}
Indeed, since $\gamma(t_0)$ is the unique lift of $\pi(\gamma(t_0))$, the action of $\Gamma$ on $\gamma(t_0)$ is trivial by property $(1)$. Property $(2)$ then gives 
\begin{align*}
\dd_{\hat{X}}(\gamma(0), \gamma(t_0)) &= \dd_{X}(\pi(\gamma(0)), \pi(\gamma(t_0))),\\
\dd_{\hat{X}}(\gamma(t_0), \gamma(1)) &= \dd_{X}(\pi(\gamma(t_0)), \pi(\gamma(1))),
\end{align*}
which is all that is needed to conclude Lemma~\ref{Lem: Geodesic break}.
We will also make use of the following lemma.
\begin{Lem}\label{Lem: involution bound on non-orientable ball} (\cite[Lemma 3.3]{BBP24}).
Let $(X, \dd, \HH^n)$ be a non-orientable, non-collapsed $\RCD(-(n-1),n)$ space without boundary and let $p \in X$. Let $\pi: (\hat{X}, \hat{\dd}, \HH^n) \to (X, d, \HH^n)$ be the ramified double cover, and let $\hat{p} \in \hat{X}$ be such that $\pi(\hat{p})=p$. Then, for any $R > 0$, $B_R(p)$ is non-orientable if and only if $\hat{d}(\hat{p}, \Gamma \hat{p}) < 2R$.
\end{Lem}

Unlike their smooth counterparts, non-collapsed $\RCD$ spaces \cite[Example 1.22]{BBP24}, and even non-collapsed Ricci limits of dimension $\geq 5$ \cite[Example 1.4]{BBP24}, can have points that have arbitrarily small non-orientable neighborhoods. This motivates the following definition.
\begin{Def}\label{Def: lno points} (Locally non-orientable point \cite[Subsection 1.5]{BBP24}).
Let $(X, \dd, \HH^n)$ be a non-collapsed $\RCD(-(n-1), n)$ space without boundary and let $p \in X$. We say that $p$ is \textit{locally non-orientable} (LNO for short) if any open set $U$ containing $p$ is non-orientable. Otherwise, we say that $p$ is \textit{locally orientable}.
\end{Def}
Clearly, the local non-orientability of $p$ is equivalent to $B_r(p)$ being non-orientable for any $r > 0$. Locally non-orientable points are exactly those that have a unique lift via $\pi$. 
\begin{Pro}\label{Pro: LNO classification}
Let $(X, \dd, \HH^n)$ be a non-orientable, non-collapsed $\RCD(-(n-1),n)$ without boundary and let $p \in X$. Let $\pi: (\hat{X}, \hat{\dd}, \HH^n) \to (X, d, \HH^n)$ be the ramified double cover, and let $\hat{p} \in \hat{X}$ be such that $\pi(\hat{p})=p$. Then $p$ is locally non-orientable if and only if $\pi^{-1}(p) = \{\hat{p}\}$. 
\end{Pro}
\begin{proof}
 Lemma~\ref{Lem: involution bound on non-orientable ball} implies that $\Gamma(\hat{p}) = \hat{p}$ if and only if $p$ is locally non-orientable. Remark~\ref{Rem: Ramified double cover properties} (1) then gives the desired conclusion. 
\end{proof}

Combining the previous proposition with the local stability of orientability, we obtain the following.
\begin{Pro}\label{Lem: non-orientable tangent cone implies fixed point}
Let $(X, \dd, \HH^n)$ be a non-orientable, non-collapsed $\RCD(-(n-1),n)$ space without boundary and let $p \in X$. Let $\pi: (\hat{X}, \hat{\dd}, \HH^n) \to (X, d, \HH^n)$ be the ramified double cover, and let $\hat{p} \in \hat{X}$ be such that $\pi(\hat{p})=p$. If there is a tangent cone at $p$ that is non-orientable, then $\pi^{-1}(p) = \{\hat{p}\}$. In particular, $p$ is locally non-orientable. 
\end{Pro}
\begin{proof}
Since there is a non-orientable tangent cone at $p$, there must be a sequence $r_i \to 0$ so that each $B_{r_i}(p)$ is non-orientable, by the local stability of orientability from Remark~\ref{Rem: local stability of orientability}. Hence, $p$ is locally non-orientable and so Proposition~\ref{Pro: LNO classification} gives that $\pi^{-1}(p) = \{\hat{p}\}$.
\end{proof}

\subsection{Limits of ramified double covers}

In this subsection, we study the pmGH limits of sequences of ramified double covers that satisfy some additional assumptions. More precisely, let $(X_i, \dd_i, \HH^n, p_i)_{i \in \N}$ be a sequence of non-orientable $\RCD(K,n)$ spaces converging to $(X, \dd, \HH^n, p)$ in the pmGH sense. Let $(\hat{X}_i, \hat{\dd_i}, \HH^n, \hat{p}_i)_{i \in \N}$ be the ramified double covers of the spaces in the sequence, with associated projection $\pi_i$ and isometric involution $\Gamma_i$, and assume in addition that there exists $R > 0$ so that 
\begin{enumerate}
    \item $B_{R}(p_i)$ is non-orientable for each $i \in \N$;
    \item $\hat{X}_i$ is locally $\cden$ on $B_{1000R}(\hat{p}_i)$ for each $i \in \N$.
\end{enumerate}
The main result of this subsection is that $B_{R'}(p)$ is non-orientable for any $R' > R$ under the above assumptions. We will use Lemmas~\ref{Lem: local cde ball precompactness} and \ref{Lem: local cde stability}, which say that, up to taking a subsequence, $(\overline{B_{200R}(p_i)})_i$ converges to some $(\tilde{X}, \tilde\dd, \tilde\mm, \tilde{p})$ that is locally $\cden$ on $B_{100R}(\tilde{p})$. A form of the above result was proved and used in \cite[Theorem 4.2]{BBP24} to show the stability of non-orientability for a sequence of spaces $X_i$ under the additional assumption that $\hat{X}_i$ is $\RCD(K,n)$. As we will show, the arguments used in \cite{BBP24} can be adjusted to work under the weaker assumptions of this subsection. 
\begin{Rem}\label{illustration}
We give an example to illustrate the idea behind the proof. A priori it might be possible to construct a sequence of non-orientable $\RCD(K,n)$ spaces $(X_i, p_i)_i$ that converges to $(C(Y^{n-1}), o)$, where $Y^{n-1}$ is an $(n-1)$-sphere of diameter $\ll \pi$, but where all non-orientable loops in $X_i$ must pass through $B_{1/i}(p_i)$. This would be a counterexample to the stability of non-orientability, and so should be conjecturally impossible. For such an example, intuitively $(\hat{X}_i, \hat{p}_i)$ should converge to a space which looks like two copies of $C(Y^{n-1})$ glued at the tip which we denote $\tilde{p}$, with $\hat{p}_i \to \tilde{p}$. Now if in addition we knew that each $\hat{X}_i$ is locally $\cden$ on $B_{1000}(\hat{p_i})$, then $B_{100}(\tilde{p})$ would also be locally $\cden$. The local $\cden$ condition would be incompatible with the limit space looking like two copies of a cone glued at the tip, and hence rule out this counterexample. 
\end{Rem}

Continuing with the notation from the beginning of the subsection, without loss of generality, we may assume $R=1$. By Lemma~\ref{Lem: involution bound on non-orientable ball}, we know that $\dd_{i}(\hat{p}_i, \Gamma_i(\hat{p}_i)) < 2$. Therefore, up to choosing a subsequence, we may assume that $\Gamma_i(\hat{p}_i) \to \tilde{p}'$ under the pmGH convergence, for some $\tilde p' \in X$ with $\dd(\tilde{p}, \tilde{p}') \leq 2$. 

Since $\pi_i, \Gamma_i$ are $1$-Lipschitz, we may apply the Arzelà--Ascoli theorem to take limits (up to a subsequence) of the sequences of maps $\pi_i: \overline{B_{100}(\hat{p}_i)} \to X_i$ and $\Gamma_i: \overline{B_{50}(\hat{p}_i)} \cup \overline{B_{50}(\Gamma_i(\hat{p}_i))} \to \overline{B_{100}(\hat{p}_i)}$. We obtain two $1$-Lipschitz maps $\pi: \overline{B_{100}(\tilde{p})} \to X$ and $\Gamma: \overline{B_{50}(\tilde{p})} \cup \overline{B_{50}(\tilde{p}')} \to \tilde{X}$ so that the following hold: 
\begin{enumerate}
    \item for any sequence $(x_i)_i$, where $x_i \in \overline{B_{100}(\hat{p}_i)}$, if $x_i \to x \in \overline{B_{100}(\tilde{p})}$ under the pmGH convergence, then $\pi_i(x_i) \to \pi(x)$;
    \item for any sequence $(x_i)_i$, where $x_i \in \overline{B_{50}(\hat{p}_i)}\cup \overline{B_{50}(\Gamma_i(\hat{p}_i))}$, if $x_i \to x \in \overline{B_{50}(\tilde{p})\cup B_{50}(\tilde{p}')}$ under the pmGH convergence, then $\Gamma_i(x_i) \to \Gamma(x)$;
    \item $\pi(\tilde{p})=\pi(\tilde{p}')=p$ and $\Gamma(\tilde{p})=\tilde{p}'$.
\end{enumerate}
For notational simplicity, we denote $\overline{B}_p := \overline{B_{50}(\tilde{p})} \cup \overline{B_{50}(\tilde{p}')}$. We list a few additional properties of the limit maps in the next proposition.  
\begin{Pro}\label{Pro: limit projection involution properties}
The following hold:
\begin{enumerate}  
    \item $\Gamma$ is an isometric involution of $\overline{B}_p$; 
    \item $\pi(\overline{B_{100}(\tilde{p})}) = \overline{B_{100}(p)}$;
    \item $\pi^{-1}(\pi(x))=\{x, \Gamma x\}$ for any $x \in \overline{B}_p$;
    \item $\dd(\pi(x), \pi(y)) = \min\{\tilde{\dd}(x, y), \tilde{\dd}(x, \Gamma(y))\}$ for any $x \in \overline{B_{100}(\tilde{p})}$ and $y \in \overline{B}_p$.
\end{enumerate}
\end{Pro}
The proposition is an easy consequence of the properties of the limit maps above and Remark~\ref{Rem: Ramified double cover properties}, so we skip its proof (note that (3) follows from (4)). We can also rule out the possibility that $\Gamma$ is the identity map and $\pi$ is an isometry (this would correspond for instance to a situation where the set of fixed points of $\Gamma_i$ in $\hat{X}_i$ gets denser and denser). More precisely, we have the following lemma.
\begin{Lem}\label{lem: regular set double cover}
For $\veps < \veps(n)$ sufficiently small, if $x \in A_\veps(X) \cap B_{10}(p)$ then $\pi^{-1}(x)$ consists of two points and there exists $r > 0$ so that if $\tilde{x} \in \pi^{-1}(x)$ then $\pi$ is an isometry between $B_{r}(\tilde{x})$ and $B_{r}(x)$. 
\end{Lem}
\begin{proof}
 Let $x \in A_\veps(X) \cap B_{10}(p)$ and let $x_i \in X_i$ be a sequence so that $x_i \to x$ in the pmGH convergence. From a standard argument (see for instance \cite[Theorem 3.1]{KM21}) using the Bishop--Gromov inequality and volume rigidity \cite[Theorem 1.6]{DPG18} for non-collapsed $\RCD(K,n)$ spaces, we have that for any $\delta > 0$ there is an $\veps > 0$ (depending only on $\delta$, $K$, and $n$) and an $r>0$ (depending on $x$) so that, for $i$ sufficiently large, for every $x' \in B_{r}(x_i)$ and every $r' < r$, $d_{GH}(B_{r'}(x'), B_{r'}(0^n)) < \delta r'$. The Cheeger--Colding metric Reifenberg Theorem \cite[Theorem A.1.1]{CC97} then gives that each $B_{r}(x_i)$ is homeomorphic to an open subset of $\R^n$, and hence orientable, as long as $\delta$ (and hence $\veps$) is sufficiently small depending only on $K,n$. By Lemma~\ref{Lem: involution bound on non-orientable ball}, if $\hat{x}_i \in (\pi_i)^{-1}(x_i)$ then $\hat\dd_i(\hat{x}_i, \hat{\Gamma}_i(x_i)) \geq 2r$. Let $\hat{x}_i \in \hat{X}_i$ be a sequence that converges to some $\tilde{x} \in \pi^{-1}(x)$. The sequence $\Gamma_i(\hat{x}_i)$ converges to $\Gamma(\tilde{x}) \in \pi^{-1}(x)$ with $\tilde\dd(\tilde{x}, \Gamma(\tilde{x})) \geq 2r$. This proves the first part. 

 As for the second part, by Remark~\ref{Rem: Ramified double cover properties} and the previous paragraph, we have $\pi^{-1}(B_r(x_i)) = B_r(\hat{x}_i) \cup B_r(\Gamma_i(\hat{x}_i))$ with $B_r(\hat{x}_i) \cap B_r(\Gamma_i(\hat{x}_i)) = \emptyset$. In view of Remark~\ref{Rem: Ramified double cover properties} (2), this implies $\pi$ is an isometry from $B_{r/2}(\hat{x}_i)$ to $B_{r/2}(x_i)$. This clearly passes to the limit by taking a sequence $\hat{x}_i$ converging to $\tilde{x}$, as required. 
\end{proof}

In the proof of \cite[Theorem 4.2]{BBP24}, the fact that the limit measure on the limit space of $\hat{X}_i$ is $\mathcal{H}^n$ was used in the last stage of the argument. Under their assumptions, this comes for free since $(\hat{X}_i)_i$ was assumed to be a non-collapsing sequence of non-collapsed $\RCD(K,n)$ spaces, and so the limit space is also a non-collapsed $\RCD(K,n)$ space by \cite{DPG18}. We also prove a version of this result in our case.  
\begin{Lem}\label{Lem: limit measure = hausdorff}
The limit measure $\tilde\mm$ is equal to $\mathcal{H}^n$ on $B_{10}(\tilde{p}) \cup B_{10}(\Gamma(\tilde{p}))$.
\end{Lem}
\begin{proof}
Let $x \in A_{\veps}(X) \cap B_{10}(p)$ and let $\tilde{x} \in \pi^{-1}(x)$, where $\veps$ is sufficiently small as in Lemma~\ref{lem: regular set double cover}. Let $\hat{x}_i \in \hat{X}_i$ be a sequence that converges to $\tilde{x}$ under the pmGH convergence.  Then, by the proof of Lemma~\ref{lem: regular set double cover}, there exists $r > 0$ so that, for $i$ sufficiently large, $B_{r}(\hat x_i)$ is isometric to $B_{r}(\pi(\hat x_i))$. Now, since $(X_i, \dd_i, \HH^n, x_i) \overset{\text{pmGH}}{\to} (X, \dd, \HH^n, x)$, we must also have $(\overline{B_{r/2}(x_i)}, \dd_i, \HH^n, x_i) \overset{\text{pmGH}}{\to} (\overline{B_{r/2}(x)}, \dd, \HH^n, x)$, where the metrics are simply restrictions from the larger spaces (provided that $\HH^n(\partial B_{r/2}(x))=0$, which holds by Bishop--Gromov).
Similarly, since $(\overline{B_{200}(\hat{p}_i)}, \hat\dd_i, \HH^n, \hat{x}_i) \overset{\text{pmGH}}{\to} (\overline{B_{200}(\tilde{p}_i)}, \tilde\dd, \tilde\mm, \tilde{x})$, we can guarantee that
$$(\overline{B_{r/2}(\hat{x}_i)}, \hat\dd_i, \HH^n, \hat{x}_i) \overset{\text{pmGH}}\to (\overline{B_{r/2}(\tilde{x})}, \tilde\dd, \tilde\mm \mrestr {\overline{B_{r/2}(\tilde{x})}}, \tilde{x}).$$
By the isometry between $\overline{B_{r/2}(x_i)}$ and 
$\overline{B_{r/2}(\hat{x}_i)}$ for large $i$ and the uniqueness of pmGH limits, we must have $\tilde\mm \mrestr {\overline{B_{r/2}(\tilde{x})}} = \HH^n$. 

We have shown that, for any $\tilde{x} \in \big(B_{10}(\tilde{p}) \cup B_{10}(\Gamma(\tilde{p}))\big) \cap \pi^{-1}(A_\veps(X))$, the measure $\tilde\mm$ restricted to $B_r(x)$ is equal to $\HH^n$  for $r$ sufficiently small. It suffices now to prove that 
\begin{equation*}
\HH^n\Big(\big(B_{10}(\tilde{p}) \cup B_{10}(\Gamma(\tilde{p}))\big) \setminus \pi^{-1}(A_\veps(X))\Big) = \tilde\mm\Big(\big(B_{10}(\tilde{p}) \cup B_{10}(\Gamma(\tilde{p}))\big) \setminus \pi^{-1}(A_\veps(X))\Big) = 0
\end{equation*}
to conclude. Since $\HH^n(B_{10}(p) \setminus A_\veps(X)) = 0$, for any $\delta>0$ we can find some cover $\{B_{r_j}(z_j)\}_j$ of $B_{10}(p) \setminus A_\veps(X)$ with $\sum_j \omega_nr_j^{n} < \veps$, $r_j < \delta$, and $z_j \in B_{20}(p)$. Lifting this cover to $\tilde{X}$ and noticing that each ball $B_{r_j}(z_j)$ gives at most two balls of radius $r_j < \delta$, we obtain
\begin{equation*}
    \HH^n_{2\delta}\Big(\big(B_{10}(\tilde{p}) \cup B_{10}(\Gamma(\tilde{p}))\big) \setminus \pi^{-1}(A_\veps(X))\Big) < 2\veps. 
\end{equation*}
Since $\veps, \delta$ are arbitrary we conclude that 
\begin{equation*}
    \HH^n\Big(\big(B_{10}(\tilde{p}) \cup B_{10}(\Gamma(\tilde{p}))\big) \setminus \pi^{-1}(A_\veps(X))\Big) = 0.
\end{equation*} 
Using the very same covers along with Lemma~\ref{Lem: limit measure bound} also shows that 
\begin{equation*}
\tilde\mm\Big(\big(B_{10}(\tilde{p}) \cup B_{10}(\Gamma(\tilde{p}))\big) \setminus \pi^{-1}(A_\veps(X))\Big) = 0. \qedhere
\end{equation*}
\end{proof}

\begin{Lem}\label{Lem: limit measure bound}
    For any $x \in \overline{B_{20}(p)}$ and any $0 < r < 20$, 
    \begin{equation*}
        \HH^n(B_r(x)) \leq \tilde{\mm}(\pi^{-1}(B_r(x))) \leq 4\HH^n(B_r(x)).
    \end{equation*}
\end{Lem}

\begin{proof}
Let $x_i \in {B_{30}(p_i)}$ be a sequence so that $x_i \to x$ in the pmGH convergence. Let $\tilde{x} \in \pi^{-1}(x)$ and choose a sequence  $\hat{x}_i \in \pi_i^{-1}(x_i) \subseteq B_{30}(\hat{p}_i)$ so that $\hat{x}_i \to \tilde{x}$. Then we have $\Gamma_i(\hat{x}_i) \to \Gamma(\tilde{x})$ as well.

By Bishop--Gromov inequality on $\RCD(K,n)$ spaces, we have
\begin{equation*}
    \HH^n(B_{r+\delta}(x_i)\setminus B_{r-\delta}(x_i)) \to 0 \quad \text{as $\delta \to 0$}
\end{equation*}
uniformly in $i$. By the definition of pmGH convergence, this gives
$$\HH^n(B_r(x_i)) \to \HH^n(B_r(x)).$$
Similarly, due to the local Bishop--Gromov inequality from the local $\cden$ property of $B_{1000}(\hat{p}_i)$ (Proposition~\ref{Pro: Bishop--Gromov}), we have
\begin{equation*}
    \HH^n(B_r(\hat{x}_i)) \to \tilde\mm(B_r(\tilde{x})) \quad \text{ and } \quad \HH^n(B_r(\Gamma_i(\hat{x}_i))) \to \tilde\mm(B_r(\Gamma(\tilde{x}))).
\end{equation*}
By Remark~\ref{Rem: Ramified double cover properties} (1)--(3), we have
\begin{equation*}
    \HH^n(B_r(\hat{x}_i)), \HH^n(B_r(\Gamma_i(\hat{x}_i))) \leq 2\HH^n(B_r(x_i)),
\end{equation*}
and so 
\begin{equation*}
    \tilde\mm(B_r(\tilde{x})), \tilde\mm(B_r(\Gamma(\tilde{x}))) \leq 2\HH^n(B_r(x)). 
\end{equation*}
Since $\pi^{-1}(B_r(x)) = B_r(\tilde{x}) \cup B_r(\Gamma(\tilde{x}))$ by Proposition~\ref{Pro: limit projection involution properties} (3)--(4), we obtain the required upper bound. The lower bound follows by a similar argument. 
\end{proof}

We have the following lemma, which is a replacement for \cite[Lemma 2.2]{BBP24}.
\begin{Lem}\label{Lem: double cover limit connetedness}
Let $C \subseteq \overline{B_{10}(p)}$ be closed with $\mathcal{H}^{n-1}(C) = 0$ and let $\tilde{C} := \pi^{-1}(C) \subseteq  B_{15}(\tilde{p})$. Then, for every $x \in B_{5}(\tilde{p})\setminus \tilde{C}$, it holds that, for $\tilde\mm$-a.e.\ $y \in B_{5}(\tilde{p})\setminus \tilde{C}$, there exists a geodesic $\gamma:[0,1] \to B_{10}(\tilde{p})$ connecting $x$ to $y$ with $\gamma([0,1]) \subseteq B_{10}(\tilde{p}) \setminus \tilde{C}$.
\end{Lem}

\begin{proof}
The proof is similar to that of \cite[Lemma 2.2]{BBP24} (see also the proof of \cite[Theorem 1.4]{R12}). For the sake of completeness, we give a full proof with the necessary changes. It suffices to prove that, for any $x, y \in B_{5}(\tilde{p})\setminus \tilde{C}$, there exists $\eta$ so that for $\tilde\mm$-a.e.\ $y' \in B_{\eta}(y)$ there exists a geodesic connecting $x$ to $y'$ with its image contained in $B_{10}(\tilde{p}) \setminus \tilde{C}$. Let $\eta$ be sufficiently small so that $B_{\eta}(y) \in B_{5}(\tilde{p})$; we will fix $\eta$ later with more conditions. 

Since $X$ is a non-collapsed $\RCD(K,n)$ space, there exists some $c_{K,N}$ so that for any $z \in B_{11}(p)$ and any $r < 2$ we have
\begin{equation*}
    \frac{\mathcal{H}^{n}(B_r(z))}{r} \leq c_{K,N}r^{n-1}.
\end{equation*}
From the assumption that $\HH^{n-1}(C) = 0$, for any $\veps > 0$, there exists some finite covering $\{B_{r_j}(z_j)\}_{j \in J}$ of $C$ in $X$ so that $z_j \in B_{11}(p)$, $r_j < 1$, and 
\begin{equation*}
    \sum_{j \in J} (2r_j)^{n-1} \leq \frac{\veps}{16c_{K,N}}. 
\end{equation*}
Moreover, since $x \notin \tilde{C}$, for $\eta$ sufficiently small we can choose a finite covering as above so that in addition $B_{\eta}(\pi(x)) \cap \big(\bigcup_{j \in J} B_{2r_j}(z_j)\big) = \emptyset$.
Combining the above inequalities, we have
\begin{equation*}
    \sum_{j \in J} \frac{\mathcal{H}^{n}(B_{2r_j}(z_j))}{r_j} \leq \frac{\veps}{8}.
\end{equation*}
Consider the collection of open balls 
\begin{equation*}
    \mathcal{B} := \{B \subseteq B_{15}(\tilde{p}) \,:\, B = B_{r_j}(\tilde{z}) \text{ for some }  \tilde{z} \in \pi^{-1}(z_j)\}.
\end{equation*}
By Proposition~\ref{Pro: limit projection involution properties}, $\mathcal{B}$ is finite (with at most two balls corresponding to each $j$) and covers $\tilde{C}$ . Moreover, indexing $\mathcal{B}$ arbitrarily so that $\mathcal{B} = \{B_{r_k}(\tilde{z}_k)\}_{k \in K}$, we have 
\begin{equation*}
    \sum_{k \in K} \frac{\tilde\mm(B_{2r_k}(\tilde{z}_k))}{r_k} \leq \veps
\end{equation*}
by Lemma~\ref{Lem: limit measure bound}. By Proposition~\ref{Pro: limit projection involution properties} again, we have $B_{\eta}(x) \cap \big(\bigcup_{k \in K} B_{2r_k}(\tilde{z}_k)\big) = \emptyset$.

Let $\nu$ be an optimal dynamical plan from $\mu_1 := \tilde\mm\mrestr {B_{\eta}(y)}/\tilde\mm(B_\eta(y))$ to $\mu_0 := \delta_x$ (note that $\tilde\mm(B_\eta(y)) \neq 0$ by the lower bound from Lemma~\ref{Lem: limit measure bound}). Clearly, $\nu$ is concentrated on the set of geodesics with length less than $10$ whose image is contained in $B_{10}(\tilde{p})$. Since $\tilde{X}$ is locally $\cden$ on $B_{100}(\tilde{p})$ by our assumptions at the beginning of the subsection, we may assume that $\nu$ satisfies the conclusions of Proposition~\ref{Pro: local cde density control}. In particular, this means that for $t \geq \eta/10$ 
\begin{equation*}
    (e_t)_*(\nu) \leq D(K, n, \eta)\tilde\mm.
\end{equation*}

Fix $k \in K$. Let $\gamma$ be a geodesic from $x$ to some $y' \in B_\eta(y)$ and assume that it intersects $B_{r_k}(\tilde{z}_k)$. Since $\gamma$ has length less than $10$, if we choose any $N \in \N$ so that $N \geq 10/r_k$, there exists some $i = 1,\dots, N$ so that $\gamma(i/N) \in B_{2r_k}(\tilde{z}_k)$. Moreover, since $B_\eta(x) \cap B_{2r_k}(\tilde{z}_k) = \emptyset$, we have $i/N \geq \eta/10=:\eta'$. Note also that we can require $N \leq 20/r_k$. Therefore, denoting by $G$ the set of all geodesics whose image is inside of $B_{10}(\tilde{p})$, we have
\begin{align*}
\nu\Big(\{\gamma \in G\,:\,\gamma(t) \in B_{r_k}(\tilde{z}_k) \text{ for some $t \in [0,1]$}\}\Big)
&\leq \sum_{i\,:\, i/N \geq \eta'}  \nu\Big(\{\gamma \in G \,:\, \gamma(i/N) \in B_{2r_k}(\tilde{z}_k)\}\Big)\\
&\leq \sum_{i\,:\, i/N \geq \eta'} (e_{i/N})_*(\nu)(B_{2r_k}(\tilde{z}_k))\\
&\leq \sum_{i\,:\, i/N \geq \eta'} D(K,n, \eta)\tilde\mm(B_{2r_k}(\tilde{z}_k))\\
&\leq DN\tilde\mm(B_{2r_k}(\tilde{z}_k))\\
&\leq 20D\frac{\tilde\mm(B_{2r_k}(\tilde{z}_k))}{r_k}.
\end{align*}
Summing across $k$ and using the bounds from earlier, we obtain
\begin{align*}
& \nu\Big(\{\gamma \in G \,:\, \gamma(t) \in \tilde{C} \text{ for some $t \in [0,1]$}\}\Big)\\
&\leq \nu\Big(\Big\{\gamma \in G\,:\,\gamma(t) \in \bigcup_{k \in K} B_{r_k}(\tilde{z_k}) \text{ for some $t \in [0,1]$}\Big\}\Big)\\
&\leq \sum_{k \in K} D\frac{\tilde\mm(B_{2r_k}(\tilde{z}_k))}{r_k}\\
&\leq D\veps.
\end{align*}
Since $\veps > 0$ is arbitrary and $D$ is independent of $\veps$, we conclude that 
\begin{equation*}
\nu\Big(\{\gamma \in G\,:\,\gamma(t) \in \tilde{C} \text{ for some $t \in [0,1]$}\}\Big) = 0,
\end{equation*}
which shows that for $\mm$-a.e.\ $y' \in B_\eta(y)$ there exists a geodesic connecting $x$ to $y'$ that does not intersect $\tilde{C}$. 
\end{proof}

We state the main result of the subsection. 
\begin{Thm}\label{Thm: stability of non-orientiable local cde case}
Let $(X_i, \dd_i, \HH^n, p_i)_{i \in \N}$ be a sequence of non-orientable $\RCD(K,n)$ spaces without boundary converging to $(X, \dd, \mathcal{H}^n, p)$ in the pmGH sense. For each $i \in \N$, let $\pi_i: (\hat{X}_i, \hat\dd_i, \HH^n, \hat{p}_i) \to (X_i, \dd_i, \HH^n, p_i)$ be the ramified double cover. Assume that 
\begin{enumerate}
    \item $B_{1}(p_i)$ is non-orientable for each $i \in \N$;
    \item $\hat{X}_i$ is locally $\cden$ on $B_{1000}(\hat{p}_i)$ for each $i \in \N$.
\end{enumerate}
Then $B_{1+\delta}(p)$ is non-orientable for any $\delta > 0$. 
\end{Thm}

\begin{proof}
We use the same notation for the various objects defined earlier in this subsection. For $\veps < \veps(n)$ sufficiently small to be fixed later and some $\delta < 1/10$, fix some $x \in B_{\delta}(p) \cap A_{\veps}(X)$. By Lemma~\ref{lem: regular set double cover}, there are two distinct lifts $\tilde{x}, \Gamma(\tilde{x}) \in B_{\delta}(\tilde{p}) \cup B_{\delta}(\Gamma(\tilde{p}))$. Since $\tilde{\dd}(\tilde{p}, \Gamma(\tilde{p})) \leq 2$, we see that $\tilde{\dd}(\tilde{x}, \Gamma(\tilde{x})) < 2+2\delta$ by the triangle inequality and $B_{\delta}(\tilde{p}) \cup B_{\delta}(\Gamma(\tilde{p})) \subseteq B_{5}(\tilde{p})$.

Taking $C := \overline{B_{10}(p)} \setminus A_\veps(X)$ in Lemma~\ref{Lem: double cover limit connetedness}, we can find $y \in B_{\delta}(\tilde{x}) \setminus \pi^{-1}(C)$ so that $y$ can be connected to $\tilde{x}$ and $\Gamma(\tilde{x})$ by geodesics whose images are contained in $B_{10}(\tilde{p}) \setminus \pi^{-1}(C)$ (here we used the fact that $\tilde\mm(B_{\delta}(\tilde{x}))$ is bounded away from $0$ by Lemma~\ref{Lem: limit measure bound}). Concatenating the two geodesics, we obtain a Lipschitz curve $\tilde\gamma:[0,2] \to B_{10}(\tilde{p})$ with the $\tilde\gamma(0) = \tilde{x}$ and $\tilde\gamma(2) = \Gamma(\tilde{x})$, and whose image is contained in $B_{10}(\tilde{p}) \setminus \pi^{-1}(C)$. Moreover, since the length of $\gamma$ is $<2+4\delta$, we have
$$\tilde\gamma([0,2]) \subseteq B_{1+2\delta}(\tilde{x}) \cup B_{1+2\delta}(\Gamma(\tilde{x})) \subseteq B_{1+3\delta}(\tilde{p}) \cup B_{1+3\delta}(\Gamma(\tilde{p})).$$

Consider the loop $\gamma :=\pi \circ \tilde\gamma:[0,2] \to B_{10}(p)$ with $\gamma(0) = \gamma(2) = x$. By Lemma~\ref{lem: regular set double cover} and Proposition~\ref{Pro: limit projection involution properties}, $\gamma$ is Lipschitz and $\gamma([0,2]) \subseteq B_{1+3\delta}(p) \cap A_\veps(X)$. It suffices now to prove that $\gamma$ is orientation-reversing.

Since this part of the proof is similar to that of \cite[Theorem 4.2]{BBP24}, we only sketch it (see also the proof of Claim~\ref{Cla: cone optimal plan support}, where we use a similar argument). As in the proof of \cite[Theorem 4.2]{BBP24} , we can construct a sequence of curves $\hat\gamma_i:[0,2] \to B_{2}(p_i)$ with $\hat\gamma_i(0)=\hat{x}_i$, $\hat\gamma_i(2)=\Gamma_i(\hat{x}_i)$, and which uniformly converges to $\tilde\gamma$ under the pmGH convergence. The uniform convergence implies that, for sufficiently large $i$, the image of $\gamma_i := \pi_i \circ \hat\gamma_i$ is in $A_{\veps'}(X_i)$, where $\veps'$ can be made arbitrarily small by choosing $\veps$ small. Combining this with the fact that $\hat\gamma_i$ connects $\hat{x}_i$ to $\Gamma_i(\hat{x}_i)$, we see that $\gamma_i$ is an orientation-reversing loop (see Definition~\ref{Def: orientation-reversing loop}). Indeed, Theorem~\ref{Thm: Ramified double cover existence} gives that $\pi_i^{-1}(A_{\veps'}(X_i))$ is homeomorphic to the orientable double cover of $A_{\veps'}(X_i)$ (as a non-orientable manifold) for sufficiently small $\veps' < \veps(n)$. Therefore, the projection of any continuous curve from $\hat{x}_i$ to $\Gamma_i(\hat{x}_i)$ whose image lies in $\pi_i^{-1}(A_{\veps'}(X_i))$ must be an orientation-reversing loop. Now, as in \cite[Theorem 4.2]{BBP24} (see also Proof II of \cite[Theorem 4.2]{BBP24}), we can conclude that $\pi \circ \gamma$ must also be an orientation-reversing loop, which implies that $B_{1+3\delta}(p)$ is non-orientable by Proposition~\ref{Pro: non-orientable loops equivalence}.
\end{proof}

\subsection{1D-localization}
In this subsection, we review the constructions of 1D-localization on $\RCD$ spaces. Our exposition will follow those of \cite{BC13, C14, CM17a, CM17b, CM18}. 

Let $(R, \mathscr{R})$ be a measurable space and let $\mathfrak{Q}: R \to Q$ be a function where $Q$ is a general set. We may endow $Q$ with the pushfoward $\sigma$-algebra $\mathscr{Q}$ of $\mathscr{R}$ defined by the rule
\begin{equation*}
S \in \mathscr{Q} \text{ if and only if } \mathfrak{Q}^{-1}(S) \in \mathscr{R}.
\end{equation*}
Any probability measure $\rho$ on $(R, \mathscr{R})$ then induces a natural probability measure $\mathfrak{q} := \mathfrak{Q}_{*}(\rho)$ on $(Q, \mathscr{Q})$ via the pushfoward. 

\begin{Def}\label{Def: Disintegration}
A \textit{disintegration} of the probability measure $\rho$ consistent with $\mathfrak{Q}$ is a map $\rho: \mathscr{R} \times Q \to [0,1]$ so that, setting $\rho_{q}(B) := \rho(B,q)$, the following hold:
\begin{enumerate}
    \item $\rho_{q}(\cdot)$ is a probability measure on $(R, \mathscr{R})$ for all $q \in Q$,
    \item $q\mapsto\rho_{q}(B)$ is $\mathfrak{q}$-measurable for all $B \in \mathscr{R}$,
\end{enumerate}
and the map satisfies for all $B \in \mathscr{R}, S \in \mathscr{Q}$ the consistency condition 
\begin{equation*}
    \rho(B \cap \mathfrak{Q}^{-1}(S)) = \int_{S} \rho_q(B)\,d\mathfrak{q}(q).
\end{equation*}
A disintegration is \textit{strongly consistent with respect to $\mathfrak{Q}$} if for all $q$ we have $\rho_q(\mathfrak{Q}^{-1}(q))=1$. The measures $\rho_q$ are called \textit{conditional probabilities}. 
\end{Def}

In what follows, we assume that $(X, \dd)$ is a complete and locally compact geodesic metric space. In particular, $X$ is proper. Let $\phi: X \to \R$ be any $1$-Lipschitz function. The \textit{$\dd$-cyclically monotone set associated with $\phi$} is defined by
\begin{equation}\label{Eq: Gamma definiiton}
    \Gamma := \{(x,y) \in X \times X\,:\,\phi(x) - \phi(y) = \dd(x,y)\}.
\end{equation}
In general, a subset $S \subseteq X \times X$ is said to be \textit{$\dd$-cyclically monotone} if, for any finite set of couples $(x_1, y_1),\ldots, (x_N, y_N) \in S$, we have
\begin{equation*}
    \sum_{i=1}^{N} \dd(x_i, y_i) \leq \sum_{i=1}^N \dd(x_i, y_{i+1}), 
\end{equation*}
where $y_{N+1} := y_1$. It is straightforward to check that the $\dd$-cyclically monotone set associated with $\phi$ is $\dd$-cyclically monotone in the above sense. 

The following lemma is immediate for any $\dd$-cyclically monotone set $\Gamma$ associated with $\phi$.

\begin{Lem}
    Let $(x,y) \in \Gamma$ and let $\gamma \in \Geo(X)$ be such that $\gamma_0 = x$ and $\gamma_1 = y$. Then 
    \begin{equation*}
        (\gamma_s, \gamma_t) \in \Gamma
    \end{equation*}
    for any $0 \leq s \leq t \leq 1$.
\end{Lem}

\begin{Def}\label{Def: transport ray def}
    We define the set of \textit{transport rays} of $\phi$ by
    \begin{equation*}
        R : = \Gamma \cup \Gamma^{-1},
    \end{equation*}
    where $\Gamma^{-1} := \{(x,y) \in X \times X \,:\, (y,x) \in \Gamma\}$. The sets of \textit{initial points} and \textit{final points} are defined by
    \begin{align*}
        &\mathfrak{a} := \{z \in X \,:\, \nexists x \in X \text{ so that } (x, z) \in \Gamma \text{ and } \dd(x, z) > 0\},\\
        &\mathfrak{b} := \{z \in X \,:\, \nexists x \in X \text{ so that } (z, x) \in \Gamma \text{ and } \dd(x, z) > 0\}. 
    \end{align*}
    The set of \textit{end points} is $\mathfrak{a} \cup \mathfrak{b}$. We define the \textit{transport set with end points}:
    \begin{equation*}
        \Tau_e := P_1(\Gamma \backslash \{x=y\}) \cup P_1(\Gamma^{-1} \backslash \{x=y\}),
    \end{equation*}
    where $P_1: X \times X \to X$ is the projection onto the first factor and $\{x=y\}$ is the diagonal subset of $X \times X$. 
\end{Def}

By \cite[Remark 3.3]{CM17b}, we know that $\Gamma, \Gamma^{-1}, R$ are closed, that $\Gamma, \Gamma^{-1}, R, \Tau_e$ are $\sigma$-compact sets (countable union of compact sets), and that $\mathfrak{a}, \mathfrak{b}$ are Borel sets. 

In order to apply disintegration, we would like to have a subset of $\Tau_e$ on which the transport rays induce an equivalence relation. In order to achieve this, it suffices to remove from $\Tau_e$ the branching points of the transport rays. We set $\Gamma(x) := P_2(\Gamma \cap (\{x\} \times X))$ and $\Gamma(x)^{-1} := P_2(\Gamma^{-1} \cap (\{x\} \times X))$. These are the set of points that come before and after $x$ on some transport ray respectively. We then define the \textit{forward} and \textit{backward branching sets}: 
\begin{align*}
&A_+ := \{x \in \Tau_e \,:\, \exists z, w \in \Gamma(x) \text{ so that } (z,w) \notin R\},\\
&A_- := \{x \in \Tau_e \,:\, \exists z, w \in \Gamma^{-1}(x) \text{ so that } (z,w) \notin R\}, 
\end{align*}
and the \textit{transport set}
\begin{equation*}
\Tau := \Tau_e \setminus (A_+ \cup A_-). 
\end{equation*}
\begin{Thm}(\cite[Theorem 5.5]{C14}).
Let $(X, \dd, \mm)$ satisfy $\RCD(K,N)$ with $1 \leq N < \infty$. Then the set of transport rays $R \subseteq X \times X$ (restricted to $\Tau \times \Tau$) is an equivalence relation on the transport set $\Tau$ and
\begin{equation*}\label{Eq: set different Tau and Tau_e}
    \mm(\Tau_\e \setminus \Tau) = 0.
\end{equation*}
Moreover, the transport set $\Tau$ is a $\sigma$-compact set. 
\end{Thm}
Denoting the set of equivalence classes of $\Tau$ with respect to the relation $R$ by $Q$, we have a quotient map $\mathfrak{Q}: \Tau \to Q$. We say that a map $F: \Tau \to \Tau$ is a \textit{section of the equivalence relation $R$} if we have the following:
\begin{enumerate}
    \item for any $x \in \Tau$, $(x, F(x)) \in R$;
    \item for any $x, y \in \Tau$, if $(x, y) \in R$, then $F(x) = F(y)$.
\end{enumerate}
A set is said to be $\mm$-measurable if it is contained in the completion of the Borel $\sigma$-algebra with respect to $\mm$. 
\begin{Pro}(\cite[Proposition 5.2]{C14}).
There exists an $\mm$-measurable section
\begin{equation*}
    \mathfrak{Q}: \Tau \to \Tau
\end{equation*}
for the equivalence relation $R$, i.e., the preimage of any open subset of $\Tau$ under $\mathfrak{Q}$ is $\mm$-measurable. 
\end{Pro}
There is a natural identification between $\mathfrak Q(\Tau)$ and the equivalence classes of $R$, so we identify $Q = \mathfrak{Q}(\Tau) = \{x \in \Tau \,:\, x = \mathfrak{Q}(x)\}$ and endow it with the subset topology. The pushforward measure $\mathfrak{q} := \mathfrak{Q}_{*}(\mm)$ on $Q$ is Borel since $\mathfrak{Q}$ is $\mm$-measurable. In the sequel, we assume that $\mm(X)<\infty$. Viewing $\mathfrak{Q}$ as a map $\Tau\to X$ and recalling that $X$ is Polish, we may then apply the standard disintegration theorem \cite[Corollary 452P]{F03} to find a strongly consistent disintegration $\{\mm_q\}_{q\in Q}$, where $\mm_q$ is a probability measure concentrated on $\mathfrak{Q}^{-1}(q)\subseteq\Tau$.

We have a ray map associated with the disintegration as in \cite[Definition 3.6]{CM17b} which we denote by $g: \text{Dom}(g) \subseteq Q \times \RR \to \Tau$. The ray map associates with each $(q,t) \in \text{Dom}(g)$ the unique element $x \in \Tau$ so that $(q,x) \in \Gamma$ at distance $t$ from $q$ if $t$ is positive or the unique element $x \in \Tau$ so that $(x,q) \in \Gamma$ at distance $-t$ from $q$ if $t$ is negative, provided that such an element exists. 

\begin{Pro}\label{Pro: Ray map properties}
The following hold:
\begin{enumerate}
    \item $g$ is continuous and $\text{Dom}(g)$ is analytic;
    \item $t \mapsto g(q,t)$ is an isometry and if $s, t \in \text{Dom}(g(q, \cdot))$ with $s \leq t$ then $g(q,s), g(q,t) \in \Gamma$;
    \item $(q,t) \mapsto g(q,t)$ is bijective from $\text{Dom}(g)$ to $\Tau$, and its inverse is 
    \begin{equation*}
        x \mapsto g^{-1}(x) = (\mathfrak{Q}(x), \pm\dd(x, \mathfrak{Q}(x))),
    \end{equation*}
    where $\mathfrak{Q}$ is the quotient map previously introduced and the positive or negative sign depends on $(\mathfrak{Q}(x), x) \in \Gamma$ or $(x, \mathfrak{Q}(x)) \in \Gamma$;
    \item $g^{-1}$ is $\mm$-measurable, i.e., the image of any open set in $Q \times \RR$ through $g$ is $\mm$-measurable. 
\end{enumerate}
\end{Pro}
The first three properties were stated in \cite[Proposition 5.4]{C14}, with the continuity of $g$ following from its definition, properness of $X$, and closedness of $\Gamma$. The last one follows 
from the fact that $g$ maps (relatively) closed subsets of $\text{Dom}(g)$, which are then analytic,
to analytic subsets of $\Tau$, as $g$ is continuous.

The following theorem summarizes the results presented so far and states the main disintegration theorem for $\RCD$ spaces \cite[Theorem 3.8]{CM17b}. 
\begin{Thm}\label{Thm: Main disintegration thm}
Let $(X, \dd, \mm)$ be $\RCD(K,N)$ for some $k, N \in \RR$ with $1 \leq N \leq  \infty$, with $\mm(X) < \infty$. Let $\phi: X \to \RR$ be a $1$-Lipschitz function. Then, in the notation already presented, the following disintegration formula holds:
\begin{equation*}
    \mm \mrestr {\Tau} = \int_{Q} \mm_q \,d\mathfrak{q}(q), \quad \mm_q(Q^{-1}(q))=1 \text{ for $\mathfrak{q}$-a.e.\ $q \in Q$}.
\end{equation*}
Finally, for $\mathfrak{q}$-a.e.\ $q \in Q$ the conditional measure $\mm_q$ is absolutely continuous with respect to $\HH^1 \mrestr {\{g(q,t) \,:\, t\in \text{Dom}(g(q, \cdot))\}} = g(q, \cdot)_{*}(\LL^1)$.
\end{Thm}
For proof of the last part of the Theorem, see \cite[Theorem 6.6]{C14}.
The curvature lower bound can be localized to the needles of the disintegration. For each $q \in Q$, denote by $X_q$ the closure of $\mathfrak{Q}^{-1}(q)$. We have the following theorem \cite{CM17b}.
\begin{Thm}\label{Thm: needle CD property}
If $(X, \dd, \mm)$ is essentially non-branching and $\CD(K,N)$, satisfying $\mm(X) < \infty$, then for $\mathfrak{q}$-a.e.\ $q \in Q$ the space $(X_q, \dd, \mm_q)$ is also $\CD(K,N)$, with $\supp(\mm_q) = X_q$.
\end{Thm}
Here by abuse of notation $\dd$ is the restriction of the distance function $\dd$ on $X$ to $X_q$ and $\mm_q$ is the pushforward of the measure $\mm_q$ on $\mathfrak{Q}^{-1}(q)$ via the inclusion map. In particular, $(X_q, \dd, \mm_q)$ verifies the $\MCP(K,N)$ property proposed by \cite{S06b, O07}. 

Following \cite[Section 2.3]{CM18}, given $K \in \R$ and $N \in (1, \infty)$, we say that a nonnegative Borel function $h$ defined on a closed interval $I \subseteq \R$ is an \textit{$\MCP(K,N)$ density} if for all $x_0, x_1 \in I$ and $t \in [0,1]$:
\begin{equation}\label{Eq: MCP(K,N) density def}
    h(tx_1 + (1-t)x_0) \geq \sigma_{K,N-1}^{(1-t)}(|x_1-x_0|)^{N-1}h(x_0),
\end{equation}
where $\sigma^t_{K,N}(\cdot)$ are the distortion coefficients defined in \eqref{Eq: distortion coefficients}. It is known that $(I, \dd_{\R}, h\LL^1)$ is $\MCP(K,N)$ if and only if $h$ is an $\MCP(K,N)$ density.  

\subsection{Disintegration with respect to the distance to a point}
In this subsection we will take our underlying space to be an $\RCD(K,N)$ space $(X, \dd, \mm)$ and assume that the $1$-Lipschitz function under consideration for the disintegration theorem from the previous subsection is $\dd_p(\cdot) := \dd(p, \cdot)$, where $p$ is a fixed point. In addition, we assume that $\mm(X) < \infty$ for ease of exposition; the assumption is so that we can apply Theorem~\ref{Thm: Main disintegration thm} directly as stated and is not necessary. Indeed, the argument goes through in the $\mm(X) = \infty$ case either by considering disintegration on large balls or by using \cite[Theorem 3.10]{L24} (see the discussion below it).

We will also use the following notation: given a measure space $(Y, \Sigma, \mu)$ and a nonnegative function $f \in L^1(Y, \mu)$, we write $\mu^f$ to indicate the measure which takes
\begin{equation*}
    \Sigma \ni S \mapsto \mu^f(S) := \int_{S} f \, d\mu.
\end{equation*}
Sometimes, we will also write $f \mu$ to indicate $\mu^f$. 

Since $X$ is a geodesic space, it is easy to see that $\Tau_e = X$ (as long as $X$ is not a single point). Therefore, by \eqref{Eq: set different Tau and Tau_e}, we have  
\begin{equation*}
    \mm(X \setminus \Tau) = 0,
\end{equation*}
and so in particular Theorem~\ref{Thm: Main disintegration thm} gives that
\begin{equation}\label{Eq: m disintegration formula}
    \mm = \int_{Q} \mm_q \, d\mathfrak{q}(q),
\end{equation}
The equality holds in that sense that, for any Borel set $B \subseteq X$, the function $Q \ni q \mapsto \mm_q(B)$ is integrable with respect to $\mathfrak{q}$ and $\int_Q \mm_q(B) \, d\mathfrak{q}(q) = \mm(B)$. We note that technically the collection $\{\mm_q\}$ in Theorem~\ref{Thm: Main disintegration thm} are probability measures on $\Tau$, which for simplicity of exposition we identify with a Borel probability measure on $X$ via pushforward through the inclusion map. Moreover, we know that, for $\mathfrak{q}$-a.e.\ $q \in Q$, $\mm_q$ is absolutely continuous with respect to $g(q, \cdot)_{*}(\LL^1)$ and $t \mapsto g(q,t)$ is an isometry onto its image. Therefore, denoting $I_q := \text{Dom}(g(q,\cdot))$, for any such $q$ we may find $m_q: I_q \to \R$ that is positive and integrable with respect to the Lebesgue measure $\LL^1$ on $I_q$, so that $g(q, \cdot)_{*}((\LL^1)^{m_q})=\mm_q$.

\begin{Lem}\label{Lem: another probability disintegration}
Let $\mu \ll \mm$ be a Borel probability measure and let $u$ be a nonnegative Borel density such that $\mu = \mm^u$. Then the following hold: 
\begin{enumerate}
    \item For $\mathfrak{q}$-a.e.\ $q \in Q$, $\mm_q^{u}$ is a nonnegative finite Borel measure on $X$.
    \item For any Borel set $B \in X$, the function $q\mapsto\mm_{q}^u(B)$ is $\mathfrak{q}$-measurable and 
    \begin{equation*}
        \mu(B) = \int_{Q} \mm_q^u(B)d\mathfrak{q}(q).
    \end{equation*}
    \item If $\{\nu_q\}_{q\in Q}$ is another collection of nonnegative finite Borel measures on $X$ such that
        \begin{enumerate}
            \item[(i)] $\nu_q$ is concentrated on $\mathfrak{Q}^{-1}(q)$ for $\mathfrak{q}$-a.e.\ $q \in Q$,
            \item[(ii)] (2) holds with $\{\nu_q\}$ in place of $\{\mm_q^u\}$,
        \end{enumerate}
    then for $\mathfrak{q}$-a.e.\ $q \in Q$ we have $\nu_q = \mm_q^u$ (on all Borel subsets of $X$). 
\end{enumerate}
\end{Lem}

\begin{proof}
Let $u_i: X \to \R$ be a sequence of nonnegative simple functions (i.e., finite linear combinations of characteristic functions of Borel sets) such that $(u_i)_{i\in \N}$ is monotone increasing and converges to $u$. It is clear from the definition of disintegration (Definition~\ref{Def: Disintegration}) that 
\begin{equation*}
    \int_{X} u_i(x) \, d\mm(x) = \int_{q \in Q} \bigg(\int_{X} u_i(x) \, d\mm_q(x)\bigg) \, d\mathfrak{q}(q),
\end{equation*}
and so the bounded monotone convergence theorem gives 
\begin{equation*}
    \int_{X} u(x) \, d\mm(x) = \int_{q \in Q} \bigg(\int_{X} u(x) \, d\mm_q(x)\bigg) \, d\mathfrak{q}(q), 
\end{equation*}
where the left-hand side integrates to $1$ since $\mu$ is a probability measure. This implies that, for $q$-a.e.\ $q \in Q$, $u$ is summable with respect to $\mm_q$ and so $\mm_q^u$ is a finite nonnegative Borel measure, proving the first claim. 

The second claim follows similarly by approximation. Fix any Borel set $B \in X$. From the definition of disintegration we have that the functions $v_i: Q \to \R$ defined by $v_i(q) := \int_{B} u_i(x) \, d\mm_q(x)$ are $\mathfrak{q}$-measurable. We see that this sequence $(v_i)_{i\in\N}$ is monotone increasing and by the bounded monotone convergence theorem converges to $v:Q \to \R$ defined by $v(q) := \int_{B} u(x) \, d\mm_q(x)$ for $\mathfrak{q}$-a.e.\ $q \in Q$. This implies that $v$ is $\mathfrak{q}$-measurable since all of the functions $v_i$ are. The equality in (2) then follows by applying the bounded monotone convergence theorem to the sequence $(v_i)_{i\in\N}$. 

For the third claim, notice that the Borel $\sigma$-algebra of $X$ is countably generated by some collection $\{B_i\}_{i \in \N}$. Using the countable additivity of measures, it then suffices to show that, for any fixed $i \in \N$, $\nu_q(B_i) = \mm_q^u(B_i)$ for $\mathfrak{q}$-a.e.\ $q \in Q$. Given $i \in \N$, for any $\mathfrak{q}$-measurable subset $R \subseteq Q$, we have that $\mathfrak{Q}^{-1}(R) \cap B_i$ is $\mm$-measurable. By definition, this means we can find a Borel set $M$ so that $M \subseteq \mathfrak{Q}^{-1}(R) \cap B_i$ and $(\mathfrak{Q}^{-1}(R)\cap B_i) \setminus M \subseteq N$, where $N$ is Borel and $\mm(N) = 0$. Using the fact that, for $\mathfrak{q}$-a.e.\ $q \in Q$, $\mu_q^u$ is concentrated on $\mathfrak{Q}^{-1}(q)$, we obtain
\begin{align*}
    \mm^u(M) &= \int_{Q} \mm_q^u(M) \, d\mathfrak{q}(q)\\
    &= \int_{R} \mm_q^u(M) \, d\mathfrak{q}(q)\\
    &\leq \int_{R} \mm_q^u(B_i) \, d\mathfrak{q}(q)\\
    &\leq \int_{Q} \mm_q^u(M \cup N) \, d\mathfrak{q}(q)\\ 
    &= \mm^u(M \cup N),
\end{align*}
and so the inequalities must be equalities because $\mm^u(M) = \mm^u(M \cup N)$. A similar computation holds for $\nu_q$. In particular, we have
\begin{equation*}
    \int_{R}\mm_q^u(B_i) \, d\mathfrak{q}(q) = \int_{R}\nu_q(B_i) \, d\mathfrak{q}(q).
\end{equation*}
In other words, the functions $Q \ni q \mapsto \mm_q^u(B_i)$ and $Q \ni q \mapsto \nu_q(B_i)$ are Borel functions whose integrals with respect to $\mathfrak{q}$ agree on all measurable subsets $R$ of $Q$. This clearly implies that $\mm_q^u(B_i)=\nu_q(B_i)$ for $\mathfrak{q}$-a.e.\ $q \in Q$ as desired. 
\end{proof}

Since, for $\mathfrak{q}$-a.e.\ $q \in Q$, $\mm_q$ is absolutely continuous with respect to $g(q, \cdot)_{*}(\LL^1)$, we have that $\mm_q^u$ is also absolutely continuous with respect to $g(q, \cdot)_{*}(\LL^1)$. Moreover, it is clear that $\mm_q^u = g(q, \cdot)_*((\LL^1)^{u(q,\cdot)m_q})$, where $u(q,t):=u(g(q,t))$
and $m_q$ is such that $\mm_q=m_q\mathcal{L}^1$.

We now consider an optimal plan between two probability measures that is concentrated in $\Gamma$ (see \eqref{Eq: Gamma definiiton}). More precisely, we assume that we have a $\mu_1, \mu_2 \in \mathcal{P}_2(X)$ so that the following hold:
\begin{enumerate}
    \item $\mu_1 := u_1\mm, \mu_2 := u_2\mm \ll \mm$;
    \item $\gamma \in \Opt(\mu_1, \mu_2)$ is such that $\gamma(\Gamma)=1$.
\end{enumerate}

From \cite{G12} it is known that, on non-branching $\CD(K,N)$ spaces, $\gamma$ is unique and is induced by a map, i.e., there exists a Borel map $T:X \to X$ so that $\gamma = (\text{Id}, T)_{*}(\mu_1)$. Since $\RCD(K,N)$ spaces were shown to be non-branching in \cite{D25}, this result applies in our setting. 

Let $\{\mm_q^{u_1}\}_{q \in Q}$ and $\{\mm_q^{u_2}\}_{q \in Q}$ be disintegrations for $\mu_1$ and $\mu_2$ as in Lemma~\ref{Lem: another probability disintegration}. 
\begin{Lem}\label{Lem: needle transport}
For $\mathfrak{q}$-a.e.\ $q \in Q$, $T_{*}(\mm_q^{u_1})=\mm_q^{u_2}$.
\end{Lem} 
\begin{proof}
The claim will follow by using the third part of Lemma~\ref{Lem: another probability disintegration} as soon as we check that $\{T_{*}(\mm_q^{u_1})\}_{q \in Q}$ is a collection of measures satisfying the conditions of the lemma for $\mu_2$. It is enough to show that, for $\mathfrak{q}$-a.e.\ $q \in Q$, $T_{*}(\mm_q^{u_1})$ is concentrated on $Q^{-1}(q)$, as the other conditions are obvious. 

By our assumption that $\gamma(\Gamma)=1$, we have $(x, T(x)) \in \Gamma$ for $\mu_1$-a.e.\ $x \in X$. Moreover, since $\mm(X \setminus \Tau) =0$ and $\mu_1, \mu_2 \ll \mm$, we have $(x, T(x)) \in \Gamma \cap (\Tau \times \Tau)$ for $\mu_1$-a.e.\ $x \in X$, which implies that $T(x) \in \mathfrak{Q}^{-1}(\mathfrak{Q}(x))$. As such, there exists a Borel set $A \subseteq X$ so that $\mu_1(A)=1$ and $T(x) \in \mathfrak{Q}^{-1}(\mathfrak{Q}(x))$ for any $x \in A$.

We have
\begin{equation*}
  \int_{q \in Q} \mm_{q}^{u_1}(A) \, d\mathfrak{q}(q) = \mm_q^{u_1}(A) = 1 =  \mm_q^{u_1}(X)=\int_{q \in Q} \mm_{q}^{u_1}(X) \, d\mathfrak{q}(q).
\end{equation*}
In particular, for $\mathfrak{q}$-a.e.\ $q \in Q$, we must have $\mm_q^{u_1}(A)=\mm_q^{u_1}(X)$, i.e., $T(x) \in \mathfrak{Q}^{-1}(\mathfrak{Q}(x))$ for $\mm_q^{u_1}$-a.e.\ $x \in X$. Since for $\mm_q^{u_1}$-a.e.\ $x \in X$ we also have that $\mathfrak{Q}(x)=q$, we conclude the proof. 
\end{proof}

For $q \in Q$ so that $\mm_q^{u_1}(X), \mm_q^{u_2}(X) \neq 0$, let $(\mu_1)_q$ and $(\mu_2)_q$ denote the normalizations of $\mm_q^{u_1}$ and $\mm_q^{u_2}$ respectively.  We have the following lemma. 
\begin{Lem}\label{Lem: needle optimal transport}
For $\mathfrak{q}$-a.e.\ $q \in Q$, either $\mm_q^{u_1}(X) = 0 = \mm_q^{u_2}(X)$ or the optimal transport from $(\mu_1)_q$ to $(\mu_2)_q$ is induced by the map $T$. 
\end{Lem}

Before giving a proof of the lemma, we review some basic but necessary tools from the theory of optimal transport following \cite{AG13}. Let $c: X \times X \to \R$ be the cost function defined by $c(x,y):=\dd(x,y)^2$. Recall that given any function $\psi: X \to \R \cup \{\pm \infty\}$, its \textit{$c$-transform} $\psi^c: X \to \R \cup \{-\infty\}$ is defined by
\begin{equation}\label{Eq: c-transform def}
\psi^c(x) := \inf_{y \in X} (c(x,y) - \psi(y)).
\end{equation}
We say that a function $\varphi: X \to \R \cup \{-\infty\}$ is \textit{$c$-concave} if there exists $\psi: X \to \R \cup \{-\infty\}$ so that $\varphi = \psi^c$. For a $c$-concave function $\varphi$, we define its \textit{$c$-superdifferential} to be 
\begin{equation}\label{Eq: c-superdifferential}
\partial^c \varphi := \Big\{(x,y) \in X \times X \,:\, \varphi(x) + \varphi^c(y) = c(x,y) \Big\}. 
\end{equation}
It follows easily from the definitions that $c$-concave functions are upper semi-continuous and $\varphi(x) + \varphi^c(y) \leq c(x,y)$ for any $x,y \in X$. These imply that $\partial^c \varphi$ is a closed set. 

We have the following fundamental theorem \cite[Theorem 1.13]{AG13}, formulated in a slightly different way for our specific case.  
\begin{Thm}[Fundamental theorem of optimal transport]\label{Thm: fundamental theorem of optimal transport]} 
Let $\mu, \nu \in \mathcal{P}_2(X)$ and let $\gamma \in \Adm(\mu, \nu)$. Then $\gamma$ is optimal if and only if there exists a $c$-concave function $\varphi$ so that $\max\{\varphi, 0\} \in L^1(\mu)$ and $\gamma(\partial^{c}\varphi)=1$.
\end{Thm}

\begin{proof}[Proof of Lemma~\ref{Lem: needle optimal transport}]
Applying Theorem~\ref{Thm: fundamental theorem of optimal transport]} to the optimal plan $(\text{Id}, T)_{*}(\mu_1)$, we find a $c$-concave function $\varphi: X \to \R\cup\{-\infty\}$ so that $(\text{Id}, T)_{*}(\mu_1)$ is concentrated on $\partial^{c}\varphi$. It now suffices to show that, for $\mathfrak{q}$-a.e.\ $q \in Q$, $(\text{Id}, T)_{*}(\mm_q^{u_1})$ is concentrated on $\partial^{c}\varphi$ to conclude. Indeed, combined with Lemma~\ref{Lem: needle transport}, this would imply that for $\mathfrak{q}$-a.e.\ $q \in Q$, either
\begin{enumerate}
    \item $\mm_q^{u_1}(X) = \mm_q^{u_2}(X) = 0$, or
    \item $\mm_q^{u_1}(X) = \mm_q^{u_2}(X) \neq 0$ and $(\text{Id}, T)_{*}(\mm_q^{u_1})$ is concentrated on $\partial^{c}\varphi$, in which case $T$ induces an optimal plan from $(\mu_1)_q$ to $(\mu_2)_q$ by Theorem~\ref{Thm: fundamental theorem of optimal transport]}.
\end{enumerate} 

Since $(\text{Id}, T)_{*}(\mu_1)$ is concentrated on $\partial^{c}\varphi$, we can find a Borel set $A \subseteq X$ so that $\mu_1(A) = 1$ and $(x, T(x)) \in \partial^c\varphi$ for any $x \in A$. By the definition of $\{\mm_{q}^{\mu_1}\}$ we have
\begin{equation*}
\int_{q \in Q} \mm_q^{\mu_1}(A) \, d\mathfrak{q}(q) = \mu_1(A) = 1 = \mu_1(X) = \int_{q \in Q} \mm_q^{\mu_1}(X) \, d\mathfrak{q}(q).
\end{equation*}
In particular, $\mm_q^{\mu_1}(A)=\mm_q^{\mu_1}(X)$ for $\mathfrak{q}$-a.e.\ $q \in Q$. 
\end{proof}

\section{Main lemma}

In this section, we prove Lemma~\ref{Lem: main lemma}. Roughly, it says that, given a ramified double cover $(\hat{X}, \hat{p})$ of a non-orientable $\RCD(K,3)$ space $(X, p)$, if $p$ has a tangent cone that is topologically a cone over $\mathbb{RP}^2$, then any optimal dynamical plan between absolutely continuous measures in $\hat{X}$ must have zero measure on the geodesics passing through $\hat{p}$. By rescaling if necessary, it suffices to consider only $\RCD(-2, 3)$ spaces. 

We establish some necessary concepts and notation. Let $(X, \dd)$ be a complete and separable geodesic space. For any $a, b \in \R$ with $a < b$, we denote by $C([a,b], X)$ and $\Lip([a,b], X)$ the spaces of continuous and Lipschitz curves from $[a,b]$ into $X$. Recall that $\Geo(X)$ is the space of all constant speed geodesics on $X$ parameterized on $[0,1]$ (see \eqref{Eq: Geo(X)}). We metrize all the spaces above with the $\sup$ norm. In addition, if we fix $p \in X$, we denote 
\begin{equation*}
C_p([a,b], X) := \Big\{\gamma \in C([a,b], X) \,:\, \gamma(t) = p \text{ for some $t \in [a,b]$} \Big\}.
\end{equation*}
The spaces $\Lip_p([a,b], X)$ and $\Geo_p(X)$ are defined in the same way. As with $\Geo(X)$, we will denote by $e_t$ the evaluation map on these various spaces in the obvious way (see \eqref{Eq: evaluation map def}). 

We recall the following theorem on the tangent cone of non-collapsed $\RCD(K,N)$ spaces, which follows from \cite[Theorem 1.2]{DPG18}, \cite[Theorem 1.1]{DPG16}, and \cite[Theorem 1.2]{K15}.
\begin{Thm}\label{Thm: non-collapsed RCD tangent cone main theorem}
Let $(X, \dd, \HH^n)$ be a non-collapsed $\RCD(K,n)$ space and let $p \in X$. Then any tangent cone at $p$ is of the form $(C(Y), \dd_{C(Y)}, \HH^n)$ where $(Y, \dd_{Y}, \HH^{n-1})$ is a non-collapsed $\RCD(n-2,n-1)$ space and $(C(Y), \dd_{C(Y)}, \HH^n)$ is the standard metric cone over $Y$ (equipped with the Hausdorff measure associated with the metric structure).
\end{Thm}

We now state our main lemma. 
\begin{Lem}\label{Lem: main lemma}(Main lemma).
Let $(X, \dd, \HH^3)$ be a non-collapsed $\RCD(-2, 3)$ space without boundary and let $(\hat{X}, \hat{\dd}, \HH^3)$ be its ramified double cover. Let $\hat{p} \in \hat{X}$ and suppose that there is a tangent cone  at $p = \pi(\hat{p})$ of the form $(C(Y), \dd_{C(Y)}, \HH^3)$, where $Y$ is homeomorphic to $\mathbb{RP}^2$. Then, for any $\hat{\mu}_0, \hat{\mu}_1 \in \mathcal{P}(\hat{X})$ with $\hat{\mu}_0, \hat{\mu}_1 \ll \mathcal{H}^3$ and any $\hat{\nu} \in \OptGeo(\hat{\mu}_0, \hat{\mu}_1)$, we have $\hat{\nu}(\Geo_{\hat{p}}(\hat{X}))=0$. 
\end{Lem}

\begin{Rem}\label{Rem: one tangent cone all tangent cone}
As observed in \cite[Remark 9.2]{BPS24}, at any point $p$, the cross-sections of the tangent cones must either all be homeomorphic to $\mathbb{RP}^2$ or all be homeomorphic to $\mathbb{S}^2$. Indeed, it can be checked that the set of all cross-sections of tangent cones at $p$ is connected in the class of non-collapsed $\RCD(1,2)$ spaces with some uniform lower volume bound (with topology induced by the Gromov--Hausdorff distance); see for instance \cite[Lemma 12.16]{BPS24}. It then follows that all such cross-sections are homeomorphic, using the characterization that all non-collapsed $\RCD(1,2)$ spaces are $2$-dimensional Alexandrov spaces with $\curv \geq 1$ \cite[Theorem 1.1]{LS23} along with either Perelman's stability theorem \cite{P91} (see also \cite{K07}) or \cite{P90} and \cite[Proposition 3.5]{BPS24}. The fact that the cross-sections are homeomorphic to $\mathbb{RP}^2$ or $\mathbb{S}^2$ follows from the topological classification of $2$-dimensional Alexandrov spaces with $\curv \geq 1$ and without boundary (see for instance \cite[Corollary 3.3]{BPS24}). 
\end{Rem}

We will prove Lemma~\ref{Lem: main lemma} by contradiction. By Proposition~\ref{Lem: non-orientable tangent cone implies fixed point}, we know that $\hat{p}$ is the unique lift of $p$ via $\pi$. Assuming that a counterexample $\hat{\nu}$ exists, the strategy of our proof is as follows.
\begin{enumerate}
    \item First, we use $\hat\nu$ to construct a sequence of pairs of probability measures contained in $B_{r_i}(\hat{p})$, where $r_i \to 0$, so that
    \begin{enumerate}
        \item[(i)] for each $i$, an optimal dynamical plan between the two measures is concentrated on the set of geodesics that pass through $p$;
        \item[(ii)] for each $i$, the two measures have uniformly bounded density with respect to $\mathcal H^3$ when $B_{r_i}(\hat{p})$ is rescaled to radius $1$.
    \end{enumerate}
    \item Second, we show that if $B_{10}(p)$ is sufficiently Gromov--Hausdorff close to the ball of radius $10$ centered at the cone tip of a cone over $Y^2$, where $Y^2$ is homeomorphic to $\mathbb{RP}^2$ and Alexandrov with $\curv \geq 1$, then the optimal dynamical plan between two probability measures with a given density bound and concentrated in $B_1(\hat{p})$ cannot be supported on the set of geodesics that pass through $\hat{p}$. This is done by exploiting our knowledge of the geometry of metrics on $\mathbb{RP}^2$ with $\curv \geq 1$.
\end{enumerate}
The first step is made difficult because we do not have a priori $\RCD$ structure on $\hat X$ to control the density of $(e_t)_*(\hat\nu)$ for intermediate times $t \in (0,1)$, which is what we will use to construct our sequence of measures. Here we will exploit the 1D-localization of $\dd_p$ and the idea that the optimal transport in question can be thought of as transporting mass along the needles of the localization. 

We proceed with our plan. Suppose that Lemma~\ref{Lem: main lemma} is false. Fix $(X, \dd, \HH^3)$, $\hat{p} \in \hat{X}$, $\hat\mu_0, \hat\mu_1 \in \mathcal{P}(\hat{X})$, and $\hat\nu \in \OptGeo(\hat\mu_0, \hat\mu_1)$ that give a counterexample. We show that we may construct possibly different $\hat\mu_0, \hat\mu_1 \in \mathcal{P}(\hat{X})$ and $\hat\nu \in \OptGeo(\hat\mu_0,\hat\mu_1)$ with a few additional properties, which still give a counterexample. First, by replacing $\hat\nu$ with $\hat\nu \mrestr {\Geo_{\hat{p}}(\hat{X})}/\hat\nu(\Geo_{\hat{p}}(\hat{X}))$ and $\hat\mu_0,\hat\mu_1$ correspondingly, we may assume that $\hat\nu$ is concentrated on $\Geo_{\hat{p}}(\hat{X})$, i.e., $\hat\nu(\Geo_p(\hat{X}))=1$. Next, writing $\hat\mu_i=\hat u_i\mathcal{H}^3$ and taking a disintegration of $\hat\nu$ with respect to $\hat\mu_0 = \hat{u}_0 \HH^3$ through the evaluation map $e_0$, we have
\begin{equation*}
    \hat\nu = \int_{\hat{X}} \hat\nu_x \, d\hat\mu_0(x) = \int_{\hat{X}} \hat\nu_x \hat{u}_0(x) \, d\HH^3(x). 
\end{equation*}
For a sufficiently large $L > 0$, the measure
\begin{equation*}
    \hat\nu' := \int_{\hat{X}} \hat\nu_x \min\{\hat{u}_0(x), L\} \, d\HH^3(x)
\end{equation*}
satisfies $\hat\nu'(\Geo(\hat{X})) \neq 0$. Replacing $\hat\nu$ with $\hat\nu'/\hat\nu'(\Geo(X))$ (which is an optimal dynamical plan, since whenever $\rho$ is an optimal plan and $0\le\rho'\le\rho$ is nonzero then
the normalization of $\rho'$ is an optimal plan), we may assume that $\hat\mu_0 = (e_0)_*(\hat\nu)$ has bounded density. Repeating the same procedure for $\hat\mu_1$, we may also assume that $\hat\mu_1 = (e_1)_*(\hat\nu)$ has bounded density. Finally, there exists $R > r > 0$ so that $\hat\nu$ restricted to the set of geodesics that start and end in $A_{r, R}(\hat{p}):= B_{R}(\hat{p})/B_{r}(\hat{p})$ has positive measure. Normalizing the restriction of $\hat\nu$ to this set and possibly rescaling $\hat{X}$ around $\hat{p}$, we may assume that $\hat\nu$ is concentrated on the set of geodesics that start and end in $A_{1, R}(\hat{p})$. 

To summarize, this means that there exist $\hat\mu_0, \hat\mu_1 \in \mathcal{P}(\hat{X})$ and an optimal dynamical plan $\hat\nu$ from $\hat\mu_0$ to $\hat\mu_1$ so that
\begin{enumerate}
    \item $\hat\nu(\Geo_{\hat{p}}(\hat{X}))=1$;
    \item there exists $C_0 > 1$ so that $\hat\mu_0, \hat\mu_1$ are concentrated in $A_{1, C_0}(\hat{p}):=B_{C_0}(\hat{p}) \setminus \overline{B_{1}(\hat{p})}$;
    \item $\hat\mu_0 = \hat{u}_0\HH^3$ and $\hat\mu_1 = \hat{u}_1\HH^3$ where $\hat{u}_0, \hat{u}_1 \in L^1 \cap L^\infty(\hat{X}, \HH^3)$.
\end{enumerate} 
Fix such a counterexample. For each $k \in \N$ and any $i = 1, \dots ,k$, we define 
\begin{equation*}
G_k^i := \Big\{\gamma \in \Geo(\hat{X}) \,:\, \gamma(t)=\hat{p} \text{ for some $t \in [(i-1)/k, i/k]$}\Big\}
\end{equation*}
and choose some $i_k$ so that $\hat\nu(G_k^{i_k}) \geq 1/k$. We define
\begin{equation}\label{Eq: nuk definition}
    \hat\nu_k := \frac{\hat\nu \mrestr {G_k^{i_k}}}{\hat\nu(G_k^{i_k})}.
\end{equation}
Since $\hat\mu_0, \hat\mu_1$ are concentrated in $A_{1,C_0}(\hat{p})$, we have $\hat\nu(G_k^1)=\hat\nu(G_k^2)=0$ for sufficiently large $k$ and so $i_k > 2$. From now on, we will only consider sufficiently large $k$ so that $i_k > 2$. 

Let $\Pi: \Geo(\hat{X}) \to C([0,1],X)$ be the map that takes a geodesic $\gamma$ to the curve $\pi \circ \gamma$. For each $k$, define the map $R_k: C([0,1],X) \to C([0,1], X)$ which takes a curve $\gamma:[0,1] \to X$ to the curve $R_k(\gamma):[0,1] \to X$ defined by $R_k(\gamma)(t) := \gamma((i_k-2)t/k)$, i.e., $R_k(\gamma)$ is a linear reparameterization on $[0,1]$ of $\gamma$ restricted to $[0, (i_k-2)/k]$. Since $\hat\nu_k$ is concentrated on $\Geo_{\hat{p}}(\hat{X})$ and $\hat{p}$ is the unique lift of $p$, by Lemma~\ref{Lem: Geodesic break} we see that $\nu_k := (R_k \circ \Pi)_{*}(\hat\nu_k)$ is concentrated on $\Geo(X) \subseteq C([0,1], X)$.
\begin{Cla}\label{Cla: projection optimal dynamical}
The measure $\nu_k$ is an optimal dynamical plan in $X$.
\end{Cla}
\begin{proof}
We prove the claim by showing that if $\eta \in \mathcal{P}(X \times X)$ is an optimal plan between $(e_0)_*(\nu_k)$ and $(e_1)_*(\nu_k)$ so that 
\begin{equation*}
    \int_{X \times X} \dd^2(x,y) \, d\eta(x,y) < \int_{X \times X} \dd^2(x, y) \, d((e_0, e_1)_*(\nu_k))(x,y),
\end{equation*}
then we can lift $\eta$ to an admissible pairing in $\hat{X} \times \hat{X}$ and use it to construct a plan between $(e_0)_*(\hat\nu_k)$ and $(e_1)_*(\hat\nu_k)$ with strictly lower cost than $(e_0, e_1)_*(\hat\nu_k)$ for a contradiction. Assume that such an $\eta$ exists. We first note that
\begin{align*}
    \int_{X \times X} \dd^2(x, y) \, d((e_0, e_1)_*(\nu_k))(x,y)
    &= \int_{\hat{X} \times \hat{X}} \dd^2(\pi(x), \pi(y)) \, d((e_0, e_{(i_k-2)/k})_*(\hat\nu_k))(x,y)\\
    &= \int_{\hat{X} \times \hat{X}} \hat{\dd}^2(x,y) \, d((e_0, e_{(i_k-2)/k})_*(\hat\nu_k))(x,y)\\
    &= W_2^2((e_0)_*(\hat\nu_k), (e_{(i_k-2)/k})_*(\hat\nu_k)),
\end{align*}
where the second equality comes from Lemma~\ref{Lem: Geodesic break}. Therefore, we have
\begin{equation}\label{Eq: first part geodesic length bound}
     \int_{X \times X} \dd^2(x,y) \, d\eta(x,y) < W_2^2((e_0)_*(\hat\nu_k), (e_{(i_k-2)/k})_*(\hat\nu_k)). 
\end{equation}

The process of taking a lift of $\eta$ is clear but somewhat technical. We first lift $\eta$ to a probability measure $\tilde\eta \in \mathcal{P}(\hat{X} \times X)$. Since $\pi: \hat{X} \to X$ is continuous and $(\pi_1)_*(\eta)=\pi_*((e_0)_*(\hat\nu_k))$, it is standard to construct, using disintegration, a probability measure $\tilde\eta \in \mathcal{P}(\hat{X} \times X)$ so that $(\pi, \text{Id})_*(\tilde\eta) = \eta$ and $(\pi_1)_*(\tilde\eta) = (e_0)_*(\hat\nu_k)$, where $\pi_i$ is the projection onto the $i$-th factor. 

Next we lift $\tilde\eta$ to a probability measure in $\mathcal{P}(\hat{X} \times \hat{X})$. For a set $A$ we denote by $P(A)$ the power set of $A$. Let $F: \hat{X} \times X \to P(\hat{X} \times \hat{X})$ be the set-valued function defined by
\begin{equation*}
F(x,y) := \{(x,z) \in \hat{X} \times \hat{X} \,:\, z \in \pi^{-1}(y) \text{ and } \hat{\dd}(x,z) = \dd(\pi(x),y)\}.
\end{equation*}
By Remark~\ref{Rem: Ramified double cover properties} (1)--(2), $F(x,y)$ is non-empty and contains at most two elements for any $(x,y) \in \hat{X} \times X$. Let $\mathscr{B}$ be the Borel $\sigma$-algebra on $\hat{X} \times X$. Then $F$ is clearly $\mathscr{B}$-weakly measurable in the sense that, for any open $U \subseteq \hat{X} \times \hat{X}$, the set
$$\{(x,y) \in \hat{X} \times X : F(x,y) \cap U \neq \emptyset\} \in \mathscr{B},$$
since in fact it is $\sigma$-compact (as $F$ has closed graph).
By Kuratowski and Ryll-Nardzewski's measurable selection theorem, we can take a section $G: \hat{X} \times X \to \hat{X} \times \hat{X}$ of $F$ (i.e., $G(x,y) \in F(x,y)$ for any $(x,y) \in \hat{X} \times X$) that is Borel. Defining $\hat\eta_1 := G_{*}(\tilde\eta)$, the following are easy consequences of the construction: 
\begin{enumerate}
    \item $(\pi, \pi)_*(\hat\eta_1)=(\pi, \text{Id})_*(\tilde\eta) = \eta$;
    \item we have
    \begin{align*}
        \int_{\hat{X} \times \hat{X}} \hat{\dd}^2(x,y) \, d\hat\eta_1(x,y) &= \int_{\hat{X} \times \hat{X}} \dd^2(\pi(x),\pi(y)) \, d\hat\eta_1(x,y)\\
        &= \int_{X \times X} \dd^2(x,y) \, d\eta(x,y)\\
        &< W_2^2((e_0)_*(\hat\nu_k), (e_{(i_k-2)/k})_*(\hat\nu_k)),
    \end{align*}
    where the first equality comes from the definition of $G$ and the last inequality comes from \eqref{Eq: first part geodesic length bound};
    \item $(\pi_1)_*(\hat\eta_1)=(\pi_1)_*(\tilde\eta) = (e_0)_*(\hat\nu_k)$ and so $\pi_{*}((\pi_1)_*(\hat\eta_1))= (e_0)_*(\nu_k)$;
    \item $(\pi_2)_*(\hat\eta_1)$ is not necessarily the same as $(e_{(i_k-2)/k})_*(\hat\nu_k)$, but their pushforwards via $\pi$ are the same and equal to $(e_1)_*(\nu_k)$. 
\end{enumerate}
We now claim that the $L^2$-Wasserstein distance from $(\pi_2)_*(\hat\eta_1)$ to $(e_1)_*(\hat\nu_k)$ is no greater than the $L^2$-Wasserstein distance from $(e_{(i_k-2)/k})_*(\hat\nu_k)$ to $(e_1)_*(\hat\nu_k)$. Indeed, by assumption $\hat\eta_2' := (e_{(i_k-2)/k}, e_1)_*(\hat\nu_k) \in \mathcal{P}(\hat{X} \times \hat{X})$ gives an optimal plan from $(e_{(i_k-2)/k})_*(\hat\nu_k)$ to $(e_1)_*(\hat\nu_k)$ that is concentrated on the set 
\begin{equation*}
B:=\{(x,y) \in \hat{X} \times \hat{X} \,:\, \text{there exists a geodesic from $x$ to $y$ passing through $\hat{p}$}\}.
\end{equation*}
Define $\tilde\eta_2' := (\pi, \text{Id})_*(\hat\eta_2') \in \mathcal{P}(X \times \hat{X})$. The first marginal is $(\pi_1)_*(\tilde\eta_2')=(e_1)_*(\nu_k)$. As before, since $\pi_*((\pi_2)_*(\hat\eta_1)) = (e_1)_*(\nu_k)$ as well, we may lift $\tilde\eta_2'$ to a probability measure $\hat\eta_2 \in \mathcal{P}(\hat{X} \times \hat{X})$ so that $(\pi, \text{Id})_*(\hat\eta_2) = \tilde\eta_2'$ and $(\pi_1)_*(\hat\eta_2)= (\pi_2)_*(\hat\eta_1)$.
Note that, since we are imposing both marginals $(\pi_1)_*(\hat\eta_2)$
and $(\pi_2)_*(\hat\eta_2)$, the measure $\hat\eta_2$ is built through disintegration of $\tilde\eta_2'$, rather than through an analogue of the map $G$. 
We claim that
\begin{equation*}
    \int_{\hat{X} \times \hat{X}} \hat{\dd}^2(x,y) \, d\hat\eta_2(x,y) \leq  \int_{\hat{X} \times \hat{X}} \hat{\dd}^2(x,y) \, d\hat\eta'_2(x,y).
\end{equation*}
Indeed, every $(x,y)$ in $B$ has the property that $\hat{\dd}(x,y) = \dd(\pi(x), p) + \dd(p, \pi(y))$ by Remark~\ref{Rem: Ramified double cover properties} (2) and so
\begin{align*}
    \int_{\hat{X} \times \hat{X}} \hat{\dd}^2(x,y) \, d\hat\eta_2(x,y) &\leq
    \int_{\hat{X} \times \hat{X}}
    \Big(\dd(\pi(x), p)+\dd(p, \pi(y))\Big)^2 \, d\hat\eta_2(x,y)\\
    &= \int_{X \times \hat{X}} \Big(\dd(\check{x}, p)+\dd(p, \pi(y))\Big)^2 \, d\tilde\eta'_2(\check{x},y)\\
    &= \int_{\hat{X} \times \hat{X}}
    \Big(\dd(\pi(x), p)+\dd(p, \pi(y))\Big)^2 \, d\hat\eta'_2(x,y)\\
    &=  \int_{\hat{X} \times \hat{X}} \hat{\dd}^2(x,y) \, d\hat\eta'_2(x,y),
\end{align*}
as required. 
To summarize, we have shown that 
\begin{equation*}
    W_2((e_0)_*(\hat\nu_k), (\pi_2)_*(\hat\eta_1)) < W_2((e_0)_*(\hat\nu_k), (e_{(i_k-2)/k})_*(\hat\nu_k))
\end{equation*}
and 
\begin{equation*}
    W_2((\pi_2)_*(\hat\eta_1), (e_1)_*(\hat\nu_k)) \leq W_2((e_{(i_k-2)/k})_*(\hat\nu_k), (e_1)_*(\hat\nu_k)),
\end{equation*}
which is a contradiction to the assumption that $\hat\nu_k$ is an optimal dynamical plan. Therefore, Claim~\ref{Cla: projection optimal dynamical} holds. 
\end{proof}

Next, we would like to show that $(e_1)_*(\nu_k) = \pi_*((\e_{(i_k-2)/k})_*(\hat\nu_k))$ has uniformly bounded density (in $k$) once the ball $B_{1/k}(p)$ is rescaled to radius $1$. We outline our strategy: first, note that $(e_0)_*(\nu_k)$ has a density upper bound which scales like $k$. Now, since $\nu_k$ can be extended a little further as an optimal dynamical plan on $X$ (by taking the reparameterization map $R_k$ in the definition of $\nu_k$ to the interval $[0, (i_k-3/2)/k]$ instead of $[0, (i_k-2)/k]$), we can use \cite[Theorem 1.1]{CM17a} to put a density upper bound on $(e_1)_*(\nu_k)$ which would scale like $k^3$ times the original density upper bound on $(e_0)_*(\nu_k)$. In all we would obtain a density upper bound which scales like $k^4$.

The rescaling would only give a factor of $k^{-3}$, so this strategy would not give us a uniform bound on the density. What we will do instead is use 1D-localization with respect to $\dd_p$ and the fact that $\hat\nu_k$ extends well past $\hat{p}$ to show that the transport of the masses along the needles by $\nu_k$ cannot bring two points on the same needle that are far away from each other too close around $p$. The intuition behind this is that if one has an optimal dynamical plan on $\R$ for quadratic cost and two parts of the initial mass are far apart, then they cannot get too close together in the middle of the optimal transport due to monotonicity. This allows us to obtain some quantitative control on the transport along the needles and shows in particular that the transport does not ``squish" the distance too much in the direction of the needles, which saves us one factor of $k$ for the upper bound on density. 

\begin{Lem}
For any $\gamma_1, \gamma_2 \in \supp(\hat\nu_k)$, if $\hat{\dd}(\gamma_1(0), \hat{p}) > \hat{\dd}(\gamma_2(0), \hat{p})$ then $\hat{\dd}(\hat{p},\gamma_1(1)) \leq \hat{\dd}(\hat{p}, \gamma_2(1))$.
\end{Lem}

\begin{proof}
Since for every $\gamma \in \supp(\hat\nu_k)$ we have $(\gamma(0), \gamma(1)) \in \supp((e_0,e_1)_*(\hat\nu_k))$, it suffices to show that for any $(x_1,y_1)$ and $(x_2, y_2)$ in $\supp((e_0,e_1)_*(\hat\nu_k))$, if $\hat{\dd}(x_1, \hat{p}) > \hat{\dd}(x_2, \hat{p})$, then $\hat{\dd}(\hat{p},y_1) \leq \hat{\dd}(\hat{p}, y_2)$. Assuming this fails, let us show that $\supp((e_0,e_1)_*(\hat\nu_k))$ is not $\hat{\dd}^2$-cyclically monotone, which is a contradiction to the fundamental theorem of optimal transport \cite[Theorem 1.13]{AG13}. Indeed, we have
$$(\hat{\dd}(x_1, \hat{p})-\hat{\dd}(x_2, \hat{p}))(\hat{\dd}(\hat{p}, y_1)-\hat{\dd}(\hat{p}, y_2))>0,$$
from which we obtain
$$\hat{\dd}(x_1, \hat{p})\hat{\dd}(\hat{p}, y_2)+\hat{\dd}(x_2, \hat{p})\hat{\dd}(\hat{p}, y_1)
<\hat{\dd}(x_1, \hat{p})\hat{\dd}(\hat{p}, y_1)+\hat{\dd}(x_2, \hat{p})\hat{\dd}(\hat{p}, y_2).$$
We deduce that
\begin{align*}
&\hat{\dd}^2(x_1,y_2)+\hat{\dd}^2(x_2,y_1)\\
    &\le (\hat{\dd}(x_1, \hat{p})+\hat{\dd}(\hat{p}, y_2))^2+(\hat{\dd}(x_2, \hat{p})+\hat{\dd}(\hat{p}, y_1))^2\\
    &< (\hat{\dd}(x_1, \hat{p})+\hat{\dd}(\hat{p}, y_1))^2+
    (\hat{\dd}(x_2, \hat{p})+\hat{\dd}(\hat{p}, y_2))^2\\
    &= \hat{\dd}^2(x_1,y_1)+\hat{\dd}^2(x_2,y_2),
\end{align*}
as required; note that every $(x,y) \in \supp((e_0,e_1)_*(\hat\nu_k))$ necessarily has the property that $\hat{\dd}(x,y) = \hat{\dd}(x,\hat{p})+\hat{\dd}(\hat{p}, y)$ by our assumption on $\hat\nu_k$.
\end{proof}

Therefore, for any $\gamma_1, \gamma_2 \in \supp(\hat\nu_k)$, if $\hat{\dd}(\gamma_1(0), \hat{p}) > \hat{\dd}(\gamma_2(0),\hat{p})$, then
\begin{align*}
\hat{\dd}\Big(\gamma_1\Big(\frac{i_k-2}{k}\Big), \hat{p}\Big) &= \hat{\dd}(\gamma_1(0), \hat{p})-\frac{i_k-2}{k}\hat{\dd}(\gamma_1(0), \gamma_1(1))\\
    &\geq \hat{\dd}(\gamma_1(0), \hat{p})-\frac{i_k-2}{k}\Big(\hat{\dd}(\gamma_1(0), \hat{p})+\hat{\dd}(\hat{p}, \gamma_2(1))\Big)\\
    &= \hat{\dd}(\gamma_1(0), \hat{p})-\frac{i_k-2}{k}\Big(\big(\hat{\dd}(\gamma_1(0), \hat{p})-\hat{\dd}(\gamma_2(0), \hat{p})\big)+\hat{\dd}(\gamma_2(0), \gamma_2(1))\Big)\\
    &=\hat{\dd}\Big(\gamma_2\Big(\frac{i_k-2}{k}\Big),\hat{p}\Big)+\Big(1-\frac{i_k-2}{k}\Big)\Big(\hat{\dd}(\gamma_1(0), \hat{p})-\hat{\dd}(\gamma_2(0), \hat{p})\Big),
\end{align*}
where we used the previous lemma in the inequality. 
By our initial assumptions, $\hat\nu$ is an optimal dynamical plan from $\hat\mu_0$ and $\hat\mu_1$ which are both concentrated in the open annulus $A_{1,C_0}(\hat{p})$. Therefore, for every $\gamma \in \supp(\hat\nu_k)$, we have 
\begin{equation*}
    \frac{i_k-2}{k} = \frac{\hat{\dd}(\gamma(0), \gamma(\frac{i_k-2}{k}))}{\hat{\dd}(\gamma(0), \gamma(1))} \leq \frac{\hat{\dd}(\gamma(0), \hat{p})}{\hat{\dd}(\gamma(0), \hat{p})+\hat{\dd}(\hat{p}, \gamma(1))} < \frac{C_0}{C_0+1} < 1. 
\end{equation*}
Putting everything together and setting $c_0 := 1/(C_0+1)$, we have proved the following lemma. 
\begin{Lem}\label{Lem: double cover monotonicity bound}
For any $\gamma_1, \gamma_2 \in \supp(\hat\nu_k)$, if $\hat{\dd}(\gamma_1(0), \hat{p}) > \hat{\dd}(\gamma_2(0),\hat{p})$ then 
\begin{equation*}
    \hat{\dd}(\gamma_1((i_k-2)/k), \hat{p})- \hat{\dd}(\gamma_2((i_k-2)/k), \hat{p}) > c_0\Big(\hat{\dd}(\gamma_1(0), \hat{p})-\hat{\dd}(\gamma_2(0), \hat{p})\Big).
\end{equation*}
\end{Lem}
Now recall that by definition $\nu_k = (R_k \circ \Pi)_*(\hat\nu_k)$. Since $R_k \circ \Pi$ is continuous, $\supp(\hat\nu_k)$ is closed, and $\Geo(\hat{X})$ is Polish, $(R_k \circ \Pi)(\supp(\hat\nu_k)) \subseteq \Geo(X)$ is analytic and hence measurable with respect to the completion of $\nu_k$. This implies that we can find a Borel set $S \subseteq \Geo(X)$ so that $S \subseteq (R_k \circ \Pi)(\supp(\hat\nu_k))$ and $\nu_k(S)=1$. For each $\gamma_1, \gamma_2 \in S$, Lemma~\ref{Lem: double cover monotonicity bound} and Lemma~\ref{Lem: Geodesic break} imply that if $\dd(\gamma_1(0), p) > \dd(\gamma_2(0),p)$ then
\begin{equation*}
    \dd(\gamma_1(1), p)- \dd(\gamma_2(1)), p) > c_0\Big(\dd(\gamma_1(0), p)-\dd(\gamma_2(0), p)\Big).
\end{equation*}
Define now the analytic set $(e_0, e_1)(S) \subseteq X \times X$. Arguing as above we can find some Borel set $S' \subseteq (e_0, e_1)(S)$ so that $(e_0, e_1)_*(\nu_k)(S')=1$.

Let $\mu_0^k := (e_0)_*(\nu_k)$ and $\mu_1^k := (e_1)_*(\nu_k)$. We see that $\mu_0^k \ll \HH^3$ and if $u_0^k$ is such that $\mu_0^k = u_0^k\HH^3$ then $\norm{u_0^k}{L^\infty(X, \HH^3)}\leq2k\norm{\hat{u}_0}{L^\infty(\hat{X}, \HH^3)}$; this follows from the fact that $\pi_*(\HH_{\hat{X}}^3)=2\HH^3_{X}$ and the definition of $\hat\nu_k$ (see \eqref{Eq: nuk definition}). Furthermore, it can be checked that $\mu_1^k$ is absolutely continuous with respect to $\HH^3$. Indeed, if we define $R_k': C([0,1], X) \to C([0,1], X)$ by $R_k'(\gamma)(t) := \gamma((i_k-3/2)t/k)$ and consider $\nu_k' := (R_k' \circ \Pi)_*(\hat\nu_k)$, arguing as in Claim~\ref{Cla: projection optimal dynamical} we can check that $\nu_k'$ is an optimal dynamical plan in $X$. Since $\mu_0^k = (e_0)_*(\nu_k')$ has bounded density and $\mu_1^k = (e_{(i_k-2)/(i_k - 3/2)})_*(\nu_k')$, \cite[Theorem 1.1]{CM17a} gives that $\mu_1^k$ is absolutely continuous with respect to $\HH^3$ (with an explicit $L^\infty(X, \HH^3)$ density bound of the order $k^4$). We write $\mu_1^k=u_1^k\mathcal{H}^3$ in the sequel.

The optimal plan $(e_0, e_1)_*(\nu_k)$ is induced by some map $T_k: X \to X$ by \cite[Theorem 1.1]{CM17a}
(see also \cite{G12} and \cite{D25}). As such, we can find a Borel set $S''\subseteq X$ so that $\mu_0^k(S'')=1$ and $(\text{Id}, T_k)(S'') \subseteq S'$. The latter implies that, for any $x, y \in S''$, if $\dd(x, p)>\dd(y, p)$ then
\begin{equation}\label{Eq: optimal plan co-Lipschitz bound}
\dd(T_k(x), p)- \dd(T_k(y), p) > c_0\Big(\dd(x, p)-\dd(y, p)\Big).
\end{equation}

Using $\mm$ to denote $\HH_{X}^3$ for notational simplicity, we now apply Theorem~\ref{Thm: Main disintegration thm} (see also \eqref{Eq: m disintegration formula}) to disintegrate $\mm$ into $\{\mm_q\}_{q \in Q}$ with respect to the $1$-Lipschitz function $d_p$. We then apply Lemma~\ref{Lem: another probability disintegration} to $\mu_0^k$ and $\mu_1^k$ to write
\begin{equation*}
\mu_0^k = \int_{q} \mm_q^{u_k^0} \, d\mathfrak{q}(q) \quad\text{and}\quad \mu_1^k = \int_{q} \mm_q^{u_k^1} \, d\mathfrak{q}(q). 
\end{equation*}

Let $Q_0$ be the set of all $q \in Q$ so that the following hold: 
\begin{enumerate}
    \item $\mm_q$ is concentrated on $\mathfrak{Q}^{-1}(q)$ and is absolutely continuous with respect to $g(q, \cdot)_*(\LL^1)$, where $g$ is the ray map introduced in Proposition~\ref{Pro: Ray map properties};
    \item $(X_q, \dd_{X_q}, \mm_q)$ verifies the $\MCP(-2,3)$ property with $\supp(\mm_q)=X_q$, where $X_q$ is the closure of $\mathfrak{Q}^{-1}(q)$ and $\dd_{X_q}$ is the restriction of $\dd$ to $X_q$;
    \item $\mm_q^{u^k_0}$ and $\mm_q^{u^k_1}(X)$ are finite Borel measures which are absolutely continuous with respect to $\mm_q$, with $(T_k)_*(\mm_q^{u^k_0})=\mm_q^{u^k_1}$;
    \item $\mm_q^{u^k_0}(X) = \mm_q^{u^k_1}(X) \neq 0$ and, setting $(\mu^k_0)_q := \mm_q^{u^k_0}/\mm_q^{u^k_0}(X)$ and $(\mu^k_1)_q := \mm_q^{u^k_1}/\mm_q^{u^k_1}(X)$, the optimal transport from $(\mu^k_0)_q$ to $(\mu^k_1)_q$ is induced by the map $T_k$ and, for $(\mu^k_0)_q \times (\mu^k_0)_q$-a.e.\ $(x,y)$, we have 
\begin{equation*}
\dd(T_k(x), p)- \dd(T_k(y), p) > c_0\Big(\dd(x, p)-\dd(y, p)\Big).
\end{equation*}
\end{enumerate}
By Theorems~\ref{Thm: Main disintegration thm}, \ref{Thm: needle CD property}, Lemmas~\ref{Lem: needle transport}, \ref{Lem: needle optimal transport}, and \eqref{Eq: optimal plan co-Lipschitz bound}, we see that the set $Q_0$ has full measure outside of the set of $q \in \mathfrak{Q}$ where $\mm_q^{u^k_0}(X) = \mm_q^{u^k_1}(X) = 0$.

It is clear from the non-branching property of $\RCD(K,N)$ spaces and Proposition~\ref{Pro: Ray map properties} that $p \in X_q$ and there exists an isometry $g_q: (X_q, \dd_q) \to (I_q, \dd_{\R})$, where $I_q$ is either a closed interval $[0,a]$ or $[0,\infty)$ and $p = g_q(0)$. Indeed, $g_q$ may be taken to be $g_q(t) = g(q, \dd(q, p)-t)$, where $g$ is the ray map, completed by continuity on endpoints. 

Let $q \in Q_0$. By property~(1) of $Q_0$ we have $(g_q^{-1})_*(\mm_q) = m_q \LL^1$ for some $m_q \in L^1(I_q)$. By property~(2), $(I_q, \dd_{\R}, m_q\LL^1)$ satisfies the $\MCP(-2,3)$ condition and so $m_q$ is an $\MCP(-2,3)$ density (see \eqref{Eq: MCP(K,N) density def} and the discussion following it). By properties~(3) and (4), we may write $(g_q^{-1})_*(\mm_q^{u^k_0}) = (u^k_0\circ g_q)m_q \LL^1$ and $(g_q^{-1})_*(\mm_q^{u^k_1}) = (u^k_1 \circ g_q)m_q \LL^1$. Define $T_k^q: I_q \to I_q$ by
\begin{align*}
T_k^q(t):=\begin{cases}
	g_q^{-1}(T_k(g_q(t))) &\text{if } T_k(g_q(t)) \in X_q,\\
	0 &\text{otherwise,}
	\end{cases}
\end{align*}
we have 
\begin{equation*}
    (T_k^q)_*((u^k_0\circ g_q)m_q \LL^1) = (u^k_1 \circ g_q)m_q \LL^1.
\end{equation*}

Let $C_1 := \norm{\hat{u}_0}{L^\infty(\hat{X}, \HH^3)}$. As mentioned previously, we have $\norm{u^k_0}{L^\infty(X, \mm)} \leq 2kC_1$. By Theorem~\ref{Thm: Main disintegration thm}, this implies that $\norm{u^k_0 \circ g_q}{L^{\infty}(I_q, m_q\LL^1)} \leq 2kC_1$ for $\mathfrak{q}$-a.e.\ $q \in Q_0$. By our initial assumption that $\mu_0$ is concentrated in $A_{1, C_0}({p})$ (as $\hat\mu_0$ is concentrated in $A_{1, C_0}(\hat{p})$) and our construction, we see that for $\mathfrak{q}$-a.e.\ $q \in Q_0$, for $(u_0^k \circ g_q)m_q \LL^1$-a.e.\ $t \in I_q$, we have
$$t \in [1,C_0] \cap I_q\quad\text{and}\quad T_k^q(t) \in [2/k, C_0] \cap I_q$$
(as the speed of any $\gamma\in\text{supp}(\hat\nu_k)$ is at least $2$
and thus $\hat\dd(\gamma((i_k-2)/k),\hat p)\ge2/k$).
Using the fact that $m_q$ is an $\MCP(-2,3)$ density, we see that for any such $t$ it holds
\begin{equation*}
    \frac{m_q(T_k^q(t))}{m_q(t)} \geq c'(C_0)/k^2.
\end{equation*} 
By property (4) of $Q_0$, we can find a set $I_q^0$ of full $(u_0^k \circ g_q)m_q \LL^1$ measure so that, for any $t_1, t_2 \in I_q^0$, if $t_2 > t_1$ then
\begin{equation*}
T_k^q(t_2) - T_k^q(t_1) > c_0(t_2 - t_1).
\end{equation*}
Therefore, we have
\begin{equation*}
(T_k^q)_*(\LL^1 \mrestr {I_q^0}) \leq \frac{1}{c_0} \LL^1\mrestr {I_q}.
\end{equation*}
Putting everything together, we obtain
\begin{equation*}
    \norm{u_1^k \circ g_q}{L^{\infty}(I_q, (u_1^k \circ g_q) m_q\LL^1)}\leq (2kC_1)(k^2/c')(1/c_0) = C(C_0,C_1)k^3.
\end{equation*}
The same bound holds for $\norm{u_1^k \circ g_q}{L^{\infty}(I_q, \mm_q\LL^1)}$ since, in general, for any nonnegative $f \in L^1(Y, m)$, we have $\norm{f}{L^{\infty}(Y, m)}=\norm{f}{L^{\infty}(Y, fm)}$. By Theorem~\ref{Thm: Main disintegration thm}, we immediately obtain 
\begin{equation}\label{Eq: main intermediate density bound}
    \norm{u_1^k}{L^{\infty}(X, \mm)} \leq C(C_0, C_1)k^3.
\end{equation}

By the assumptions of Lemma~\ref{Lem: main lemma} and Remark~\ref{Rem: one tangent cone all tangent cone}, there exists an increasing sequence of natural numbers $k_i \to \infty$ so that $(X, k_i\dd, \HH^3, p) \to (C(Y), \dd_{C(Y)}, \HH^3, o_{C(Y)})$ in the pmGH sense, where each $\HH^3$ is understood to be the $3$-dimensional Hausdorff measure corresponding to the associated metric, $(Y, \dd_Y)$ is a $2$-dimensional Alexandrov space homeomorphic to $\mathbb{RP}^2$, $\dd_{C(Y)}$ is the standard cone metric on $C(Y)$ induced from $\dd_Y$, and $o_{C(Y)}$ is the cone tip. 

For notational clarity we will from now on set $(X_i, \dd_i, \HH^3, p_i) := (X, k_i\dd, \HH^3, p)$ and denote by $\pi_i: (\hat{X}_i, \hat{\dd}_i, \HH^3, \hat{p}_i) \to (X_i, \dd_i, \HH^3, p_i)$ their respective ramified double covers with involutions $\Gamma_i$. We claim that if Lemma~\ref{Lem: main lemma} is false, then we can find $R, C > 0$ so that the following claim holds.  
\begin{Cla}\label{Cla: Existence of sequence}
For each $i$, there exist $\hat{\mu}^0_i, \hat\mu^1_i \in \mathcal{P}(\hat{X}_i)$ so that the following hold:
\begin{enumerate}
    \item $\hat\mu^0_i, \hat\mu^1_i$ are supported in $A_{1,R}(\hat{p}_i)$;
    \item $\hat\mu^0_i, \hat\mu^1_i$ are absolutely continuous with respect to $\HH^3$ and their densities $\hat{u}^0_i, \hat{u}^1_i$ satisfy
    \begin{equation*}
        \norm{\hat{u}^0_i}{L^{\infty}(\hat{X}_i, \HH^3)}, \norm{\hat{u}^1_i}{L^{\infty}(\hat{X}_i, \HH^3)} \leq C;
    \end{equation*}
    \item there exists $\hat\nu_i \in \OptGeo(\hat\mu^0_i, \hat\mu^1_i)$ so that $\hat\nu_i$ is concentrated on $\Geo_{\hat{p}_i}(\hat X_i)$.
\end{enumerate}
\end{Cla}
To show the claim one may simply take $\hat\mu_i^0 := (e_{(i_{k_i}-2)/k_i})_*(\hat\nu_{k_i})$ from the previous part. The bound obtained in \eqref{Eq: main intermediate density bound} gives a uniform 
$L^\infty$ bound for the density of $(\pi_i)_*(\hat\mu^0_i)$, which, when combined with Remark~\ref{Rem: Ramified double cover properties} (3), gives a uniform $L^\infty$ bound for the density of $\hat\mu_i^0$ as well. Similarly, one may take $\hat\mu_i^1 = (e_{(i_{k_i}+1)/k_i})_*(\hat\nu_{k_i})$. A symmetric analysis starting at $\hat\mu_1$ as in the previous part gives the required uniform $L^\infty$ density bound for $\hat\mu_i^1$.

Here we collect a few facts about $2$-dimensional Alexandrov spaces with $\curv \geq 1$ and homeomorphic to $\mathbb{RP}^2$. These will be useful for the next stage of our argument. 
\begin{Rem}\label{Fac: Alexandrov facts}
Let $(Y, \dd_Y)$ be such a space.
\begin{enumerate}
    \item Defining $\diam(Y) := \sup_{x, y \in Y} \dd_{Y}(x,y)$, we have $\diam(Y) \leq \pi/2$.
    This follows from the diameter sphere theorem for Alexandrov spaces \cite[Theorem 4.5]{P91} (see also \cite{GP93}), which in the $2$-dimensional case says that any Alexandrov space with $\curv \geq 1$ and $\diam > \pi/2$ is homeomorphic to the $2$-dimensional sphere. 
    \item There exists a $2$-dimensional Alexandrov space $(\hat{Y}, \dd_{\hat{Y}})$ with $\curv \geq 1$ and homeomorphic to $\mathbb{S}^2$, equipped with a free involutive isometry $\Gamma: \hat{Y} \to \hat{Y}$, so that 
    \begin{equation*}
    (\hat{Y}/\langle \Gamma \rangle, \dd_{\hat{Y}/\langle \Gamma \rangle}) = (Y, \dd_Y).
    \end{equation*}
    This follows from the globalization theorem for Alexandrov spaces \cite{BGP92}. 
    \item If $x, y \in \hat{Y}$ are so that $\dd_{\hat{Y}}(x,y) = \pi$, then $\Gamma(x) = y$. Indeed, from the maximal diameter theorem for Alexandrov spaces with $\curv \geq 1$ it holds that, since $\dd_{\hat{Y}}(x,y)=\pi$ is maximal,  $\hat{Y}$ is a spherical suspension over a circle with diameter at most $2\pi$, with $x$ and $y$ as the tips of the suspension. If $\Gamma(x) \neq y$ then we have two distinct pairs of points of $\hat{Y}$ whose distance is maximal, namely $(x,y)$ and $(\Gamma(x),\Gamma(y))$. This is only possible if $\hat{Y}$ is a spherical suspension over a circle of diameter $2\pi$ and so $\hat{Y}$ is the standard sphere, but the only non-trivial free involutive isometry on the sphere is the standard one. 
\end{enumerate}
\end{Rem}

We now show that Claim~\ref{Cla: Existence of sequence} leads to a contradiction. Let $\hat\Pi_i := (e_0, e_1)_*(\hat\nu_i)$ be the optimal pairing between $\hat\mu_i^0$ and $\hat\mu_i^1$ associated with $\hat\nu_i$. We denote by $\mu_i^0$, $\mu_i^1$, and $\Pi_i$ the probability measures $(\pi_i)_*(\hat\mu_i^0)$, $(\pi_i)_*(\hat\mu_i^1)$, and $(\pi_i, \pi_i)_*(\hat\Pi_i)$ respectively. 

Let $(Z, \dd_{Z}, \{\iota_i\})$ be a realization (see Definition~\ref{Def: pmGH convergence}) for the pmGH convergence of spaces $(X_i, \dd_i, \mathcal{H}^3, p_i) \to (C(Y), \dd_{C(Y)}, \mathcal{H}^3, o_{C(Y)})$. It follows that $(X_i \times X_i, \dd_i \times \dd_i, \HH^{6}, (p_i, p_i))$ also converges to $(C(Y) \times C(Y), \dd_{C(Y)} \times \dd_{C(Y)}, \HH^{6}, (o_{C(Y)}, o_{C(Y)}))$ in the pmGH sense and $(Z \times Z, \dd_{Z} \times \dd_{Z}, \{(\iota_i, \iota_i)\})$ can be taken as a realization of the convergence. As such, up to a subsequence, we may take a weak limit of the sequences $(\mu_i^0)_i$, $(\mu_i^1)_i$ and $(\Pi_i)_i$ in the sense of Lemma~\ref{Lem: weak convergence of probability measures} with respect to these realizations. Denote by $\mu_\infty^0, \mu_\infty^1 \in \mathcal{P}(C(Y))$ and $\Pi_\infty \in \mathcal{P}(C(Y) \times C(Y))$ these limits respectively. It is not difficult to check the following from the definition of pmGH convergence and property (1) of Claim~\ref{Cla: Existence of sequence}: 
\begin{enumerate}
    \item[(i)] $(\pr_1)_*(\Pi_\infty) = \mu_\infty^0$ and $(\pr_2)_*(\Pi_\infty) = \mu_\infty^1$, where $\pr_j: C(Y) \times C(Y) \to C(Y)$ is the projection onto the $j$-th factor for $j=1,2$;
    \item[(ii)] $\mu_\infty^0$ and $\mu_\infty^1$ are supported in $\overline{A_{1, R}(o_{C(Y)})}$.
\end{enumerate}

Property (2) of Claim~\ref{Cla: Existence of sequence} and property (3) of Remark~\ref{Rem: Ramified double cover properties} gives that $\mu^0_i$ and $\mu^1_i$ are absolutely continuous with respect to $\HH^3_{X_i}$ and the $L^\infty$-norms of their densities are bounded by $2C$. Therefore, $\mu^0_\infty$ and $\mu^1_\infty$ are also absolutely continuous with respect to $\HH^3_{C(Y)}$ with their densities bounded by $2C$ due to Lemma~\ref{Lem: weak convergence density bound}.

Let $(\hat{Y}, \dd_{\hat{Y}}, \Gamma)$ be as in Remark~\ref{Fac: Alexandrov facts} (2) and let $\pi^{\hat{Y}}_Y: \hat{Y} \to Y$ be the associated projection map.
This induces a projection map $\pi^{C(\hat{Y})}_{C(Y)}:C(\hat{Y}) \to C(Y)$. In order to avoid confusion in the sequel, let us stress that $C(\hat Y)$ will be used as an auxiliary space
and that we will not claim at any point in the following argument that $C(\hat Y)$ is a limit of $\hat X_i$
(cf.\ Remark~\ref{illustration}), even if this fact will hold a posteriori.

Let $\veps < \veps(n=3)$ be sufficiently small and consider $A_\veps := A_\veps(C(Y))$; notice that, as long as we set $\veps$ sufficiently small, $o_{C(Y)} \notin A_\veps$. Since $\mu^0_\infty$, $\mu^1_\infty$ are absolutely continuous with respect to $\HH^3_{C(Y)}$ and $\HH^3(C(Y) \setminus A_\veps)=0$, we see that $\mu^0_\infty$, $\mu^0_\infty$ are concentrated on $A_\veps$ and $\Pi_\infty$ is concentrated on $A_\veps \times A_\veps$.

Let $\pi^{C(Y)}_Y: C(Y) \setminus \{o_{C(Y)}\} \to Y$ be the standard projection map. 
\begin{Cla}\label{Cla: cone optimal plan support}
$\Pi_\infty$ is concentrated on the set
\begin{equation*}
\Big\{(x_0,x_1) \in A_\veps \times A_\veps \,:\, \pi^{C(Y)}_Y(x_0) = \pi^{C(Y)}_Y(x_1)\Big\}.
\end{equation*}
\end{Cla}

\begin{Rem}
Before giving the full proof, let us illustrate the idea in the simpler case where $C(Y)\setminus \{o_{C(Y)}\} = A_\veps(C(Y))$. If in this case the above claim is not true, then we can find $(x_0,x_1) \in \supp(\Pi_\infty)$ so that $\pi^{C(Y)}_Y(x_0) \neq \pi^{C(Y)}_Y(x_1)$. Choose a lift $\hat{x}_0$ of $x_0$ and then consider two geodesics from $\hat{x}_0$ to the two lifts of $x_1$. Since $\pi^{C(Y)}_Y(x_0) \neq \pi^{C(Y)}_Y(x_1)$, by Remark~\ref{Fac: Alexandrov facts} (3), we see that neither of these geodesics passes through $o_{C(\hat{Y})}$
(see, e.g., \cite[Definition 3.6.12]{BBI01} for the formula of $\dd_{C(\hat Y)}$).
By projecting these two geodesics and concatenating, we arrive at an orientation-reversing curve in $C(Y) \setminus \{o_{C(Y)}\} = A_\veps(C(Y))$. Moreover, this loop can be subdivided into two curves from $x_0$ to $x_1$ both of whose length is strictly shorter than $\dd_{C(Y)}(x_0, o_{C(Y)})+\dd_{C(Y)}(o_{C(Y)}, x_1)$.

It can then be shown that this loop can be approximated by orientation-reversing loops in $A_{\delta}(X_i) \subseteq X_i$ for large $i$ (where $\delta$ can be made arbitrarily small by choosing $\veps$ small). We can find a sequence $(\hat{x}_0^i, \hat{x}_1^i) \in \supp(\hat{\Pi}_i)$ so that $(\pi_i(\hat{x}_0^i), \pi_i(\hat{x}_1^i))$ converges to $(x_0, x_1)$ in the pmGH convergence. It is then not difficult to use the orientation-reversing loop in $X_i$ to construct a curve from $\hat{x}_0^i$ to $\hat{x}_1^i$ whose length is strictly shorter than $\dd_{i}(x_0^i, p_i) + \dd_{i}(p_i, x_1^i)$ for large $i$, which contradicts our assumption on $\supp(\hat{\Pi}_i)$. The proof in the general case follows essentially the same idea, but makes technical adjustments to make sure that the constructed curves stay in a sufficiently regular part of $C(Y)$, and thus also the analogue for the approximating curves in $X_i$ (so that it makes sense to talk about orientation-reversing curves). 
\end{Rem}

\begin{proof}
It suffices to show that, for any $(x_0, x_1) \in \supp(\Pi_\infty)\cap (A_\veps \times A_\veps)$, $\pi^{C(Y)}_Y(x_0) = \pi^{C(Y)}_Y(x_1)$. Suppose for the sake of contradiction that this is not true. Denoting $\theta_0 := \pi^{C(Y)}_Y(x_0)$ and $\theta_1 := \pi^{C(Y)}_Y(x_1)$, we have $\theta_0 \neq \theta_1$. Choose $\hat{x}_1 \in (\pi^{C(\hat{Y})}_{C(Y)})^{-1}(x_1)$ and let $\hat{x}_0$ and $\hat{x}_0'$ be the distinct points in $(\pi^{C(\hat{Y})}_{C(Y)})^{-1}(x_0)$. Let $\pi^{C(\hat{Y})}_{\hat{Y}}: C(\hat{Y}) \setminus \{o_{C(\hat{Y})}\} \to \hat{Y}$ denote the standard projection map and let $\hat\theta_0, \hat\theta_0', \hat\theta_1$ denote $\pi^{C(\hat{Y})}_{\hat{Y}}(\hat{x}_0), \pi^{C(\hat{Y})}_{\hat{Y}}(\hat{x}_0'), \pi^{C(\hat{Y})}_{\hat{Y}}(\hat{x}_1)$ respectively. It is clear that $\pi^{\hat{Y}}_{Y}(\hat\theta_0) =\pi^{\hat{Y}}_{Y}(\hat\theta_0')= \theta_0$ and $\pi^{\hat{Y}}_{Y}(\hat\theta_1) = \theta_1$. Now, since $\theta_0 \neq \theta_1$, we have $\Gamma(\hat\theta_1) \neq \hat\theta_0, \hat\theta_0'$. By Remark~\ref{Fac: Alexandrov facts} (3), this implies $\dd_{\hat{Y}}(\hat\theta_0, \hat\theta_1), \dd_{\hat{Y}}(\hat\theta_0', \hat\theta_1) < \pi$. Therefore, by the formula of
$\dd_{C(\hat Y)}$ in terms of $\dd_{\hat Y}$ (see, e.g., \cite[Definition 3.6.12]{BBI01}), we have
\begin{align}\label{Eq: distance to cone tip bound}
\begin{split}
    \dd_{C(\hat{Y})}(\hat{x}_0, \hat{x}_1) &< \dd_{C(\hat{Y})}(\hat{x}_0, o_{C(\hat{Y})}) + \dd_{C(\hat{Y})}(o_{C(\hat{Y})}, \hat{x}_1),\\
    \dd_{C(\hat{Y})}(\hat{x}_0', \hat{x}_1) &< \dd_{C(\hat{Y})}(\hat{x}_0', o_{C(\hat{Y})}) + \dd_{C(\hat{Y})}(o_{C(\hat{Y})}, \hat{x}_1).
\end{split}
\end{align}

Recall the following lemma from \cite[Lemma 2.2]{BBP24}.
\begin{Lem}\label{Lem: RCD almost everywhere connectedness}
Let $(X, \dd, \HH^n)$ be a non-collapsed $\RCD(-(n-1),n)$ space without boundary. Let $C \subseteq X$ be closed with $\HH^{n-1}(C)$ = 0. Then, for every $x \in X \setminus C$, the following holds: For $\mathcal{H}^n$-a.e.\ $y \in X \setminus C$, there exists a geodesic $\gamma:[0,1] \to X$ connecting $x$ to $y$ with $\gamma([0,1]) \subseteq X \setminus C$.
\end{Lem}
Define $\hat{A}_\veps := (\pi^{C(\hat{Y})}_{C(Y)})^{-1}(A_\veps)$. Since $C(Y) \setminus A_\veps$ has Hausdorff dimension at most $n-2$, $C(\hat{Y}) \setminus \hat{A}_{\veps}$ does as well. Applying Lemma~\ref{Lem: RCD almost everywhere connectedness} with $C := C(\hat{Y}) \setminus \hat{A}_\veps$, we see that for any ball $B_\delta(\hat{x}_1)$ we can find $\hat{z} \in B_\delta(\hat{x}_1)$ so that there exist geodesics from $\hat{z}$ to $\hat{x}_1$, $\hat{x}_0$, $\hat{x}_0'$ respectively that only pass through $\hat A_\veps$. Let $\hat\gamma: [0,1] \to C(\hat{Y})$
(resp.\ $\hat\gamma': [0,1] \to C(\hat{Y})$) be the constant speed reparameterization of the curve obtained by concatenating the above geodesics from $\hat{x}_0$ (resp.\ $\hat{x}_0'$) to $\hat{z}$ and then from $\hat{z}$ to $\hat{x}_1$. Denoting by $\ell(\cdot)$ the length of a rectifiable curve, by construction we have $0\leq\ell(\hat{\gamma}) - \dd_{C(\hat{Y})}(\hat{x}_0, \hat{x}_1), \ell(\hat{\gamma}') - \dd_{C(\hat{Y})}(\hat{x}_0', \hat{x}_1) < 2\delta$. Therefore, by \eqref{Eq: distance to cone tip bound}, we have
\begin{align}\label{Eq: distance to cone tip bound 3}
\begin{split}
    \ell(\hat\gamma) &< \dd_{C(\hat{Y})}(\hat{x}_0, o_{C(\hat{Y})}) + \dd_{C(\hat{Y})}(o_{C(\hat{Y})}, \hat{x}_1)
    =\dd_{C(Y)}(x_0, o_{C(Y)}) + \dd_{C(Y)}(o_{C(Y)}, x_1),\\
   \ell(\hat\gamma')  &< \dd_{C(\hat{Y})}(\hat{x}_0', o_{C(\hat{Y})}) + \dd_{C(\hat{Y})}(o_{C(\hat{Y})}, \hat{x}_1)
   =\dd_{C(Y)}(x_0, o_{C(Y)}) + \dd_{C(Y)}(o_{C(Y)}, x_1),
\end{split}
\end{align}
provided that we chose $\delta$ sufficiently small.

Define the curve $\gamma:[0,2] \to C(Y)$ by
\begin{align}
\begin{split}
 \gamma(t):=\begin{cases}
    \pi^{C(\hat{Y})}_{C(Y)}(\hat\gamma(t))
	    &\text{if } t \in [0,1],\\
	\pi^{C(\hat{Y})}_{C(Y)}(\hat\gamma'(2-t)) &\text{if } t \in [1,2].
	\end{cases}
\end{split}
\end{align}
Since $\pi^{C(\hat{Y)}}_{C(Y)}$ is a local isometry away from $o_{C(\hat{Y})}$, we see that $\gamma$ is a Lipschitz curve, and $\ell(\gamma \lvert_{[0,1]}) = \ell(\hat{\gamma})$, $\ell(\gamma \lvert_{[1,2]}) = \ell(\hat{\gamma}')$. Moreover, $C(\hat{Y}) \setminus \{o_{C(\hat{Y})}\}$ is the orientable double cover of $C(Y) \setminus \{o_{C(Y)}\}$ as topological manifolds, so $\gamma$ is an orientation-reversing loop (see Definition~\ref{Def: orientation-reversing loop}) in $C(Y) \setminus \{o_{C(Y)}\}$ and hence in $A_\veps(C(Y))$ as well. 

From Proof II of \cite[Theorem 4.1]{BBP24}, one can construct loops $\gamma_i: [0,2] \to A_{\veps'}(X_i)$, where $\veps'$ can be made arbitrarily small by choosing $\veps$ sufficiently small, which are non-orientable for large $i$ and uniformly converge to $\gamma$ under the pmGH convergence realized by $Z$. More precisely, the loops are constructed for each $i$ by approximating $\gamma(k/2^i)$ for each $k = 0,\dots, 2^i$ by points $x_i^k \in X_i$ under the pmGH convergence and connecting consecutive points by geodesics. Since $\gamma$ is Lipschitz, $\gamma_i$ converges to $\gamma$ uniformly in $Z$. It is evident in the proof of \cite[Theorem 4.1]{BBP24} that uniform convergence of $\gamma_i$ to $\gamma$ is all that is needed. In particular, approximating exactly $2^i$ points is not important and the proof would also work by approximating $\gamma(k/m_i)$ for $k = 0,\dots, m_i$ by points in $X_i$ for any sequence $m_i \to \infty$. Therefore, by choosing $m_i \to \infty$ slowly enough, it is possible to construct $\gamma_i$ so that $\ell(\gamma_i \lvert_{[0,1]})$ (resp.\ $\ell(\gamma_i \lvert_{[1,2]})$) converges to $\ell(\gamma \lvert_{[0,1]})$ (resp.\ $\ell(\gamma \lvert_{[1,2]})$).

Define $x_0^i := \gamma_i(0) = \gamma_i(2)$ and $x_1^i := \gamma_i(1)$. By the uniform convergence of $\gamma_i$, we have that $x_0^i \to x_0 = \gamma(0)$ and  $x_1^i \to x_1 = \gamma(1)$ under the pmGH convergence realized by $Z$. 

Let $\delta > 0$ to be fixed later. Since $(x_0, x_1) \in \supp(\Pi_\infty)$, for any $i$ sufficiently large there exists some $(y^i_0, y^i_1) \in B_{\delta}(x_0^i) \times B_{\delta}(x_1^i)$ so that $(y^i_0, y^i_1) \in \supp(\Pi_i)$. By property (3) of Claim~\ref{Cla: Existence of sequence}, it holds that for any $\hat{y}^i_1 \in (\pi_i)^{-1}(y^i_1)$ we can choose $\hat{y}^i_0 \in (\pi_i)^{-1}(y^i_0)$ so that
 \begin{equation}\label{Eq: contradiction equation}
    \hat{\dd}_i(\hat{y}^i_0, \hat{y}^i_1) =  \hat\dd_{i}(\hat{y}^i_0, \hat{p}_i) +  \hat\dd_{i}(\hat{p}_i, \hat{y}^i_1) = \dd_{i}(y^i_0, p_i) +  \dd_{i}(p_i,y^i_1). 
 \end{equation}
Fix some $\hat{x}_1^i \in (\pi_i)^{-1}(x_1^i)$. Consider the curves $\gamma_i^1, \gamma_i^2:[0,1] \to X_i$ defined by $\gamma_i^1(t) := \gamma_i(t)$ and $\gamma_i^2(t) := \gamma_i(2-t)$. Since $\gamma_i([0,2]) \subseteq A_{\veps'}(X_i)$ and the restriction of $\pi_i$ to $(\pi_i)^{-1}(A_{\veps'}(X_i))$ is a local isometry (see Theorem~\ref{Thm: Ramified double cover existence}), there are unique lifts $\hat{\gamma}_i^1, \hat{\gamma}_i^2: [0,1] \to \hat{X}_i$ of $\gamma_i^1$, $\gamma_i^2$ so that $\hat{\gamma}_i^1(1), \hat{\gamma}_i^2(1) = \hat{x}_1^i$ and, moreover, $\ell(\gamma_i^1)=\ell(\hat{\gamma}_i^1)$, $\ell(\gamma_i^2)=\ell(\hat{\gamma}_i^2)$. Since $\gamma_i$ is orientation-reversing, we see that $\gamma_i^1(0), \gamma_i^2(0)$ must be the two distinct lifts of $x_0^i$.

 By Remark~\ref{Rem: Ramified double cover properties} (2), since $\dd_i(x_1^i, y_1^i) < \delta$, there must be some $\hat{y}^i_1 \in (\pi_i)^{-1}(y^i_1)$ so that $\hat{\dd}_i(\hat{x}_1^i, \hat{y}^i_1) < \delta$. As mentioned previously, we can choose some $\hat{y}^i_0 \in (\pi_i)^{-1}(y^i_0)$ so that \eqref{Eq: contradiction equation} holds. Finally, by Remark~\ref{Rem: Ramified double cover properties} (2), we can choose (at least) one of $\hat\gamma_i^1(0)$ and $\hat\gamma_i^2(0)$, which we denote by $\hat{x}_0^i$, so that $\hat\dd_i(\hat{x}_0^i, \hat{y}_0^i) < \delta$. By \eqref{Eq: contradiction equation}, we have 
 \begin{align*}
    \dd_i(x^i_0, p_i) + \dd_i(p_i, x^i_1)&= \hat{\dd}_i(\hat{x}^i_0, \hat{p}_i) +\hat{\dd}_i(\hat{p}_i, \hat{x}^i_1)\\
    &< \hat{\dd}_i(\hat{y}^i_0, \hat{p}_i) +\hat{\dd}_i(\hat{p}_i, \hat{y}^i_1) + 2\delta\\
    &= \hat{\dd}_i(\hat{y}^i_0,\hat{y}^i_1)+2\delta\\
    &< \hat{\dd}_i(\hat{x}^i_0, \hat{x}^i_1)+4\delta\\
    &\leq \max\{\ell(\gamma_i^1), \ell(\gamma_i^2)\}+4\delta.
 \end{align*}
 On the other hand, as $i \to \infty$ we have $\max\{\ell(\gamma_i^1), \ell(\gamma_i^2)\} \to \max\{\ell(\hat{\gamma}), \ell(\hat{\gamma}')\}$ and $\dd_i(x^i_0, p_i) + \dd_i(p_i, x^i_1) \to \dd_{C(Y)}(x_0, o_{C(Y)}) + \dd_{C(Y)}(o_{C(Y)}, x_1)$.
 Therefore, we have a contradiction with \eqref{Eq: distance to cone tip bound 3} when $\delta$ is sufficiently small and $i$ is sufficiently large. 
\end{proof}

By Claim~\ref{Cla: cone optimal plan support} and the fact that $(\pr_1)_*(\Pi_\infty)=\mu_\infty^0 \ll \mathcal{H}^3$, we can find $(x_0, x_1), (z_0, z_1) \in \supp(\Pi_\infty)$ so that
\begin{enumerate}
    \item $x_0, x_1, z_0, z_1 \in A_\veps(C(Y))$ with $\pi_Y^{C(Y)}(x_0) = \pi_Y^{C(Y)}(x_1)$ and $\pi_Y^{C(Y)}(z_0) = \pi_Y^{C(Y)}(z_1)$;
    \item $\pi_Y^{C(Y)}(x_0) \neq \pi_Y^{C(Y)}(z_0)$;
    \item $\dd_{C(Y)}(x_0, o_{C(Y)})=\dd_{C(Y)}(z_0, o_{C(Y)})$.
\end{enumerate}
Let $(\hat{x}_0^i, \hat{x}_1^i), (\hat{z}_0^i, \hat{z}_1^i) \in \supp(\hat{\Pi}_i)$ so that $\pi_i(\hat{x}_0^i), \pi_i(\hat{x}_1^i), \pi_i(\hat{z}_0^i), \pi_i(\hat{z}_1^i)$ converge to $x_0, x_1, z_0, z_1$ respectively under the pmGH convergence. Due to conditions (1) and (2) above, we can apply exactly the arguments of the previous proof to the pairs $(x_0, z_1)$ and $(z_0, x_1)$. The conclusion of the argument is that, for sufficiently large $i$, we have
\begin{align*}
    \hat{\dd}_i(\hat{x}_0^i, \hat{z}_1^i) &< \hat{\dd}_i(\hat{x}_0^i, \hat{p}_i)+\hat\dd_i(\hat{p}_i, \hat{z}_1^i) - 2\delta = \dd_i(\pi_i(\hat{x}_0^i),p_i) + \dd_i(p_i, \pi_i(\hat{z}_1^i)) - 2\delta,\\
    \hat{\dd}_i(\hat{z}_0^i, \hat{x}_1^i) &< \hat{\dd}_i(\hat{z}_0^i, \hat{p}_i)+\hat\dd_i(\hat{p}_i, \hat{x}_1^i) - 2\delta = \dd_i(\pi_i(\hat{z}_0^i),p_i) + \dd_i(p_i, \pi_i(\hat{x}_1^i)) - 2\delta,
\end{align*}
for some $\delta$ independent of $i$. Using condition (3), these can be rewritten as
\begin{align*}
    \hat{\dd}_i(\hat{x}_0^i, \hat{z}_1^i) &< \hat{\dd}_i(\hat{z}_0^i, \hat{p}_i)+\hat\dd_i(\hat{p}_i, \hat{z}_1^i) - \delta = \hat{\dd}_i(\hat{z}_0^i, \hat{z}_1^i) - \delta,\\
    \hat{\dd}_i(\hat{z}_0^i, \hat{x}_1^i) &< \hat{\dd}_i(\hat{x}_0^i, \hat{p}_i)+\hat\dd_i(\hat{p}_i, \hat{x}_1^i) - \delta = \hat{\dd}_i(\hat{x}_0^i, \hat{x}_1^i) - \delta
\end{align*}
for sufficiently large $i$, where we used the assumption that all geodesics in the support of $\hat\nu_i$ pass through $\hat p_i$.
Hence, the $\hat{\dd}_i^2$-cyclical monotonicity is violated for the pairs $(\hat{x}_0^i, \hat{x}_1^i), (\hat{z}_0^i, \hat{z}_1^i)$. In all, we have shown that Claim~\ref{Cla: Existence of sequence} is false, which means that Lemma~\ref{Lem: main lemma} must be true.

\section{Orbifold structure}

The main achievement of \cite{BPS24} was showing that
a non-collapsed $\RCD(K,3)$ space $X$ without boundary is a topological manifold
if and only if no point $x\in X$ admits a tangent cone whose cross-section is homeomorphic to $\mathbb{RP}^2$, see \cite[Theorem 1.8]{BPS24}.
In fact, the proof is entirely local, and gives the following.

\begin{Thm}\label{loc.reg}
    Let $(X,\dd,\mathcal{H}^3)$ be a non-collapsed $\rcd(K,3)$ space without boundary and $U\subseteq X$ an open set such that, for each $x\in U$,
    the cross-sections of all tangent cones at $x$ are homeomorphic to $\mathbb{S}^2$. Then $U$ is a topological manifold.
\end{Thm}

Recall that, by Remark~\ref{Rem: one tangent cone all tangent cone}, at each $x$ either all tangent cones have cross-section homeomorphic to $\mathbb{S}^2$ or all of them have cross-section homeomorphic to $\mathbb{RP}^2$.
Before detailing the simple adaptation required in the local version, let us recall
some properties of \emph{Green balls} $\mathbb{B}_r(p)$, an object which played a fundamental role in \cite{BPS24}.
To lighten notation, we will often write $\bar B_r(p)$ and $\bar{\mathbb{B}}_r(p)$
for the closures of $B_r(p)$ and $\mathbb{B}_r(p)$, respectively.

\begin{Pro}\label{green.balls}
    Given a non-collapsed $\rcd(-2,3)$ space $(X^3,\dd,\HH^3,p)$ with $\HH^3(B_1(p))\ge v>0$,
    there exist $\delta_0(v)\in(0,1)$, $C(v)>1$, a Borel set $\mathcal{G}_p\subset(0,\delta_0^2)$ and,
    for each $r\in\mathcal{G}_p$, an open set $\mathbb{B}_r(p)$ called the \emph{Green ball}
    of radius $r$ and center $p$, with the following properties:
    \begin{enumerate}
        \item $(r,Cr)\cap\mathcal{G}_p\neq\emptyset$
        for all $r\in(0,\delta_0^2/C)$;
        \item all values in $\{0\}\cup\mathcal{G}_p$ have density $1$ in $\mathcal{G}_p$
        (with respect to the measure $\mathcal{L}^1\lfloor (0,\infty)$);
        \item $B_s(p)$ is $\delta_0$-conical for all $s\in[\delta_0r,r/\delta_0]$
        (see \cite[Definition 3.11]{BPS24});
        \item $B_{r/2}(p)\subseteq \mathbb{B}_r(p)\subseteq B_{2r}(p)$;
        \item the \emph{Green sphere} $\mathbb{S}_r(p):=\partial\mathbb{B}_r(p)$
        is a topological surface, homeomorphic to either $\mathbb{S}^2$ or $\mathbb{RP}^2$
        (in fact, to the Alexandrov surface $\Sigma^2$ such that $B_r(p)$ is $\delta_0r$-GH close to $B_r(o)\subset C(\Sigma^2)$);
        \item for any $s\in(0,r/C)$ and $q\in\mathbb{S}_r(p)$, $\mathbb{S}_r(p)\cap B_s(q)$
        is 2-connected in $\mathbb{S}_r(p)\cap B_{Cs}(q)$;
        \item there exists a retraction $\rho:U\to\mathbb{S}_r(p)$,
        where $U$ is the $(r/C)$-neighborhood of the Green sphere, and moreover
        $$\dd(x,\rho(x))\le C\dd(x,\mathbb{S}_r(p))\quad\text{for all }x\in U;$$
        \item if $\mathbb{S}_r(p)$ is included in the topologically regular part of $X$,
        then it is tamely embedded there; in particular, it admits a neighborhood
        $V$ and a homeomorphism $h:V\to(-1,1)\times\mathbb{S}_r(p)$ such that
        $$h(V\cap\mathbb{B}_r(p))=(-1,0)\times\mathbb{S}_r(p),
        \quad h(V\setminus\bar{\mathbb{B}}_r(p))=(0,1)\times\mathbb{S}_r(p).$$
    \end{enumerate}
    In fact, each $\mathbb{B}_r(p)=\{b_p<r\}$ for a suitable local Green-type distance $b_p:B_{2r}(p)\to[0,\infty)$ (see \cite[Definition 4.3]{BPS24} and \cite[Remark 4.5]{BPS24}),
    while $\mathbb{S}_r(p)=\{b_p=r\}$.
\end{Pro}

\begin{proof}
    The set $\mathcal{G}_p$ is constructed as in \cite[Lemma 9.8]{BPS24}, guaranteeing in particular conclusions (1)--(3).
    Properties (4)--(6) are proved in \cite[Proposition 9.4]{BPS24}: in particular, (4) follows
    from the fact that $|b_p-\dd_p|<r/4$ on $B_{2r}(p)$,
    while (5) and (6) are precisely \cite[Proposition 9.4 (vi)]{BPS24} and \cite[Proposition 9.4 (v)]{BPS24}, respectively.
    Finally, exactly the same proof of \cite[Lemma 9.10]{BPS24} gives a retraction $\rho:U\to\mathbb{S}_r(p)$ with
    $$\dd(\rho(x),y)\le C(v)\dd(x,y)\quad\text{for all }x\in U,\ y\in\mathbb{S}_r(p),$$
    from which the desired bound in (7) follows. Property (8) is checked along the proof of \cite[Proposition 9.21]{BPS24}.
\end{proof}

\begin{Rem}
    In fact, property (6) implies that $\mathbb{S}_r(p)\cap B_s(q)$
        is contractible in $\mathbb{S}_r(p)\cap B_{Cs}(q)$ (for a possibly larger $C$), although we will not need this fact in the sequel.
        Indeed, we can triangulate the surface $\mathbb{S}_r(p)\cap B_s(q)$ and build a homotopy between the inclusion and the constant map (given by $q$), defining it inductively on the $k$-skeleton, for $k=0,1,2$,
        by exploiting the fact that maps $S^k\to\mathbb{S}_r(p)\cap B_s(q)$
        are nullhomotopic in $\mathbb{S}_r(p)\cap B_{Cs}(q)$,
\end{Rem}

\begin{proof}[Proof of Theorem~\ref{loc.reg}]
     Without loss of generality we can assume that $X$ is $\rcd(-2,3)$ and that the assumption on tangent cones holds on a superset $V\supseteq U$, with
    $$B_{10}(p)\subseteq V,\quad \mathcal{H}^3(B_1(p))\ge \nu>0$$
    for all $p\in U$. All the steps in the proof of \cite[Theorem 1.8]{BPS24} are local
    (with several quantitative statements depending only on the non-collapsing lower bound $\nu>0$).

    The only additional observation that we need in order to carry out the proof is the following:
    letting $\mathcal{R}_{\textrm{gm}^+}$ denote the (open) set defined in \cite[Definition 9.26]{BPS24}, assuming by contradiction that $U\not\subseteq\mathcal{R}_{\textrm{gm}^+}$, for $r>0$ we let
    $$\mathcal{S}_{\textrm{gm}^+}^r:=\{p\in U\setminus\mathcal{R}_{\textrm{gm}^+}\,:\,B_r(p)\subseteq U,\ B_s(p)\text{ is $\delta_0$-conical for all $0<s<r$}\},$$
    where the precise notion of $\delta$-conicality is given in \cite[Definition 3.11]{BPS24}.
    Each $\mathcal{S}_{gm^+}^r$ is closed in $U$ and, trivially,
    $$\mathcal{S}_{\textrm{gm}^+}:=U\setminus \mathcal{R}_{\textrm{gm}^+}
    =\bigcup_{r\in(0,1)}\mathcal{S}_{\textrm{gm}^+}^r.$$
    
    Although $U$ may not be a complete metric space with respect to the the metric $\dd$, it is a Polish space (namely, it is complete for another metric inducing the same topology, as it is an open set in a complete space: see, e.g., \cite[Example 6.1.11]{B07}). The same then holds for $\mathcal{S}_{\textrm{gm}^+}$, as it is closed in $U$.
    Thus, we can still apply Baire's category theorem and find $r\in(0,1)$
    such that $\mathcal{S}_{\textrm{gm}^+}^r$ (which is closed in $\mathcal{S}_{\textrm{gm}^+}$) has nonempty interior in $\mathcal{S}_{\textrm{gm}^+}$. This means that, up to rescaling, we have
    $$\mathcal{S}_{\textrm{gm}^+}\cap B_{10}(p_0)\subseteq\mathcal{S}_{\textrm{gm}^+}^{10}$$
    for some $p_0\in\mathcal{S}_{\textrm{gm}^+}$. The rest of the proof proceeds as in \cite{BPS24}.
\end{proof}

\begin{Def}\label{Def: P definition}
    Let $\mathcal{P}\subset X$ be the set of locally non-orientable points of $X$ as in Definition~\ref{Def: lno points}. 
\end{Def}

\begin{Rem}
    Once the orbifold structure of $X$ is proved, the set $\mathcal{P}$
    a posteriori consists exactly of the topological singularities
    (as, removing the cone tip $o$, we have $C(\mathbb{RP}^2)\setminus\{o\}\cong\R\times\mathbb{RP}^2$, which is non-orientable). We will also see through our arguments in this section that $\mathcal{P}$ consists exactly of those points whose tangent cone cross-sections are homeomorphic to $\mathbb{RP}^2$. Note that it follows from Proposition~\ref{Lem: non-orientable tangent cone implies fixed point} that all such points are contained in $\mathcal{P}$. 
\end{Rem}

The following lemma characterizes $p \in \mathcal{P}$ in the case that it is an isolated point.
\begin{Lem}\label{Lem: isolated lno}
Let $(X, \dd, \HH^3)$ be a non-collapsed $\RCD(-2,3)$ space without boundary and assume that $p \in \mathcal{P}$. If $p$ is isolated in $\mathcal{P}$, i.e., there exists $r>0$ so that $B_{r}(p) \cap \mathcal{P}=\{p\}$, then the cross-sections of all tangent cones at $p$ are homeomorphic to $\mathbb{RP}^2$. 
\end{Lem}
\begin{proof}
Assume that this is not the case. Then the cross-sections of all tangent cones at $p$ are homeomorphic to $\mathbb{S}^2$ by Remark~\ref{Rem: one tangent cone all tangent cone}. Moreover, any $x \in B_{r}(x) \cap \mathcal{P}$ does not admit a tangent cone whose cross-section is homeomorphic to $\mathbb{RP}^2$ either, since any such point is in $\mathcal{P}$ by Proposition~\ref{Lem: non-orientable tangent cone implies fixed point}. Therefore, Theorem~\ref{loc.reg} applies and $B_r(p)$ is a topological manifold. This implies that $p$ is locally orientable, which is a contradiction. 
\end{proof}

We have the following stability theorem for locally non-orientable points in the case where they are locally finite. 
\begin{Thm}\label{Thm: P stability}
    Let $(X_i, \dd_i, \HH^3, p_i)_{i \in \N}$ be a sequence of non-orientable $\rcd(-2,3)$ spaces without boundary converging to a non-collapsed $\rcd(-2,3)$ space $(X, \dd, \HH^3, p)$ in the pmGH sense. Let $\mathcal{P}_i$ (resp.\ $\mathcal{P})$ be the set of locally non-orientable points of $X_i$ (resp.\ $X$). If $p_i \in \mathcal{P}_i$ and $\mathcal{P}_i \cap B_{1}(p_i)$ is locally finite in $B_1(p_i)$ for every $i \in \N$, then $p \in \mathcal{P}$. In particular, $X$ is non-orientable. 
\end{Thm}

\begin{proof}
By \cite[Theorem 1.6]{BNS22}, the limit space $(X,\dd,\HH^3)$ has no boundary as well.
It suffices to prove that $B_r(p)$ is non-orientable for any $r < 1/10^5$. Let $(\hat{X}_i, \hat\dd_i, \HH^n, \hat{p}_i)$ be the ramified double cover of $X_i$, with $\pi_i$ and $\Gamma_i$ the associated projection and isometric involution maps respectively. Thanks to Theorem~\ref{Thm: stability of non-orientiable local cde case}, we need only to prove that each $\hat{X}_i$ is locally $\text{CD}^e(-2, 3)$ on $B_{1/10}(\hat{p}_i)$ for each $i$ to conclude. 

Fix any $i \in \N$. We will show that $\hat{X}_i$ is locally strongly $\CD^e(-2,3)$ at every point in $B_{1/2}(\hat{p}_i) = \pi^{-1}(B_{1/2}(p_i))$, where the equality follows from Proposition~\ref{Pro: LNO classification} and Remark~\ref{Rem: Ramified double cover properties}. We first consider the lifts of locally orientable points. Let $x \in B_{1/2}(p_i) \setminus \mathcal{P}_i$ and choose any $\hat{x} \in \pi^{-1}(x)$. Then $\Gamma_i(\hat{x}) \neq \hat{x}$ by Proposition~\ref{Pro: LNO classification}. Let $r := \hat\dd_i(\Gamma_i(\hat{x}), \hat{x})$. By Remark~\ref{Rem: Ramified double cover properties} (2) and the fact that $\Gamma(B_{r/4}(\hat x))$ has distance at least $r/2$ from $B_{r/4}(\hat x)$, we see that $\pi$ is an isometry from $B_{r/4}(\hat{x})$ to $B_{r/4}(x)$. As a consequence, $\pi: (B_{r/4}(\hat{x}), \hat\dd_i, \HH^3) \to (B_{r/4}(x), \dd_i, \HH^3)$ is a metric measure isomorphism. Since $B_{r/4}(x)$ is a subset of an $\RCD(-2,3)$ space, we conclude that $\hat{X}_i$ is locally strongly $\CD^e(-2,3)$ on $B_{r/8}(\hat{x})$.

Next, we consider the locally non-orientable points. Let $x \in B_{1/2}(p_i) \cap \mathcal{P}_i$ and let $\hat{x} \in \pi^{-1}(x)$ be the unique lift. Since $\mathcal{P}_i$ is locally finite in $B_{1}(p_i)$, there exists $r > 0$ so that $B_{r}(x) \cap \mathcal{P}_i = \{x\}$. By Lemma~\ref{Lem: isolated lno}, all tangent cones at $x$ have cross-sections homeomorphic to $\mathbb{RP}^2$. Consider $\overline{B_{r/2}(\hat{x})} = \pi^{-1}(\overline{B_{r/2}(x)})$. If $\hat{y} \in \overline{B_{r/2}(\hat{x})}$ and $\hat{y} \neq \hat{x}$, then $\hat{y} \in \pi^{-1}(y)$ for some $y \notin \mathcal{P}_i$. Using the same argument as with the first case, we have that $\hat{X}_i$ is locally strongly $\CD^e(-2,3)$ on $B_s(\hat{y})$ for some $s > 0$. Applying Lemma~\ref{Lem: main lemma} and the globalization Theorem~\ref{Thm: main globalization theorem}, we conclude that $\hat{X}_i$ is locally strongly $\CD^e(-2,3)$ on $B_{r/4}(x)$. 

At this point, we have shown that, for every $x \in B_{1/2}(\hat{p}_i)$, there exists $r>0$ so that $\hat{X}_i$ is locally strongly $\CD^e(-2,3)$ on $B_r(x)$. Using a variation of the same globalization theorem, where in this case one assumes that all points have a neighborhood that is locally $\CD^e(-2,3)$ (the proof of Theorem~\ref{Thm: main globalization theorem} works under this assumption: cf.\ \cite[Theorem 3.14]{EKS15}), we conclude that $\hat{X}_i$ is locally strongly $\CD^e(-2,3)$ on $B_{1/10}(\hat{p}_i)$ as desired. 
\end{proof}

Continuing with our local finiteness assumption on $\mathcal{P}$, we can now prove that the ramified double cover is an $\RCD$ space. At this point, we recall that a metric measure space $(X, \dd, \mm)$ is $\RCD(K,N)$ if and only if it is $\CD(K,N)$ and infinitesimally Hilbertian. A metric measure space is said to be \textit{infinitesimally Hilbertian} according to \cite[Definition 4.19]{G15} if the normed Sobolev space $(W^{1,2}(X), \norm{\cdot}{1,2})$ is a Hilbert space, i.e., $\norm{\cdot}{1,2}$ satisfies the parallelogram law. Here $(W^{1,2}(X), \norm{\cdot}{1,2})$ is defined as in \cite{AGS14}, as a variant of \cite{C99} (see, for instance, \cite[Section 2]{NPS25} for a concise summary).
\begin{Lem}\label{Lem: ramified double cover is RCD LF version}
Let $(X, \dd, \HH^3)$ be a non-orientable $\RCD(-2,3)$ space without boundary and let $(\hat{X}, \hat\dd, \HH^3)$ be its ramified double cover. If $\mathcal{P}$ is locally finite then $\hat{X}$ is an orientable non-collapsed $\RCD(-2,3)$ space without boundary and a topological manifold.  
\end{Lem}
\begin{proof}
Let $\pi: \hat{X} \to X$ denote the associated projection map. It follows from the proof of Theorem~\ref{Thm: P stability} that $\hat{X}$ is locally strongly $\CD^e(-2,3)$ at each $x \in X$. Applying the globalization theorem \cite[Theorem 3.14]{EKS15}, we see that $\hat{X}$ satisfies the strong $\CD^e(-2,3)$ condition. It is known that strong $\CD^e(-2,3)$ spaces are essentially non-branching by \cite{EKS15} (cf.\ Lemma~\ref{Lem: cde implies essential non-branching}). It follows that $\hat{X}$ satisfies the $\CD^{*}(-2,3)$ condition by \cite[Theorem 3.12]{EKS15}, and hence the $\CD(-2,3)$ condition by \cite[Corollary 13.6]{CM21}. 

We now show that $\hat{X}$ is infinitesimally Hilbertian:
this would follow easily from what we have proven so far about the tangent cones ($\HH^3$-a.e.) of $(\hat{X}, \hat\dd, \HH^3)$ and \cite[Theorem 1.2]{NPS25} but, since there is no need for the full generality of \cite{NPS25} in our situation, we also sketch a proof using basic nonsmooth calculus. It suffices to prove the parallelogram law for any $f, g \in W^{1,2}(X)$, i.e.,
\begin{equation*}
    \norm{f+g}{1,2}^2+\norm{f-g}{1,2}^2=2\norm{f}{1,2}^2+2\norm{g}{1,2}^2. 
\end{equation*} 
As we have shown before, Remark~\ref{Rem: Ramified double cover properties} and Proposition~\ref{Pro: LNO classification} imply that $\pi$ yields a local metric measure isomorphism $\pi^{-1}(X \setminus \mathcal{P}) \to X \setminus \mathcal{P}$. Since $X$ is $\RCD(-2,3)$, it follows that, for any $\hat{x} \in \pi^{-1}(X \setminus \mathcal{P})$, the parallelogram law is satisfied for any $f, g \in W^{1,2}(X)$ supported in some open neighborhood of $\hat{x}$. It follows by a partition of unity argument, using the properties of the minimal relaxed gradient \cite[Section 4.1]{AGS14}, that any pair $f, g \in W^{1,2}(X)$ supported in $X \setminus \mathcal{P}$ satisfies the parallelogram law. Therefore, to prove the parallelogram law for any general $f,g \in W^{1,2}(X)$, it suffices to show that any $f \in W^{1,2}(X)$ can be approximated in $W^{1,2}$ by functions supported in $X \setminus \mathcal{P}$. This is equivalent to asking that $\mathcal{P}$ has null $2$-capacity. Since $\mathcal{P}$ is locally finite, one can show this directly by using the distance functions to the points $\hat{p} \in \mathcal{P}$, along with 
the fact that $\limsup_{r \to 0^+} \frac{\HH^3(B_r(\hat{p}))}{\omega_3 r^3} \leq 2$ (by Remark~\ref{Rem: Ramified double cover properties} (2)--(3) and the fact that $X$ is non-collapsed $\RCD(-2,3)$).

Next, we show that $\hat{X}$ has no boundary. Assume otherwise and let $\partial \hat{X}$ denote the boundary. By \cite[Theorem 1.2]{BNS22}, $\mathcal{H}^{2}(\partial \hat{X}) > 0$ if $\partial \hat{X} \neq \emptyset$. As before, $\pi: \pi^{-1}(X \setminus \mathcal{P}) \to X \setminus \mathcal{P}$ is a local metric measure isomorphism. Combining this with the assumption that $X$ has no boundary, we see that $\partial \hat{X} \cap \pi^{-1}(X \setminus \mathcal{P}) = \emptyset$. As such, we have $\partial \hat{X} \subseteq \pi^{-1}(\mathcal{P})$,
which is at most countable, contradicting $\mathcal{H}^{2}(\partial \hat{X}) > 0$.

The orientability of $\hat{X}$ now follows directly from Theorem~\ref{Thm: Ramified double cover existence} (2) and Proposition~\ref{Pro: Equivalent definition of orientability}. The stability of orientability (Theorem~\ref{Thm: stability of orientability}) gives that every tangent cone of (any point in) $\hat{X}$ has cross-section homeomorphic to $\mathbb{S}^2$, which by \cite[Theorem 1.8]{BPS24} proves that $\hat{X}$ is a topological manifold. 
\end{proof} 

We are now in a position to prove our main orbifold structure theorem.

\begin{Thm}\label{Thm: orbifold structure thm}(Orbifold structure theorem).
    Let $(X, \dd, \HH^3)$ be a non-collapsed $\RCD(-2,3)$ space without boundary. Then $X$ is an orbifold and, denoting $\mathcal{P}$ the set of its locally non-orientable points, we have:
    \begin{enumerate}
        \item $\mathcal{P}$ is locally finite;
        \item $X\setminus\mathcal{P}$ is a topological manifold;
        \item each $x\in\mathcal{P}$ has a neighborhood homeomorphic to $C(\mathbb{RP}^2)$.
    \end{enumerate}
\end{Thm}

We begin with the following observation. 

\begin{Pro}\label{Pro: orbifold struture thm up to local finiteness}
    Theorem~\ref{Thm: orbifold structure thm} follows from the local finiteness of $\mathcal{P}$
\end{Pro}

\begin{proof}

    Assume that $\mathcal{P}$ is locally finite. By Proposition~\ref{Lem: non-orientable tangent cone implies fixed point} and Lemma~\ref{Lem: isolated lno}, we have that $\mathcal{P}$ consists precisely of those points whose tangent cones all have cross-sections homeomorphic to $\mathbb{RP}^2$. It follows that $X\setminus\mathcal{P}$ is a topological manifold by Theorem~\ref{loc.reg}. Moreover, $X$ admits a ramified double cover $(\hat{X}, \hat\dd, \HH^3)$ that is a non-collapsed orientable $\RCD(-2,3)$ space without boundary and a topological manifold thanks to Lemma~\ref{Lem: ramified double cover is RCD LF version}. Denote by $\pi: \hat{X} \to X$ the associated projection map. Since $\hat{X}$ is orientable, it follows from the stability of orientability (Theorem~\ref{Thm: stability of orientability}) that every tangent cone of (any point in) $\hat{X}$ has cross-section homeomorphic to $\mathbb{S}^2$. 

    Now, fix some $x\in\mathcal{P}$ and let $\hat{x} \in \pi^{-1}(x)$ be its unique lift, and fix
    a decreasing sequence of radii $R_k\to0$
    such that ${{B}}_{R_1}(x)\cap\mathcal{P}=\{x\}$
    and $B_{s}(x),B_{s}(\hat x)$ are $\delta_0$-conical for all $s\in[C^{-1}\delta_0R_k,R_k/\delta_0]$.
    Applying \cite[Theorem 4.7]{BPS24}, we find a good Green-type distance $b_k:B_{R_k}(x)\to[0,\infty)$ (see \cite[Remark 4.11]{BPS24}).
    Since $\pi$ is a (two-sheeted) local metric measure isomorphism on
    ${{B}}_{R_k}(\hat x)\setminus\{\hat x\}$, we deduce that the lift
    $$\hat b_k:=b_k\circ\pi:B_{R_k}(\hat x)\to[0,\infty)$$
    still satisfies the conclusions of \cite[Theorem 4.7]{BPS24} (recall that $\pi^{-1}(B_{R_k}(x))=B_{R_k}(\hat x)$ by Remark~\ref{Rem: Ramified double cover properties} and $\Gamma(\hat x)=\hat x$).
    By construction of $\mathcal{G}_x$ (which is built using \cite[Theorem 5.4]{BPS24}), we can now take
    $$r_k\in(R_k/C(v),R_k/2)\cap\mathcal{G}_x\cap\mathcal{G}_{\hat x}$$
    such that the conclusions of Proposition~\ref{green.balls} hold for the Green balls
    $$\mathbb{B}_{r_k}(x):=\{b_k<r_k\},\quad \mathbb{B}_{r_k}(\hat x):=\{\hat b_k<r_k\}.$$
    
    By Proposition~\ref{green.balls} (5), we know that
    $$\mathbb{S}_{r_k}(x)\cong\mathbb{RP}^2,\quad\mathbb{S}_{r_k}(\hat x)\cong\mathbb{S}^2.$$
    By \cite[Proposition 9.15]{BPS24}, $\bar{\mathbb{B}}_{r_k}(\hat x)$ is simply connected.
    Assuming without loss of generality $r_k>r_{k+1}$, a direct application of Van Kampen's theorem and
    Proposition~\ref{green.balls} (8) shows that
     $$\hat A_k:=\bar{\mathbb{B}}_{r_k}(\hat x)\setminus \mathbb{B}_{r_{k+1}}(\hat x)$$
     is simply connected as well. It easily follows that
     $$\hat A_k\cong[0,1]\times \mathbb{S}^2.$$
     Indeed, by capping $\hat A_k$ off with two copies of $\bar B^3$ (glued according to the
     homeomorphisms $\partial\mathbb{B}_{r_k}(\hat x),\partial\mathbb{B}_{r_{k+1}}(\hat x)\cong \mathbb{S}^2$), we obtain a simply connected
     manifold $M_k$, which is homeomorphic to $\mathbb{S}^3$ by the solution to the Poincar\'e conjecture.
     We can endow $M_k$ with a smooth structure (diffeomorphic to $\mathbb{S}^3$) with respect to which
     the boundary components $S_k',S_k''$ of $\hat A_k\subset M_k\cong \mathbb{S}^3$ are two embedded copies of $\mathbb{S}^2$.
     Thus, $\hat A_k$ is diffeomorphic to $\mathbb{S}^3$ with two balls removed;
     since any smooth embedding $\bar B^3\hookrightarrow \mathbb{S}^3$ is isotopic to a spherical cap, the claim follows.
     
     Since $\hat A_k$ covers $A_k:=\bar{\mathbb{B}}_{r_k}(x)\setminus\mathbb{B}_{r_{k+1}}(x)$,
     we can apply \cite[Theorem 1]{L63} to obtain that
     $$A_k\cong[0,1]\times\mathbb{RP}^2\cong[2^{-k-1},2^{-k}]\times\mathbb{RP}^2,$$
     with $\partial\bar{\mathbb{B}}_{r_k}(x)$ corresponding to $\{2^{-k}\}\times\mathbb{RP}^2$
     and $\partial\bar{\mathbb{B}}_{r_{k+1}}(x)$ corresponding to $\{2^{-k-1}\}\times\mathbb{RP}^2$.
     Pasting together these homeomorphisms, we see that
     $\bar{\mathbb{B}}_{r_1}(x)$ is homeomorphic to (the compact version of) the cone over $\mathbb{RP}^2$.
\end{proof}

In order to show the local finiteness of $\mathcal{P}$, let us start with two simple observations.

\begin{Lem}\label{cone.ok}
    Given $\theta_0>0$, there exists $\gamma>0$ with the following property:
    Let $(X,\dd,\mathcal{H}^3)$ be any non-collapsed $\rcd(-2,3)$ space without boundary such that $\mathcal{H}^3(B_r(p))\ge\theta_0r^3/3$ for some ball $B_r(p)\subseteq X$ with $r\le1$.
    Moreover, assume that there exists $q\in \mathcal{P}$ such that $\dd(p,q)=r$ and
    $$\mathcal{P}\cap B_{r/2}(q)\text{ is locally finite in }B_{r/2}(q).$$
    Then, for any non-collapsed $\rcd(0,3)$ space $Y=C(\Sigma^2)$, where $\Sigma^2$ is a $2$-dimensional Alexandrov space with $\curv \geq 1$, we must have
    $$d_{GH}(B_{2r}(p),B_{2r}(o))>\gamma\cdot 2r,$$
    where $B_{2r}(o)\subset Y$ is the ball of radius $2r$ centered at the cone tip of $Y$.
\end{Lem}

\begin{proof}
    Up to rescaling, we can assume $r=1$. The lemma then follows immediately by a compactness argument,
    using Theorem~\ref{Thm: P stability} and the fact that $\mathcal{P}(Y)\subseteq\{o\}$
    for any $Y=C(\Sigma^2)$, as $\Sigma^2$ is an Alexandrov space and thus a topological surface, and hence $Y\setminus\{0\}$ is a topological manifold.
\end{proof}

Recall that, for any point $p$ in a non-collapsed $\RCD(K,n)$ space, the \textit{density} at $p$ is defined as
\begin{equation*}
    \theta(p) := \HH^{n-1}(\Sigma_p),
\end{equation*}
where $\Sigma_p$ is the cross-section of a tangent cone at $p$. It is not difficult to check that $\theta_p$ is well-defined (i.e., independent of the choice of cross-section) by the Bishop--Gromov inequality and \cite{DPG18}.

\begin{Lem}\label{proj.dens}
    Given $\theta_0>0$, there exists $\delta>0$ with the following property. Let $(\Sigma^2, \dd_{\Sigma^2}, \HH^2)$ be any $2$-dimensional Alexandrov space with $\curv \geq 1$ that is homeomorphic to $\mathbb{RP}^2$ and let $(Y, \dd_Y, \HH^3, o)$ be the metric measure cone over $\Sigma^2$, where $o$ is the cone tip. If $\mathcal{H}^2(\Sigma^2) \geq \theta_0$, or equivalently $\mathcal{H}^3(B_1(o))\ge\theta_0/3$, then the density $\theta(q) \geq \theta(o) + \delta$ for any $q \in Y \setminus \{o\}$.
\end{Lem}

\begin{proof}
We note that the density is the same along any radial line of the cone away from $\{o\}$. Therefore, if the claim is false then we can find a sequence of cones $(Y_i,\dd_i,o_i)$ and points $q_i \in Y_i$ with $\dd_{i}(q_i, o_i) = 1$,
so that $\mathcal{H}^3(B_r(o_i))=\theta(o_i)r^3/3\ge\theta_0r^3/3$ for all $i$
and $\theta(q_i)\le\theta(o_i)+\delta_i$, with $\delta_i\to0$.
Up to a subsequence, we could find a limit cone $(Y_\infty, \dd_\infty, o_\infty)$ and a point $q_\infty \in Y_\infty$ so that $q_i \to q_\infty$ under the pmGH convergence. By lower semi-continuity of density,
\begin{equation*}
\theta(q_\infty)\le\bar\theta:=\lim_{i\to\infty}\theta(q_i)=\lim_{i\to\infty}\theta(o_i)
\end{equation*}
(where the limit exists up to a subsequence, as $\theta(o_i)\in[\theta_0,4\pi]$),
while $B_r(o_\infty)$ has volume $\bar\theta r^3/3$ for all $r>0$. Hence, $Y_\infty$ must split a factor $\R$.
Its cross-section is then homeomorphic to $\mathbb{S}^2$;
by uniform local contractibility of Alexandrov surfaces and \cite{P90}, the cross-section
$\Sigma_i^2$ of $Y_i$ must be homotopically equivalent to $\mathbb{S}^2$ for $i$ large enough, a contradiction.
\end{proof}

We are now ready to show that $\mathcal{P}$ is locally finite, which will immediately imply the orbifold structure theorem (Theorem~\ref{Thm: orbifold structure thm}) thanks to Proposition~\ref{Pro: orbifold struture thm up to local finiteness}. Let $\mathcal{P}'$ be the set of accumulation points of the set $\mathcal{P}$.
Since $\mathcal{P}$ is closed by definition, we have $\mathcal{P}'\subseteq\mathcal{P}$. The local finiteness of $\mathcal{P}$ is clearly equivalent to $\mathcal{P}' = \emptyset$.

\begin{Pro}
 Let $(X, \dd, \HH^3)$ be a non-collapsed $\RCD(-2,3)$ space without boundary and let $\mathcal{P}$ be its set of locally non-orientable points. Then $\mathcal{P}$ is locally finite in $X$.
\end{Pro}

\begin{proof}
Let $U \subseteq X$ be bounded and open. By the lower semi-continuity of density there exists some $\theta_0 > 0$ so that $\theta(x)>\theta_0$ for all $x\in U$.
Let $K\in\N$ be the smallest integer such that $\theta_0+K\delta>4\pi$, where $\delta>0$ is given by Lemma~\ref{proj.dens}.

By induction on $k=0,1,\dots,K$, we will show that $\mathcal{P}$ has no accumulation points in 
\begin{equation*}
U_k:=\{x\in U\,:\,\theta(x)>\theta_k\},\quad\theta_k:=\theta_0+(K-k)\delta.
\end{equation*}
Clearly, this is enough to show that $\mathcal{P}$ is locally finite in $X$ since $U=U_K$ is arbitrary. We note that $U_k$ is open by the lower semi-continuity of density. The base case is trivial since $U_0=\emptyset$, as no point can have density larger than $4\pi$.

Assuming now that the claim holds for $k$, let us prove it for $k+1$. Assume by contradiction that
\begin{equation*}
\mathcal{P}'\cap U_{k+1}\neq\emptyset.
\end{equation*}
We claim that none of the points in $\mathcal{P}'\cap U_{k+1}$ have a tangent cone with cross-section homeomorphic to $\mathbb{RP}^2$. Indeed, if this is not the case then we can find a sequence of distinct $x_i\in\mathcal{P}$
converging to $x \in\mathcal{P}'\cap U_{k+1}$, a point where (some, and so by Remark~\ref{Rem: one tangent cone all tangent cone}) all tangent cones have cross-section homeomorphic to $\mathbb{RP}^2$. Let $\rho_i:=d(x_i,x)>0$
and consider the pointed spaces $X_i$ obtained by dilating $X$ by a factor $\rho_i^{-1}$ (with reference point $x$). Up to a subsequence, these converge in the pmGH sense to a limit cone $(C(\Sigma^2), \dd_{C(\Sigma^2)}, \HH^3, o)$, where $o$ is the cone tip, with $\Sigma^2 \cong \mathbb{RP}^2$ and $\theta(o)=\theta(x)$. Calling $\tilde x_i\in X_i$ the point corresponding to $x_i\in X$, we may assume that, up to a subsequence, $\tilde x_i$ converges to some $\tilde{x} \in C(\Sigma^2)$ with $\dd_{C(\Sigma^2)}(\tilde{x},o) = 1$. By Lemma~\ref{proj.dens},
we obtain
\begin{equation*}
\theta_{k+1}<\theta(x)=\theta(o)\le\theta(\tilde{x})-\delta\le\liminf_{i\to\infty}\theta(\tilde{x}_i)-\delta.
\end{equation*}
Since $\theta(\tilde{x}_i)=\theta(x_i)$ we deduce that, for sufficiently large $i$,
\begin{equation*}
\theta(x_i)>\theta_{k+1}+\delta=\theta_k.
\end{equation*}
In fact, we can repeat the same argument for any sequence of points $x_i'\in B_{\rho_i/2}(x_i)$.
Since the corresponding points in the rescaled spaces $X_i$ belong to $B_{1/2}(\tilde x_i)$, they converge up to a subsequence to a point $\tilde{x}'\in\bar B_{1/2}(\tilde{x})\subseteq Y\setminus\{o\}$.
We deduce that
$$B_{\rho_i/2}(x_i)\subseteq U_k$$
for $i$ large enough. Hence, by inductive assumption, we are in the hypotheses of Lemma~\ref{cone.ok},
which then yields a contradiction for $i$ large enough.

Now, taking $\gamma>0$ as in Lemma~\ref{cone.ok},
recalling the terminology of \cite[Definition 3.11]{BPS24}
we say that a ball $B_{r}(x)\subseteq X$ with $r\in(0,\gamma)$ is \emph{$\gamma$-conical} if
\begin{equation*}
d_{GH}(B_{r}(x),B_{r}(o))\le\gamma r\quad\text{for some non-collapsed $\rcd(0,3)$ cone }C(\Sigma^2),
\end{equation*}
where as usual $B_{r}(o)$ denotes the ball centered at the tip $o\in C(\Sigma^2)$.
We let
$$\mathcal{P}_r':=\{x\in\mathcal{P}'\,:\,B_{s}(x)\text{ is $\gamma$-conical for all }s\in(0,r)\}.$$
It is clear that $\mathcal{P}_r'\subset X$ is a closed set (as $\mathcal{P}'$ is closed).
Since $U_{k+1}$ is open in $X$, it is a Polish space (i.e., $U$ is homeomorphic to a complete metric space:
see \cite[Example 6.1.11]{B07}). Since $\mathcal{P}'\cap U_{k+1}$ is closed in $U_{k+1}$,
the topological space $\mathcal{P}'\cap U_{k+1}$
is itself a Polish space, in which each $\mathcal{P}_r'\cap U_{k+1}$ is closed.
Since
$$\bigcup_{k=1}^{\infty}(\mathcal{P}_{2^{-k}\gamma}'\cap U_{k+1})=\mathcal{P}'\cap U_{k+1}\neq\emptyset,$$
by Baire's lemma there exists $r\in(0,\gamma)$ such that $\mathcal{P}_r'\cap U_{k+1}$ has nonempty interior, relative to $\mathcal{P}'\cap U_{k+1}$.

In other words, for this fixed $r$, there exist $p\in \mathcal{P}'\cap U_{k+1}$ and $\rho\in(0,r)$
such that
$$B_\rho(p)\subseteq U_{k+1},\quad \mathcal{P}'\cap B_\rho(p)\subseteq\mathcal{P}_r'.$$
Since $p\in\mathcal{P}' \subseteq \mathcal{P}$, there exists $q\in\mathcal{P}$ with $\dd(p,q)<\rho/2$ such that any tangent cone at $q$ has cross-section homeomorphic to $\mathbb{RP}^2$. Indeed, if this is not possible then $B_{\rho/2}(p)$ would be a topological manifold by Theorem~\ref{loc.reg} and so $p$ would be locally orientable. 

Now let $\bar p\in\mathcal{P}'$ be a closest point to $q$ in $\mathcal{P}'$,
so that $\sigma:=d(\bar p,q)\le d(p,q)<\rho/2$. Since $q$ has tangent cones with cross-sections homeomorphic to $\mathbb{RP}^2$, as proved above we have $q\notin\mathcal{P}'$ and so $\sigma>0$.
Also,
$$\bar p\in \mathcal{P}'\cap B_\rho(p)\subseteq\mathcal{P}_r'.$$
By definition of closest point, we have
$$B_{\sigma/2}(q)\cap\mathcal{P}'=\emptyset;$$
in other words, $\mathcal{P}$ is locally finite in $B_{\sigma/2}(q)$.
Since
$$\sigma<\rho/2<r/2,$$
the ball $B_{2\sigma}(\bar p)$ is $\gamma$-conical. However, this contradicts Lemma~\ref{cone.ok}.
\end{proof}

As corollaries, we immediately obtain Theorem~\ref{Thm: orbifold structure thm} and the following structure and stability theorems for the ramified double cover. 

\begin{Thm}\label{Thm: ramified double cover is RCD} (Regularity of ramified double cover).
Let $(X, \dd, \HH^3)$ be a non-orientable $\RCD(-2,3)$ space without boundary. Then the ramified double cover $(\hat{X}, \hat\dd, \HH^3)$ is an orientable $\RCD(-2,3)$ space without boundary. 
\end{Thm}

\begin{Thm}\label{Thm: ramified double cover stability} (Stability of non-orientability and ramified double cover).
Consider a sequence $(X_i, \dd_i, \HH^3, p_i)_{i\in \N}$ of non-orientable $\RCD(-2,3)$ spaces without boundary converging in the pmGH sense to an $\RCD(-2,3)$ space without boundary $(X, \dd, \HH^3, p)$. If there exists $R> 0$ such that $B_{R}(p_i)$ is non-orientable for all $i \in \N$, then $B_{R'}(p)$ is non-orientable for all $R' > R$. Moreover, denoting by $(\hat{X}_i, \hat\dd_i, \HH^3, \hat{p}_i)$ the ramified double covers of $X_i$, it holds that $(\hat{X}_i, \hat\dd_i, \HH^3, \hat{p}_i)$ converges in the pmGH sense to $(\hat{X}, \hat\dd, \HH^3, \hat{p})$, where the latter is the ramified double cover of $X$ and $\pi(\hat{p}) = p$ (with $\pi: \hat{X} \to X$ being the associated projection map).
\end{Thm}

The first theorem follows from Lemma~\ref{Lem: ramified double cover is RCD LF version}, and the second theorem follows from \cite[Theorem 4.2]{BBP24} and the first. 


As another consequence,
we have the following classification theorem for non-collapsed $\RCD(K,3)$ spaces without boundary when $K > 0$. 

\begin{Cor}\label{Cor: topological classification}
    Let $K > 0$ and let $(X, \dd, \HH^3)$ be a non-collapsed $\rcd(K,3)$ space without boundary. Then one of the following holds:
    \begin{enumerate}
     \item $X$ is a spherical $3$-manifold, i.e., $X$ is an orientable topological manifold homeomorphic to $\mathbb{S}^3/\Gamma$, where $\Gamma < SO(4)$ is a finite subgroup acting freely by rotations on $\mathbb{S}^3$;
     \item $X$ is the spherical suspension over $\mathbb{RP}^2$ and thus it is non-orientable.
 \end{enumerate}
\end{Cor}

\begin{proof}
    From the proof of Theorem~\ref{Thm: orbifold structure thm}, $X$ admits a ramified double cover $\hat X$ which is an orientable $\RCD(K,3)$ topological manifold. Moreover, since $K>0$, both $X$ and $\hat X$ are compact, with finite fundamental group:
    indeed, both have $\RCD(K,3)$ universal covers by \cite{MW19} and \cite{W24},
    with the latter being compact for both spaces as $K>0$. Let $\mathcal{P}=\{p_1,\dots,p_N\}$ be the set of locally non-orientable points of $X$ and consider a small good Green ball $B_i$ around each $p_i$.

    As seen in the proof of Proposition~\ref{Pro: orbifold struture thm up to local finiteness},
    we can assume that each $B_i$ lifts to a Green ball $\hat B_i\subset\hat X$.
    By Proposition~\ref{green.balls} (8) and Van Kampen's theorem,
    $\pi_1(\hat X\setminus\bigcup_i \hat B_i)$
    is also finite. The same then holds for
    $\pi_1(X\setminus\bigcup_i B_i)$, since $\pi:\hat X\to X$ is a covering map on
    the preimage of $X\setminus\bigcup_i B_i$.
    
    Since $\partial B_i\cong\mathbb{RP}^2$ by Proposition~\ref{green.balls} (5), if $N\ge1$ then
    $X\setminus\bigcup_i B_i$ cannot be orientable;
    we can then apply \cite[Theorem 2]{L63} to deduce that
    $$X\setminus\bigcup_i B_i\cong[0,1]\times\mathbb{RP}^2.$$
    We deduce that in this case $N=2$ and,
    since $\bar B_i$ is homeomorphic to the (compact) cone
    over $\mathbb{RP}^2$, the claim follows.
    
    If instead $N=0$, then $X$ is a topological manifold with finite fundamental group.
    In this case, the claim follows from the solution to Thurston's elliptization conjecture.
\end{proof}

The above classification theorem was obtained independently for closed positively curved $3$-dimensional Alexandrov spaces in \cite{GG15, HS17}. We refer to \cite{GN22, RGGGH23} for a comprehensive survey of the topology of $3$-dimensional Alexandrov spaces. The classification theorem was also proved for $\RCD(K,3)$ spaces which are $3$-dimensional Alexandrov spaces (with an arbitrary lower sectional curvature bound) without boundary \cite{DGGM18}. The proofs are essentially the same once the previous results of the paper are established, which in the case of \cite{DGGM18} is made considerably simpler due to the presence of Alexandrov structure. Indeed, the orbifold structure and local finiteness of singularities from Theorem~\ref{Thm: orbifold structure thm intro} in that case are direct consequences of Perelman's conical neighborhood theorem \cite{P93}. Moreover, it can be checked that the ramified double cover is an Alexandrov space (see \cite[Proposition 2.4]{DGGM18}) following the methods of \cite{GW14}, which gives Lemma~\ref{Lem: main lemma} immediately. 

\section{Topological stability}

Let $\mathcal{P}$ be the set of locally non-orientable points as before. Note that by the results of the previous section $\mathcal{P}$ is equal to set of points whose tangents cones have cross sections homeomorphic to $\mathbb{RP}^2$. We start with the following lemma, which says that $\mathcal{P}$ is uniformly separated depending only on the non-collapsing constant.

\begin{Lem}\label{p.lb}
    Given $v>0$, there exists $\rho(v)>0$ such that the following holds:
    If $(X,\dd,\HH^3,p)$ is a non-collapsed $\rcd(-2,3)$ space without boundary with
    $\HH^3(B_1(p))\ge v$ and $p\in\mathcal{P}$, then
    $$\mathcal{P}\cap B_{\rho}(p)=\{p\}.$$
\end{Lem}

\begin{proof}
    Taking $\delta_0,C$ as in Proposition~\ref{green.balls} and $\delta_1,\Lambda$ as in Lemma~\ref{conicality.lift} below  (for the chosen $\delta_0$), let $\rho(v)>0$ such that,
    for some $\sigma\in(2C\rho,\delta_0^2/C)$, we have $s<\delta_1$ and
    the ball $B_{\Lambda s}(p)\subset X$ is $\delta_1$-conical for all $s\in[C^{-1}\delta_0\sigma,\sigma/\delta_0]$. In particular, both $B_s(p)$ and $B_s(\hat p)$ are $\delta_0$-conical for these $s$.
    By \cite[Theorem 4.7]{BPS24}, there exists a good Green-type distance
    $b:B_{\sigma}(p)\to[0,\infty)$ (see \cite[Remark 4.11]{BPS24}).
    As in the proof of Proposition~\ref{Pro: orbifold struture thm up to local finiteness},
    $\hat b:=b\circ\pi$ is also a good Green-type distance on $B_\sigma(\hat p)$.

    By construction of $\mathcal{G}_{\hat p}$ (which is built using \cite[Theorem 5.4]{BPS24}),
    we can now take
    $$r\in(\sigma/C,\sigma)\cap\mathcal{G}_{p}\cap\mathcal{G}_{\hat p},\quad\mathbb{B}_r(p):=\{b<r\},\quad\mathbb{B}_r(\hat p):=\{\hat b<r\}=\pi^{-1}(\mathbb{B}_r(\hat p)),$$
    so that in particular $r>2\rho$.
    We can also assume without loss of generality that $\mathbb{S}_r(\hat p)\cap\pi^{-1}(\mathcal{P})=\emptyset$.

    Letting $\tilde{\mathcal{P}}:=\pi^{-1}(\mathcal{P})\cap\mathbb{B}_r(p)$,
    since $\pi(\mathbb{B}_r(\hat p))\supseteq\pi(B_\rho(\hat p))=B_\rho(p)$ (by Proposition~\ref{green.balls} (4)), it suffices to show that $\tilde{\mathcal{P}}=\{\hat p\}$.
    Since $\mathcal{P}$ is locally finite, we have $N:=\#\tilde{\mathcal{P}}<\infty$.
    Let $(B_i)_{i=1}^k$ be a family of small Green balls around the points of $\tilde{\mathcal{P}}$,
    with disjoint closures included in $B_0:=\mathbb{B}_r(\hat p)$. As in the proof of
    Proposition~\ref{Pro: orbifold struture thm up to local finiteness},
    we can assume that $B_i$ is the lift of a Green ball $\pi(B_i)\subset X$ with $\partial\pi(B_i)\cong\mathbb{RP}^2$, for each $i=1,\dots,k$.

    Since $\hat X$ is an oriented topological manifold,
    $B_i$ is contractible for each $i=0,\dots,k$  by \cite[Proposition 9.24]{BPS24}; also, $\bar B_i$ is a manifold with boundary
    (by Proposition~\ref{green.balls} (8)) and $\partial B_i\cong S^2$.
    Hence, the set
    $$M:=\bar B_0\setminus\bigcup_{i=1}^k B_i$$
    is a compact manifold with boundary.
    Given any compact submanifold $S\subseteq\hat X$ (with or without boundary) with $\Gamma(S)=S$, let us denote by
    $$\chi_\Gamma(S):=\sum_{j=0}^\infty(-1)^j\operatorname{tr}(\Gamma_*:H_j(S;\Q)\to H_j(S;\Q))$$
    the Lefschetz number of $\Gamma|_S:S\to S$. We have
    $$\chi_\Gamma(B_0)=\chi_\Gamma(M)-\chi_\Gamma(M\cap N)
    +\chi_\Gamma(N),\quad N:=\bigcup_{i=1}^k\bar B_i.$$
    Since each ball is contractible, we have $\chi_\Gamma(B_i)=1$ for all $i=0,\dots,k$. Thus,
    $$\chi_\Gamma(M)=1+\chi_\Gamma(M\cap N)-k.$$
    Moreover, $\Gamma|_{\partial B_i}$ is orientation-reversing for each $i=1,\dots,k$,
    as $\pi(\partial B_i)\cong\mathbb{RP}^2$, giving $\chi_\Gamma(\partial B_i)=0$ for these $i$. Since $M\cap N=\bigcup_{i=1}^k\partial B_i$, we obtain $\chi_\Gamma(M\cap N)=0$, and hence
    $$\chi_\Gamma(M)=1-k.$$
    
    Since $\Gamma|_M$ has no fixed point and $M$ can be triangulated (as it is a 3-manifold with boundary),
    Lefschetz's theorem gives $\chi_\Gamma(M)=0$, so that $k=1$ and thus $\tilde{\mathcal{P}}=\{\hat p\}$.
\end{proof}

We used the following fact, whose proof is contained in the one of \cite[Lemma 9.17 (iii)]{BPS24}, and is thus omitted.

\begin{Lem}\label{conicality.lift}
    There exist $\delta_1(v, \delta_0)>0$ and $\Lambda(v)>0$ such that the following holds:
    Assume that $(X,\dd,\HH^3,p)$ is a non-orientable $\RCD(-2,3)$ space with $\HH^3(B_1(p))\ge v$
    and that $p\in\mathcal{P}$. If $B_{\Lambda r}(p)$ is $\delta_1$-conical for some $r\in(0,\delta_1)$,
    then $B_r(\hat p)\subset\hat X$ is $\delta_0$-conical.
\end{Lem}

\begin{Cor}\label{unif.cone.rp2}
    In the same situation of Lemma~\ref{p.lb}, there exists $C(v)$ such that,
    given $r\in(0,\rho/C)$, there exists $r'\in\mathcal{G}_p\cap(r,Cr)$
    such that the closed Green ball $\bar{\mathbb{B}}_{r'}(p)$
    is homeomorphic to the (compact) cone over $\mathbb{RP}^2$.
\end{Cor}

\begin{proof}
    Indeed, arguing as in the previous proof, we can find a Green ball $B_0=\mathbb{B}_{r'}(p)\subset \hat X$ whose projection is a
    Green ball $\pi(B_0)\subset B_\rho(p)$. Taking a decreasing sequence of small
    Green balls $B_1\supset B_2\supset\dots$ (in $\hat X$) whose projections are Green balls in $X$, we can conclude by arguing exactly as in the proof of Proposition~\ref{Pro: orbifold struture thm up to local finiteness}.
\end{proof}

\begin{Lem}\label{displ.lb}
    Given $v>0$, there exists $\delta(v)>0$ such that the following holds:
    If $(X,\dd,\HH^3,p)$ is a non-collapsed $\rcd(-2,3)$ space without boundary such that
    $\HH^3(B_1(p))\ge v$ and $\mathcal{P}\neq\emptyset$, then we have
    $$\hat\dd(\hat q,\Gamma\hat q)\ge\delta\min\{1,\hat\dd(\hat q,\pi^{-1}(\mathcal P))\}\quad\text{for all }\hat q\in B_1(\hat p)$$
    on the ramified double cover $\hat X$, where $\pi(\hat p)=p$.
\end{Lem}

\begin{proof}
    Let $v'(v)>0$ be such that $\HH(B_1(p'))\ge v'$ for all $p'\in B_2(p)$,
    and let $\rho(v')>0$ be given by Lemma~\ref{p.lb}.
    If $\mathcal{P}\cap B_{\rho(v')/2}(q)=\emptyset$, the claim
    follows immediately from a simple compactness argument,
    together with the fact that the latter condition is stable under pmGH limits (see the first part of
    the proof of Theorem~\ref{main.homeo} below).

    If instead $\mathcal{P}\cap B_{\rho(v')/2}(q)\neq\emptyset$,
    then (up to replacing $v>0$ with $v'>0$) we can assume that $p\in\mathcal{P}$ is a closest point to $q$ in $\mathcal{P}$.
    It suffices to show that, for some $\delta>0$, we have
    $$\hat\dd(\hat q,\Gamma\hat q)\ge\delta\hat\dd(\hat p,\hat q)\quad\text{for all }\hat q\in \bar B_{\rho/2}(\hat p).$$
    Up to rescaling, it is enough to do it when $\hat\dd(\hat p,\hat q)=\rho/2$.

    Assume by contradiction that $(X_i,\dd_i,\HH^3,p_i)$,
    together with $\hat q_i\in\hat X_i$, provide a counterexample for $\delta=2^{-i}\to0$,
    and let $(X,\dd,\HH^3,p)$
    be a limit in the pmGH sense, up to a subsequence. We also let
    $$\hat q:=\lim_{i\to\infty}\hat q_i=\lim_{i\to\infty}\Gamma_i\hat q_i$$
    be the limit point in $\hat X$, so that $q:=\pi(\hat q)\in\mathcal{P}$ and $p\in\mathcal{P}$,
    as well as
    $$\dd(p,q)=\hat\dd(\hat p,\hat q)=\rho/2.$$
    Thus, $\mathcal{P}\cap B_\rho(p)$ contains at least two points,
    contradicting the previous lemma.
\end{proof}

\begin{Cor}\label{cor.contr}
    In the situation of Lemma~\ref{displ.lb},
    given $q\in B_{1}(p)$
    and $r\in(0,\delta\min\{1,\dd(q,\mathcal{P})\})$, there exists $r'\in(2r,2Cr)\cap\mathcal{G}_p$ and, for any such $r'$, the Green ball $\bar{\mathbb{B}}_{r'}(q)$ is contractible
    (for a possibly smaller $\delta(v)>0$ and larger $C(v)>0$).
    In particular, $B_r(q)$ is contractible in $B_{4Cr}(q)$.
\end{Cor}

\begin{proof}
    Indeed, for $\delta(v)>0$ given by the previous lemma,
    letting $\kappa:=\delta\min\{1,\dd(q,\mathcal{P})\}$ we have $\hat\dd(\hat q,\Gamma \hat q)\ge\kappa$.
    Hence, $\pi$ restricts to an isometry from $B_{\kappa/4}(\hat q)$
    to $B_{\kappa/4}(q)$.
    Now \cite[Proposition 9.24]{BPS24} shows that any Green ball $\mathbb{B}_{r'}(p)$
    with $\bar{\mathbb{B}}_{r'}(q)\subset\mathbb{B}_{\kappa/4}(q)$ is contractible.
    Replacing $\delta$ with a smaller $\delta'$ such that $2C\delta'<\delta/4$
    and recalling Proposition~\ref{green.balls} (4), we obtain the claim.
\end{proof}

Combining Corollary~\ref{unif.cone.rp2} and Corollary~\ref{cor.contr},
we obtain the following.

\begin{Cor}\label{cor.contr.bis}
    If $(X,\dd,\HH^3,p)$ is a non-collapsed $\rcd(-2,3)$ with $\HH^3(B_1(p))\ge v$,
    for any $r\in(0,\delta)$ 
    the ball $B_r(q)$ is contractible in $B_{4Cr}(q)$ (for a possibly smaller $\delta(v)>0$ and larger $C(v)>0$).
\end{Cor}

\begin{proof}
    Indeed, taking $\delta$ as in Corollary~\ref{cor.contr} and assuming $r<\delta'\le\delta$,
    if $r<\delta\dd(p,\mathcal{P})$ then the claim holds by the same result.
    Otherwise, we have $\dd(p,p')\le r/\delta$ for some $p'\in\mathcal{P}$,
    which then has $\HH^3(B_1(p'))\ge v'(v)>0$.
    Once we take $\delta'$ so small that $8(C+\delta^{-1})r<\rho(v')$, we can find radii
    $r'\in(2r,2Cr)\cap\mathcal{G}_p$ and $s>\rho(v')$ (from Corollary~\ref{unif.cone.rp2})
    such that $\mathbb{B}_{s}(p')$ is contractible and moreover
    $$\mathbb{B}_{r'}(p)\subseteq B_{4Cr}(p)\subseteq B_{\rho(v')/2}(p')\subseteq\mathbb{B}_s(p'),$$
    giving the claim.
\end{proof}

The topological stability (Theorem~\ref{Thm: stability thm intro}) is a direct consequence of the following theorem.

\begin{Thm}\label{main.homeo}
    Assume that $(X_i,\dd_i,\HH^3)$ is a sequence of
    non-collapsed $\rcd(-2,3)$ spaces without boundary converging
    to $(X,\dd,\HH^3)$ in the mGH sense. Also, assume that each has diameter $\le D<\infty$,
    so that for some $v>0$ we have $\HH^3(B_1(p))\ge v>0$ for all $i$ and $p\in X_i$.
    Then eventually there is a homeomorphism $h_i:X_i\to X$;
    moreover, given a sequence of $\veps_i$-GH isometries $f_i:X_i\to X$ with $\veps_i\to0$ and $\veps_i$-inverses $g_i:X\to X_i$, we can take $h_i$ such that
    $$\sup_{x\in X_i}\dd(f_i(x),h_i(x))\to0,\quad\sup_{x\in X}\dd_i(g_i(x),h_i^{-1}(x))\to0.$$
\end{Thm}

Note that the very last part of the statement easily follows from $\sup_{x\in X_i}\dd(f_i(x),h_i(x))\to0$.
We will use two important tools from geometric topology.
For the first one, we refer to \cite{P90};
although not explicitly stated in this form,
it follows from the proof of the main result of \cite{P90}.

\begin{Pro}\label{pet.pro}
    Given $n\in\N$, a function $\gamma:(0,\tau)\to(0,\infty)$ with
    $$\gamma(r)\ge r\text{ for all }r,\quad\lim_{r\to0}\gamma(r)=0,$$
    and $\veps>0$, there exist $\sigma\in(0,\tau)$ and $\eta>0$ such that the following holds: If $f:X\to Y$ is an $\eta$-GH isometry between two metric spaces of (Lebesgue) covering dimension $\le n$ and
    any ball $B_r(p)$ in either space is $n$-connected in $B_{\gamma(r)}(p)$
    for all $r\in(\sigma,\tau)$, then $f$ is $\veps$-close to a homotopy equivalence $\tilde f$.
\end{Pro}

We recall that $B_r(p)$ is \emph{$n$-connected} in $B_{\gamma(r)}(p)$ if
any continuous map $\varphi:\mathbb{S}^\ell\to B_r(p)$ can be extended to
a continuous map $\bar B^{\ell+1}\to B_{\gamma(r)}(p)$ defined on the closed Euclidean ball, for any $\ell=0,\dots,n$.

\begin{Rem}\label{pet.equiv}
    It follows that $\tilde f$ is a $C\veps$-GH isometry.
    In fact, the proof gives a homotopy inverse $\tilde g$ of $\tilde f$ which is also a $C\veps$-GH isometry, and $\tilde f$ is a $C\veps$-equivalence: this means that there exist homotopies $H$ and $H'$, connecting $\tilde g\circ\tilde f$ to $\operatorname{id}_X$ and $\tilde f\circ\tilde g$ to $\operatorname{id}_Y$ respectively,
    such that $\{\tilde f\circ H(t,x) \,:\,  t\in[0,1]\}$
    and $\{H'(t,y) \,:\,  t\in[0,1]\}$ have diameter at most $C\veps$
    (for every $x\in X$ and $y\in Y$).
\end{Rem}

The second tool is the following result, first obtained by Chapman--Ferry in dimension $n\ge5$ \cite{CF79}; its proof in dimension $n=3$ is due to Jakobsche \cite{J88} (note that the technical assumption in the main statements of \cite{J88}, namely the absence of \emph{fake $3$-cells}, is guaranteed by the positive resolution of the Poincaré conjecture).

\begin{Thm}\label{hom.thm}
    Given two metric spaces $X,Y$ which are closed topological $3$-manifolds
    and given $\veps>0$, there exists $\eta(Y,\veps)>0$
    such that any $\eta$-equivalence $f:X\to Y$ is $\veps$-close to a homeomorphism $h:X\to Y$.
\end{Thm}

\begin{proof}[Proof of Theorem~\ref{main.homeo}]
    First of all, we have $\mathcal{P}_i\to\mathcal{P}$
    in the Hausdorff sense: from Theorem~\ref{Thm: P stability}, it is clear that
    a (subsequential) limit of points in $\mathcal{P}_i$ belongs to $\mathcal{P}$.
    Moreover, if $p\in\mathcal{P}$ then there must exist $p_i\in\mathcal{P}_i$ converging to $p$: if not, then there exists
    $r'>0$ such that $B_{r'}(p_i)\subset X_i$ is a topological manifold (up to a subsequence),
    where $p_i\in X_i$ are arbitrary points converging to $p\in X$.
    Since $p\in\mathcal{P}$, there exists a good radius $r\in\mathcal{G}_p\cap(0,r')$
    such that $\partial\mathbb{B}_r(p)\cong\mathbb{RP}^2$.
    Then, exactly as in the proof of \cite[Theorem 8.1]{BPS24}, eventually we have $\partial\mathbb{B}_{r_i}(p)\cong\mathbb{RP}^2$
    for a suitable sequence $r_i\to r$ of good radii $r_i\in\mathcal{G}_{p_i}$, which is impossible since $\partial\mathbb{B}_{r_i}(p)$
    bounds the topological manifold $\bar{\mathbb{B}}_{r_i}(p)$.

    From Lemma~\ref{p.lb}, it follows that $\mathcal{P}_i$ and $\mathcal{P}$ have also the same cardinality eventually.
    We can then write
    $$\mathcal{P}_i=\{p_{i,k} \,:\,  k=1,\dots,N\},\quad
    \mathcal{P}=\{p_{k} \,:\,  k=1,\dots,N\},$$
    with $p_{i,k}\to p_k$. We now choose good radii $r_k\in(\rho/(4C),\rho/4)$
    given by Corollary~\ref{unif.cone.rp2} and similarly we choose $r_{i,k}\to r_k$
    along the sequence. We then have the homeomorphism
    $$\bar{\mathbb{B}}_{r_{i,k}}(p_{i,k})\cong\bar{\mathbb{B}}_{r_k}(p_{k}).$$

    To conclude, we claim that $M_i\cong M$, where
    $$M:=X\setminus\bigcup_{k=1}^N\mathbb{B}_{r_k}(p_{k}),\quad M_i:=X_i\setminus\bigcup_{k=1}^N\mathbb{B}_{r_{i,k}}(p_{i,k});$$
    note that the balls in the union have disjoint closures.
    To do this,
    let $\tilde M$ be the doubling of $M$, i.e., the 3-manifold given by the union of two copies
    of $M$ glued along the boundary, and similarly let $\tilde M_i$ be the doubling of $M_i$.
    The doubling $\tilde M$ is a metric space: viewing $M\subset \tilde M$, we extend the metric $\dd$ by letting $\dd(w',z'):=\dd(w,z)$ whenever $w',z'\in\tilde M\setminus M$ correspond to $w,z\in M$ in the other half, and
    $$\dd(w,z'):=\min_{v\in\partial M}[\dd(w,v)+\dd(v,z)];$$
    it is straightforward to check that the extension is still a geodesic metric (and similarly for $\tilde M_i$).
    We claim that $\tilde M$ and $\tilde M_i$ satisfy the assumptions of Proposition~\ref{pet.pro}.

    Once this is done, calling $\tilde f_i:\tilde M_i\to \tilde M$ the almost GH-isometry induced by $f_i$ and applying Proposition~\ref{pet.pro}
    in conjunction with Remark~\ref{pet.equiv} and Theorem~\ref{hom.thm}, we deduce that $\tilde f_i$ (together with its inverse) becomes arbitrarily close to a homeomorphism $\tilde h_i:\tilde M_i\to \tilde M$ as $i\to\infty$.

    With $\tilde h_i$ in hand, calling $S_k:=\mathbb{S}_{r_k}(p_{k})$ and
    viewing $M_i\subset \tilde M_i$ and $M\subset \tilde M$, we can apply Proposition~\ref{green.balls} (8) and find a neighborhood
    $S_k\subset V_k\subset \tilde M$
    with $\bar V_k\cong [-1,1]\times\mathbb{RP}^2$.
    In fact, we claim that
    \begin{equation}\label{isotop}\bar V_k\cap\tilde h_i(M_i)\cong[0,1]\times\mathbb{RP}^2.\end{equation}
    Indeed, for $i$ large enough, $\tilde h_i(\partial M_i)\cap V_k$ is a (tamely embedded) projective plane separating the two boundary components of $\bar V_k$
    and hence, in the oriented cover $W_k$ of $\bar V_k$ (namely $[-1,1]\times S^2$, which we view as $S^3$ with two spherical caps removed), it lifts to a sphere $\Sigma_k$ separating the boundary components (i.e., the two caps in $S^3$).
    
    By applying Alexander's embedded sphere theorem (see, e.g., \cite[Theorem 1.1]{H00}), we deduce
    that $\Sigma_k\subset S^3$ bounds a topological ball. Endowing $S^3$ with a smooth structure such that $\Sigma_k$ is smoothly embedded, so that it bounds a diffeomorphic copy of $\bar B^3$, we deduce that $\Sigma_k$ is isotopic to the equator of $S^3$ (as this ball can be shrunk to an almost round ball by isotopies). Since $\Sigma_k$ separates the two spherical caps in $S^3$, we deduce that $\Sigma_k$ bounds two copies of $[0,1]\times S^2$ in $W_k$ (for the same reason); by applying \cite[Theorem 1]{L63}, we obtain \eqref{isotop}. Thus,
    $\tilde h_i(\partial M_i)\cap V_k$ is isotopic to one of the boundary components of $\bar V_k$, and hence to $S_k$, within a slightly larger neighborhood $V_k'\supseteq V_k$,
    completing the proof that $M_i\cong M$ (through a slight perturbation
    of $\tilde h_i$).

    We are left to show that $\tilde M_i$ and $\tilde M$ satisfy the assumptions of Proposition~\ref{pet.pro}, for the same function $\gamma$.
    We check this just for $\tilde M$, since the argument gives a $\gamma$ depending only on $v$.
    Let $\delta(v),C(v)$ be as in Corollary~\ref{cor.contr} and fix $\tau(v)>0$ small such that
    $\tau<\delta\dd(M,\mathcal{P})$, as well as $4C\tau<\delta$.
    Taking an arbitrary $B_s(q)\subset \tilde M$ with $q\in M$ and
    $s\in(0,\tau)$, if $4Cs\le\dd(q,\partial M)$
    then Corollary~\ref{cor.contr} implies that $B_s(q)$ is contractible
     in $B_{4Cs}(q)\subset M$. Assume then that $4Cs>\dd(q,\partial M)$.

     Assume that we have a continuous map $\phi:S^\ell\to B_s(q)\subset\tilde M$ with $\ell\le3$. We wish to extend $\phi$ to the closed Euclidean ball $\bar B^{\ell+1}$,
     taking values in a controlled enlargement $B_{C's}(q)\subset\tilde M$. The claim is obvious if $\ell=0$,
     since $\tilde X$ is a geodesic space. Also, once we prove it for $\ell\le2$,
     it follows for $\ell=3$ as well, since $B_s(q)$ is an open $3$-manifold and thus its third homology group vanishes, allowing us to conclude the proof of the claim via a local version of
     Hurewicz theorem (see \cite[Theorem 0.8.3]{DV09}).

    Let us assume then $\ell\in\{1,2\}$.
    If $\phi$ takes values in $M$, then (viewing $M\subset X$)
    we can apply Corollary~\ref{cor.contr} and extend it to a map
    $\bar\phi:\bar B^{\ell+1}\to B_{4Cs}(q)$.
    By Proposition~\ref{green.balls} (7) and the lower bound $\rho(v)$ on pairwise distances of points in $\mathcal{P}$ given by Lemma~\ref{p.lb}, there exists a neighborhood $\partial M\subset U\subset X$
    of size $\lambda(v)r$ (i.e., $U=B_{\lambda r}(\partial M)$), where $r:=\min\{r_1,\dots,r_N\}$, and a retraction $R:U\to\partial M$ such that
    $$\dd(x,R(x))\le C\dd(x,\partial M)\quad\text{for all }x\in U.$$
    As a consequence, as long as $8C\tau<\lambda r$, we have $B_{4Cs}(q)\subset U$
    and thus $R\circ\bar\phi$ is the desired extension with values in $M\subset N$.

    The proof is similar if $\phi$ takes values in the other half $M'$ of the doubling. To conclude, given $\phi$ taking values in both halves,
    we will homotope it to a constant by reducing to the previous situation.
    We only show how this can be done when $\ell=2$, since the case $\ell=1$ is similar (and easier).
     
     By endowing $\tilde M$ with a smooth structure with respect to which
     $\partial M$ is smoothly embedded, up to a slight perturbation
     we can assume that $\phi$ is transverse to $\partial M$.
     Thus, $\phi^{-1}(\partial M)$ is a union of embedded circles in $S^2$.
     Among these, we can always find one, which we call $\Gamma$, bounding a (compact) topological disk $\Delta$. Now we exploit Proposition~\ref{green.balls} (6):
     namely, since $\partial M=\bigcup_k\partial\mathbb{B}_{r_k}(p_k)$, we can fill $\phi|_\Gamma$ with a continuous map $\psi:\Delta\to B_{C_0s}(q)\cap\partial M$.
     
     Since $\phi|_\Delta$ and $\psi$ are taking values in the same half of the doubling,
     the previous case (for $\ell=2$) shows that $\phi|_\Delta$ is homotopic to $\psi$
     in $B_{CC_0s}(q)$ (up to decreasing $\tau$ accordingly).
     Thus, up to replacing $\phi|_\Delta$ with $\psi$
     and perturbing the resulting map slightly (detaching the image of $\Delta$ from $\partial M$), we can decrease the number of connected components in $\phi^{-1}(\partial M)$. By iteration, we eventually reduce to the situation where $\phi$ takes values in just one half;
     note carefully that at each step the modified map $\tilde\phi$
     takes values in $B_{C_0s}(q)$ for the \emph{same} constant $C_0$, since (up to a slight perturbation) each successive modification differs from the initial map $\phi$
     just by a filling of one of the initial loops $\phi|_\Gamma$.

     The previous argument showed the existence of $\tau(v)>0$ and $C_1(v)$
     such that each ball $B_s(q)\subset\tilde M$ with $s\in(0,\tau)$
     is contractible in $B_{C_1s}(q)\subset\tilde M$, as desired.
\end{proof}

\end{document}